\pdfoutput=1
\documentclass{article}
\usepackage{amsthm}

\usepackage[utf8]{inputenc}
\usepackage[T1]{fontenc}
\usepackage{geometry}
\usepackage{amsmath,amssymb}
\usepackage{tikz}
\usepackage{subcaption}
\usepackage{mathtools}
\usepackage{booktabs}
\usepackage{microtype}
\usepackage{bm}
\usepackage{etoolbox}
\usepackage{enumitem}
\usepackage{mathrsfs} 
\usepackage{blkarray,bigstrut} 
\usepackage{etoc} 
\usepackage{relsize}  
\usepackage{exscale}  
\usepackage[ugly]{nicefrac} 
\usepackage{xfrac} 
\usepackage{old-arrows} 
\usepackage{pifont}
\usepackage[sort & compress,square,comma,numbers]{natbib}
\usepackage[title]{appendix} 
\usepackage{authblk}
\usepackage[color=white,disable]{todonotes}
\usepackage{tcolorbox}
\usepackage[pagewise]{lineno}
\usepackage[colorlinks,citecolor=blue]{hyperref}
\usepackage{doi} 

\usetikzlibrary{shapes,positioning,fit,matrix,arrows,arrows.meta}
\usetikzlibrary{decorations.pathmorphing,decorations.pathreplacing,decorations.markings}
\usetikzlibrary{calc}

\DeclareMathSymbol{\shortminus}{\mathbin}{AMSa}{"39}

\setuptodonotes{inline}

\tikzset{
dot/.style={draw,fill,circle,inner sep = 0pt,minimum size = 3pt},
bigdot/.style={dot,minimum size = 4pt},
evenbiggerdot/.style={draw, thick, circle, inner sep = 1pt, minimum size = 12pt},
terminal/.style={draw,circle, inner sep=2.5pt},
vcolour/.style={draw,inner sep=1.5pt,font=\scriptsize,label distance=2pt},
22box/.style={draw,minimum width=2cm,minimum height=1.5cm,font=\LARGE,node contents=\( \Downarrow \)},
22boxsmall/.style={22box,minimum width=1.5cm, minimum height=1cm},
Vset/.style={draw=black!30,ellipse,minimum width=1.5cm,minimum height=2.5cm},
}

\newtheorem{theorem}{Theorem}
\newtheorem{lemma}{Lemma}
\newtheorem{corollary}{Corollary}
\newtheorem{observation}{Observation}
\newtheorem{problem}{Problem}

\newtheoremstyle{crampedthm}
{3pt}
{3pt}
{}
{}
{\bfseries}
{.\\}
{.5em}
{\thmname{#1}\thmnumber{ #2}\thmnote{ (#3)}}
\theoremstyle{crampedthm}

\newtheoremstyle{freethm}
{3pt}
{3pt}
{}
{}
{\bfseries}
{.\\}
{.5em}
{\thmname{#1}\thmnumber{ #2}\thmnote{ (#3)}}

\theoremstyle{freethm}

\theoremstyle{plain}

\newtheoremstyle{thmpart}
{3pt}
{3pt}
{}
{}
{\bfseries}
{:}
{.5em}
{\thmname{#1}\thmnumber{ #2}\thmnote{ (#3)}}

\theoremstyle{thmpart}
\newtheorem{myclaim}{Claim}%
\theoremstyle{plain}

\newtheorem{innercustomthm}{Theorem}
\newenvironment{customthm}[1]
  {\renewcommand\theinnercustomthm{#1 (Restated)}\innercustomthm}
  {\endinnercustomthm}

\newtheorem{innercustomcor}{Corollary}
\newenvironment{customcor}[1]
  {\renewcommand\theinnercustomcor{#1 (Restated)}\innercustomcor}
  {\endinnercustomthm}

\newtheorem{innercustomlem}{Lemma}
\newenvironment{customlem}[1]
  {\renewcommand\theinnercustomlem{#1 (Restated)}\innercustomlem}
  {\endinnercustomlem}

\theoremstyle{freethm}

\newcommand{\colcol}[1]{\textcolor{blue}{#1}}

\newcommand{\inline}[1]{ }
\newcommand{\inlinenew}[1]{ }

\providetoggle{forThesis}
\providetoggle{extended}
\providetoggle{dropFaces}
\providetoggle{iiitdmV}

\togglefalse{forThesis}  

\providecommand{\keywords}[1]{\textbf{\textit{Keywords: }} #1}

\begin{document}
\title{Star colouring and locally constrained graph homomorphisms}
%
%
\author[1]{Cyriac Antony}
\affil[1]{Luiss University, Rome, Italy; Email: \href{mailto:acyriac@luiss.it}{\texttt{acyriac@luiss.it}}}
\author[2]{Shalu M.\ A.}
\affil[2]{Indian Institute of Information Technology, Design \& Manufacturing\protect\\ (IIITDM) Kancheepuram, Chennai, India; Email: \href{mailto:shalu@iiitdm.ac.in}{\texttt{shalu@iiitdm.ac.in}}}

\date{}
%
%
%
\maketitle              
\begin{abstract}
We relate star colouring of even-degree regular graphs to the notions of locally constrained graph homomorphisms to the oriented line graph \( \vec{L}(K_q) \) of the complete graph \( K_q \) and to its underlying undirected graph \( L^*(K_q) \). 
Our results have consequences for locally constrained graph homomorphisms and oriented line graphs in addition to star colouring. 
We show that \( L^*(H) \) is a 2-lift of the line graph \( L(H) \) for every graph \( H \). 
Dvo{\v{r}}{\'a}k, Mohar and \v{S}\'amal~(J.\ Graph Theory, 2013) proved that for every 3-regular graph \( G \), the line graph of \( G \) is 4-star colourable if and only if \( G \) admits a locally bijective homomorphism to the cube~\( Q_3 \). 
We generalise this result as follows: for \( p\geq 2 \), a \( K_{1,p+1} \)-free \( 2p \)-regular graph \( G \) admits a \( (p+2) \)-star colouring if and only if \( G \) admits a locally bijective homomorphism to \( L^*(K_{p+2}) \). 
As a result, if a \( K_{p+1} \)-free \( 2p \)-regular graph \( G \) with \( p\geq 2 \) is \( (p+2) \)-star colourable, then \( -2 \) and \( p-2 \) are eigenvalues of \( G \). 
We also prove the following:  
(i)~for \( p\geq 2 \), a \( 2p \)-regular graph \( G \) admits a \( (p+2) \)-star colouring if and only if \( G \) has an orientation that admits an out-neighbourhood bijective homomorphism to \( \vec{L}(K_{p+2}) \);  
(ii)~the line graph of a 3-regular graph \( G \) is 4-star colourable if and only if \( G \) is bipartite and distance-two 4-colourable; and (iii)~it is NP-complete to check whether a planar 4-regular 3-connected graph is 4-star colourable. 
\end{abstract}
%
%
%


\todo{
\inlinenew{\Large \textbf{Reviews}}\\[7pt]
\textbf{Reviewer F:}

This paper mostly concerns characterisations of \( 2p \)-regular graphs that have a start colouring that uses \( p+2 \) colors (which is the smallest possible numer). The first theorem says that a \( 2p \)-regular and \( K_{1, p+1} \)-free graph \( G \) has a \( (p+2) \)-star coloring if and only if there is a locally bijective homomorphism from \( G \) to some specific graph \( G_{2p} \). It is a generalization of an earlier result, that a 4-regular line graph \( L(H) \) has 4-star coloring if \( H \) has a locally bijective homomorhism to the 3-dimensional cube. The second theorem gives a characterization of \( 2p \)-regular graphs that have a \( (p+2) \)-star coloring: the existance of a star coloring of \( G \) is equivalent to existance of a out-neighborhood bijective homomorphism from an orientation of \( G \) to an orientation of the same specific graph \( G_{2p} \). This result is a follow-up to an earlier work by the same authors (reference [29]). The proof of this theorem is not self-contained (it references part of the proof of a theorem from [29]). The paper also contains a number of other observations and corollaries of lesser importance.

The results are, in my opinion, very narrow. However, they may be interesting to some people and have some potential to inspire future research. The proofs are elementary, relying on some case-checking and routine arguments; however, they would not be straightforward to find.

The writing is appropriate. Personally I found it hard to follow the text, but I would attribute it to a large number of technical details and definitions that the reader must keep in mind.

My overall impression is that the paper is quite solid and definitely worth publishing, but I am not sure if it meets high standards of The Electronic Journal of Combinatorics. Therefore, I would recommend to accept the paper, but with a fairly low degree of confidence (hence my ``See comments'' recommendation).

I conclude with three minor remarks (moved to todo notes).

~\\
\textbf{Reviewer G:}

The paper presents several results that relate Star colorings to locally bijective homomorphisms (and to the related oriented variant).

These results are new to me and shall be published. I would strongly recommend to revisit some of the specific statements related to graph covers (see suggestions in the attached file [moved to todo notes]). It seems to me that several proofs could be simplified in these more general settiong. Indeed I convinced myself that several statements are correct by my own argument (could be short, see e.g.\ Thm 9) rather than reading the proof.

The paper shall be throughly revised taking into account the comments in the attached file (moved to todo notes), with emphasis on:

- change the flow of the presentation, so all proofs follow the statements

- all definitions precede the first use of the defined concepts and denoted symbols

- the notation is chosen in a coherent and natural way

- use simpler arguments whenever possible, and consider consequences of your finding (e.g. ``open'' Problem 2).

Then it shall be considered again whether it is good enough for The Electronic Journal of Combinatorics.

Recommendation: Major revisions required, resubmit for further review
}

\section{Introduction}
Star colouring is an extensively studied colouring variant~\cite{fertin2004,albertson,nesetril_mendez2003,borodin2013,timmons2009,kawarabayashi_mohar,bu_cranston,choi_park,gebremedhin2007,sanders,lyons,yue,omoomi}, and there is an exclusive survey~\cite{lei_shi} on star colouring of line graphs. 
A star colouring is a proper vertex colouring without any bicoloured 4-vertex path. 
Our focus in this paper is on star colouring of regular graphs and related notions, such as graph orientations and homomorphisms. 
Albertson et al.~\cite{albertson} and independently Ne\v{s}et\v{r}il and Mendez~\cite{nesetril_mendez2003} found that star colourings of a graph \( G \) are associated with certain orientations of \( G \), called in-orientations. 
The number of colours required to star colour a \( d \)-regular graph is at least \( \lceil (d+3)/2 \rceil \)~\cite{fertin2003}, and at least \( \lceil (d+4)/2 \rceil \) for \( d\geq 3 \)~\cite{shalu_cyriac3}. 
Even-degree regular graphs attaining this bound (i.e., \( 2p \)-regular \( (p+2) \)-star colourable graphs) are characterised in terms of a special case of in-orientations in \cite{shalu_cyriac3}. 
We show that the structure of \( 2p \)-regular \( (p+2) \)-star colourable graphs is closely related to locally constrained graph homomorphisms and a graph operation called oriented line graph operation (which is in turn related to the notions of line graph of a graph and line digraph of a digraph).

The locally constrained graph homomorphisms which are central to this paper are the well-known Locally Bijective Homomorphism (LBH), and an oriented version we introduce called Out-neighbourhood Bijective Homomorphism (OBH). 
An LBH from a graph \( G \) to a graph \( H \) is a mapping \( \psi\colon V(G)\to V(H) \) such that for every vertex \( v \) of \( G \), the restriction of \( \psi \) to the neighbourhood \( N_G(v) \) is a bijection from \( N_G(v) \) onto \( N_H(\psi(v)) \)~\cite{fiala_kratochvil}. 
An OBH from an oriented graph \( \vec{G} \) to an oriented graph \( \vec{H} \) is a mapping \( \psi\colon V(\vec{G})\to V(\vec{H}) \) such that for every vertex \( v \) of \( \vec{G} \), the restriction of \( \psi \) to the out-neighbourhood \( N_{\vec{G}}^+(v) \) is a bijection from \( N_{\vec{G}}^+(v) \) to \( N_{\vec{H}}^+(\psi(v)) \).

The oriented line graph of (an undirected) graph \( G \) is the oriented graph with vertex set \( \bigcup_{uv\in E(G)}\{(u,v), (v,u)\} \) and there is an arc from a vertex \( (u,v) \) to a vertex \( (v,w) \) in it when \( u\neq w \)~\cite{parzanchevski}. 
We denote the oriented line graph of a graph \( G \) by \( \vec{L}(G) \), and its underlying undirected graph by \( L^*(G) \). 
We show that \( L^*(G) \) is always a 2-lift of the line graph \( L(G) \), which means that \( L^*(G) \) has double the number of vertices of \( L(G) \) and it admits an LBH to \( L(G) \). 

Dvo{\v{r}}{\'a}k, Mohar and \v{S}\'amal~\cite{dvorak} proved that for every 3-regular graph \( H \), the line graph of \( H \) is 4\nobreakdash-star colourable if and only if \( H \) admits an LBH to the cube~\( Q_3 \). 
Thanks to the properties of locally bijective homomorphisms (see Theorem~\ref{thm:lbh goes with line graph operation}), this result is equivalent to the following: for every 3-regular graph \( H \), the line graph \( L(H) \) is 4-star colourable if and only if \( L(H) \) admits an LBH to \( L(Q_3) \). 
Clearly, the following statement is stronger: a \( K_{1,3} \)-free 4-regular graph \( G \) is 4-star colourable if and only if \( G \) admits an LBH to \( L(Q_3) \). 
Since \( L(Q_3)\cong L^*(K_4) \), it follows that a \( K_{1,3} \)-free 4-regular graph \( G \) is 4-star colourable if and only if \( G \) admits an LBH to \( L^*(K_4) \). 
We prove the following generalisation of the statement. 
\begin{itemize}
\item 
A \( K_{1,p+1} \)-free \( 2p \)-regular graph \( G \) with \( p\geq 2 \) admits a \( (p+2) \)-star colouring if and only if \( G \) admits an LBH to \( L^*(K_{p+2}) \) (see Theorem~\ref{thm:p+2 star colourable iff LBH}).  
\end{itemize}
Other main contributions of the paper are as follows. 
\begin{itemize}
\item
A \( 2p \)-regular graph \( G \) with \( p\geq 2 \) admits a \( (p+2) \)-star colouring if and only if \( G \) admits an OBH to \( \vec{L}(K_{p+2}) \) (see Theorem~\ref{thm:characterise 2p-regular p+2-star colourable}). 
\item 
For every 3-regular graph \( G \), the line graph of \( G \) is 4-star colourable if and only if \( G \) is bipartite and distance-two 4-colourable (see Theorem~\ref{thm:4-star colourable in terms of distance-two colouring}). 
\item   
It is NP-complete to check whether a planar 4-regular graph is 4-star colourable  (see Corollary~\ref{cor:4-star colouring 4-regular NPC}). 
\item
\( K_{1,p+1} \)-free \( 2p \)-regular \( (p+2) \)-star colourable graphs \( G\) have the following properties for \( p\geq 2 \):
  \begin{itemize}
  \item Eigenvalues of adjacency matrix of \( G \) include \( -2 \) and \( p-2 \) (see Theorem~\ref{thm:star col eigenvalue}).
  \item \( G \) has an intersection model, namely a clique graph (see Theorem~\ref{thm:characterise 2p-regular p+2-star colourable}). 
  \item if \( p=2 \), then \( G \) is a line graph (see Theorem~\ref{thm:4reg 4-starcol clawfree imply line}). 
  \end{itemize}
\end{itemize}

\todo{
(omitted) [re-written]

\inlinenew{The theorem statements are moved so that proofs directly follow them.}

\textbf{Theorem.}
\emph{
Let \( G \) be a \( K_{1,p+1} \)-free \( 2p \)-regular graph, where \( p\geq 2 \). 
Then, \( G \) admits a \( (p+2) \)-star colouring if and only if \( G \) admits a locally bijective homomorphism to \( G_{2p} \). 
}
Our results on \( 2p \)-regular \( (p+2) \)-star colourable graphs also include the following (see Section~\ref{sec:def} for definitions). 

\textbf{Theorem.}
\emph{
For \( p\geq 2 \), the following are equivalent for every \( 2p \)-regular graph \( G \). 
~\\
\( I \). \( G \) admits a \( (p+2) \)-star colouring. 
~\\
\( II \). \( G \) admits a \( (p+2) \)-colouring such that every bicoloured component is \( K_{1,p} \). 
~\\
\( III \). \( G \) admits a \( (p+2) \)-colourful Eulerian orientation. 
~\\
\( IV \). \( G \) has an orientation \( \vec{G} \) that admits an out-neighbourhood\\ bijective homomorphism to \( \overrightarrow{G_{2p}} \). 
}


\textbf{Theorem.}
\emph{
Let \( G \) be a \( K_{1,p+1} \)-free \( 2p \)-regular graph, where \( p\geq 2 \). 
If \( G \) admits a \( (p+2) \)-star colouring, then \( G \) is a locally linear graph and a clique graph, and \( K\big(K(G)\big)\cong G \). 
}

\textbf{Theorem.}
\emph{
For every 3-regular graph \( G \), the line graph of \( G \) is 4-star colourable if and only if \( G \) is bipartite and distance-two 4-colourable. 
}

This is an extension of the work in \cite{shalu_cyriac3}. 
}
This is an extension of the work in \cite{shalu_cyriac3}, but can be read independently. 
Parts of this work appeared in \cite{cyriac} and \cite{shalu_cyriac6}. 
This paper is organised as follows. 
Section~\ref{sec:def} provides the definitions. 
Section~\ref{sec:prelims} presents some preliminaries and results on the tools we employ. 
The main results and proofs appear in Section~\ref{sec:2p-regular p+2-star colourable}. 
We conclude with Section~\ref{sec:open problems} devoted to open problems.

\section{Definitions}\label{sec:def}
We denote the set of positive integers by \( \mathbb{N} \). 
All graphs considered in this paper are finite and simple, and undirected unless otherwise specificed.
We follow West~\cite{west} for graph theory terminology and notation. 
A \( q \)-clique in a graph \( G \) is a set of \( q \) pairwise adjacent vertices of~\( G \). 
We assume that \( \mathbb{Z}_q \) is the vertex set of the complete graph \( K_q \). 
The graph \( K_{1.3} \) is also called a claw, and the graph \( K_4-e \) is also called a diamond. 
For a fixed graph \( H \), a graph \(  \) is said to be \emph{\( H \)-free} if no induced subgraph of \( G \) is isomorphic to \( H \). 
A graph \( G \) is said to be odd-hole-free if \( G \) is \( C_{2q+3} \)-free for every \( q\in\mathbb{N} \). 
\todo{
Reviewer G:\\
1. this would be for all holes, not only odd. \colcol{(corrected)}\\
2. there are different conventions whether N includes 0 or not, so clarify \colcol{(done)}.
}
The \emph{bipartite double} \( G\times K_2 \) of \( G \) is a graph with vertex set \( V(G)\times \mathbb{Z}_2 \), and two vertices \( (u,i) \) and \( (v,j) \) in it are adjacent if \( uv\in E(G) \) and \( i\neq j \). 
%
%

For \( k\in \mathbb{N} \), a \( k\)-colouring of a graph \( G \) is a function \( f\colon V(G)\to \mathbb{Z}_k \) such that \( f(u)\neq f(v) \) for every edge \( uv \) of \( G \). 
A coloured graph is an ordered pair \( (G,f) \) where \( G \) is a graph and \( f \) is a colouring of \( G \). 
A coloured graph \( (G,f) \) is said to be a \( k \)-coloured graph if \( f \) is a \( k \)-colouring of \( G \). 
A \emph{bicoloured component} of a \( k \)-coloured graph \( (G,f) \) is a component of the subgraph of \( G \) induced by some pair of colour classes (i.e., a component of \( G[V_i\cup V_j] \), where \( V_\ell=f^{-1}(\ell) \) for each \( \ell\in \mathbb{Z}_k \)). 
A \( k \)-colouring \( f \) of \( G \) is a \emph{\( k \)-star colouring} of \( G \) if every bicoloured component of \( (G,f) \) is a star (i.e., \( K_{1,q} \), where \( q\geq 0 \)). 
A \( k \)-colouring \( f \) of \( G \) is a \emph{distance-two \( k \)-colouring} of \( G \) if every bicoloured component of \( (G,f) \) is \( K_1 \) or \( K_2 \). 
The \emph{line graph} of a graph \( G \), denoted by \( L(G) \), is the graph with vertex set \( E(G) \), and two vertices in \( L(G) \) are adjacent if the corresponding edges in \( G \) are incident on a common vertex in \( G \). 
The \emph{clique graph} of a graph \( G \), denoted by \( K(G) \), is the intersection graph of maximal cliques in~\( G \). 
That is, the vertex set of \( K(G) \) is the set of all maximal cliques in \( G \), and two vertices in \( K(G) \) are adjacent if the corresponding cliques in \( G \) intersect. 
A graph \( G \) is said to be a \emph{clique graph} if there exists a graph \( H \) such that \( G\cong K(H) \). 
A graph \( G \) is \emph{locally linear} if every edge in \( G \) is in exactly one triangle in \( G \) \cite{froncek}. 
A locally linear graph \( G \) is an even-degree graph, and for each vertex \( v \) of \( G \), the neighbourhood of \( v \) in \( G \) induces a matching in \( G \) (hence, a locally linear graph is also called a locally matching graph). 
\todo{Reviewer F:\\
I would remove the word "perfect", because I understand a "perfect matching in G" as a matching that covers all vertices of G, which is clearly incorrect. \colcol{(corrected)}
}
For a fixed graph \( H \), a graph \( G \) is said to be \emph{locally-\( H \)} if the neighbourhood of each vertex of \( G \) induces \( H \) (i.e., \( G[N_G(v)]\cong H \) for all \( v\in V(G) \)). 
For every pair of positive integers \( q \) and \( r \), let \( F(q,r) \) denote the family of connected \( qr \)-regular graphs \( G \) such that the neighbourhood of each vertex of \( G \) induces \( qK_r \)~\cite{devillers}. 
That is, \( F(q,r) \) is the family of connected locally-\( qK_r \) graphs. 
The next observation follows from the definitions. 
\begin{observation}\label{obs:locally linear iff Fp2}
For \( p\geq 2 \), a connected \( 2p \)-regular graph \( G \) is a locally linear graph \( ( \)i.e., locally matching graph\( ) \) if and only if \( G\in F(p,2) \). 
\qed
\end{observation}
Devillers et al. proved that for every \( q\geq 2 \), \( r\geq 1 \) and \( G\in F(q,r) \), we have \( K(G)\in F(r+1,q-1) \) and \( K(K(G))\cong G \) \cite[Theorem~1.4]{devillers}.

\todo{
(omitted) [from previous version]

Let us recall the definition of the graph \( G_{2p} \) in \cite{shalu_cyriac3}. 
For each \( p\geq 2 \), the graph \( G_{2p} \) has vertex set \( \{(i,j)\in \mathbb{Z}\times \mathbb{Z}\colon 0\leq i\leq p+1,\ 0\leq j\leq p+1,\ i\neq j\} \), and \( (i,j)(k,\ell) \) is an edge in \( G_{2p} \) if (i)~\( k=j \) and \( \ell\notin \{i,j\} \), or (ii)~\( k\notin \{i,j\} \) and \( \ell=i \). 
Note that every edge of \( G_{2p} \) is of the form \( (i,j)(k,i) \) for some pairwise distinct indices \( i,j,k\in \{0,1,\dots,p+1\} \). 
}
\todo{
Reviewer G:\\
1. As this graph is important for this paper, choose another symbol than \( G_{2p} \) not to be confusely related to a general graph \( G \) (see e.g. Thm 8.). Symbols \( F_{2p} \) or \( H_{2p} \) would be possible candidates. \\ 
2. It would be significantly easier to grasp for a reader, if (oriented) edges are of form \( (i,j)(j,k) \). I do not see a reason to follow the other convention. See also my comment on Obs 7.
}
\todo[color=gray!10!white]{
The convention of orienting edges as \( (i,j)\to (j,k) \) is applied. 
This makes it clear that the oriented graph \( \overrightarrow{G_{2p}} \) is precisely the oriented line graph of \( K_{p+2} \). 
The paper is re-written accordingly. 
The notations \( G_{2p} \) and \( \overrightarrow{G_{2p}} \) are replaced with \( L^*(K_{p+2}) \) and \( \vec{L}(K_{p+2}) \), respectively. 
}

\todo[color=blue!5!white]{
Figure showing two drawings of \( G_4 \), definition of \( G_{2p} 
\), definition of \( 
\overrightarrow{G_{2p}} 
\), etc. are omittted}
%

An \emph{orientation} \( \vec{G} \) of a graph \( G \) is the directed graph obtained by assigning a direction on each edge of \( G \); that is, if \( uv \) is an edge in \( G \), then exactly one of \( (u,v) \) or \( (v,u) \) is an arc in \( \vec{G} \). 
An orientation \( \vec{G} \) is \emph{Eulerian} if the in-degree equals the out-degree for every vertex. 
A orientation \( \vec{G} \) is \emph{strongly connected} if for every pair of vertices \( u \) and \( v \), there is a directed \( u,v \)-path in \( \vec{G} \). 
%

A \emph{homomorphism} from a graph \( G \) to a graph \( H \) is a mapping \( \psi \colon V(G)\to V(H) \) such that \( \psi(u)\psi(v) \) is an edge in \( H \) whenever \( uv \) is an edge in \( G \). 
If \( \psi \) is a homomorphism from \( G \) to \( H \) and \( \psi(v)=w \), then we say that \emph{\( v \) is a copy of \( w \) in \( G \)} (under \( \psi \)). 
Let \( G \) and \( H \) be two graphs with orientations \( \vec{G} \) and \( \vec{H} \), respectively. 
A \emph{homomorphism} from the orientation \( \vec{G} \) to the orientation \( \vec{H} \) is a mapping \( \psi\colon V(\vec{G})\to V(\vec{H}) \) such that \( (\psi(u),\psi(v)) \) is an arc in \( \vec{H} \) whenever \( (u,v) \) is an arc in \( \vec{G} \). 
If \( \psi \) is a homomorphism from \( \vec{G} \) to \( \vec{H} \) and \( \psi(v)=w \), then we say that \( v \) is a copy of \( w \) in \( \vec{G} \) (under~\( \psi \)).  
We say that a homomorphism from \( G \) to \( H \) (or from \( \vec{G} \) to \( \vec{H} \)) is \emph{degree-preserving} if \( \deg_G(v)=\deg_H(\psi(v)) \) for every \( v\in V(G) \).

A \emph{locally bijective homomorphism} (in short, \emph{LBH}) from \( G \) to \( H \) is a mapping \( \psi\colon V(G)\to V(H) \) such that for every vertex \( v \) of \( G \), the restriction of \( \psi \) to the neighbourhood \( N_G(v) \) is a bijection from \( N_G(v) \) onto \( N_H(\psi(v)) \)~\cite{fiala_kratochvil} (see Figure~\ref{fig:eg LBH} for an example; observe that such a mapping \( \psi \) is always a homomorphism from \( G \) to \( H \)). 
%
A homomorphism \( \psi \) from \( G \) to \( H \) is \emph{locally injective} if for every vertex \( v \) of \( G \), the restriction of \( \psi \) to the neighbourhood \( N_G(v) \) is an injection from \( N_G(v) \) to \( N_H(\psi(v)) \)~\cite{fiala_kratochvil}. 
In other words, a homomorphism \( \psi \) from \( G \) to \( H \) is locally bijective (resp.\ injective) if for each vertex \( w \) of \( H \) and each neighbour \( x \) of \( w \) in \( H \), each copy of \( w \) in \( G \) has exactly one (resp.\ at most one) copy of \( x \) in \( G \) as its neighbour. 

\begin{figure}[hbt] 
\centering
\begin{tikzpicture}[scale=0.8,thick]
\draw (0,0) circle [radius=1.5cm];
\path
(180:1.5) node(10) [draw,fill=white,rectangle]{1}
(140:1.5) node(20) [draw,fill=white,regular polygon, regular polygon sides=5,inner sep=1.0pt]{2}
(100:1.5) node(30) [draw,fill=white,circle,inner sep=1.2pt]{3}
 (60:1.5) node(11) [draw,fill=white,rectangle]{1}
 (20:1.5) node(21) [draw,fill=white,regular polygon, regular polygon sides=5,inner sep=1.0pt]{2}
 (-20:1.5) node(31) [draw,fill=white,circle,inner sep=1.2pt]{3}
 (-60:1.5) node(12) [draw,fill=white,rectangle]{1}
(-100:1.5) node(22) [draw,fill=white,regular polygon, regular polygon sides=5,inner sep=1.0pt]{2}
(-140:1.5) node(32) [draw,fill=white,circle,inner sep=1.2pt]{3};
\path 
(10) --+(180:1.5) node(00) [draw,isosceles triangle,isosceles triangle apex angle=60,inner sep=0.5pt]{\,0}
(11) --+(60:1.5)  node(01) [draw,isosceles triangle,isosceles triangle apex angle=60,inner sep=0.5pt]{\,0}
(12) --+(-60:1.5) node(02) [draw,isosceles triangle,isosceles triangle apex angle=60,inner sep=0.5pt]{\,0};
\draw (00)--(10)  (01)--(11)  (02)--(12);

\path (-90:3.25) node{\( G \)}; 

\path (21) -- coordinate[xshift=1cm](arrowBeg) (31);
\path (arrowBeg) --+(5.5,0) coordinate(arrowEnd);
\draw [->] (arrowBeg) --node[above,align=center]{Locally bijective\\homomorphism (LBH)} (arrowEnd);

\path (arrowEnd) ++(1,0) node(0) [draw,isosceles triangle,isosceles triangle apex angle=60,inner sep=0.5pt]{\,0} ++(1.5,0) node(1) [draw,rectangle]{1} ++(1,1) node(2) [draw,regular polygon, regular polygon sides=5,inner sep=1.0pt]{2};
\path (1) ++(1,-1) node(3) [draw,circle,inner sep=1.2pt]{3};
\draw (0)--(1)--(2)--(3)--(1);
\draw (0) to[bend left=30] (2);

\path (1) +(0,-1.5) node{\( H \)};

\draw (22) to[bend left=35] (01);
\draw (21) to[bend left=35] (02);
\draw (20) to[bend right=35] (00);
\end{tikzpicture}
\caption{A locally bijective homomorphism from a graph \( G \) to a graph \( H \). 
The vertices in \( H \) are labelled distinct and are drawn by distinct shapes. 
For each vertex \( w \) of \( H \), each copy of \( w \) in \( G \) is drawn in the same shape as~\( w \) (and labelled the same). 
}
\label{fig:eg LBH}
\end{figure}

\todo[color=blue!5!white]{
Figure~\ref{fig:eg obh} is newly added.
}

\begin{figure}[hbt]
\centering
\begin{tikzpicture}[scale=0.75,thick]
\tikzset{
>={Straight Barb[length=1.0mm]}
}
\path
(180:1.5) node(10) [draw,fill=white,rectangle]{1}
(120:1.5) node(20) [draw,fill=white,regular polygon, regular polygon sides=5,inner sep=1.0pt]{2}
 (60:1.5) node(30) [draw,fill=white,circle,inner sep=1.2pt]{3}
  (0:1.5) node(11) [draw,fill=white,rectangle]{1}
(-60:1.5) node(21) [draw,fill=white,regular polygon, regular polygon sides=5,inner sep=1.0pt]{2}
(-120:1.5) node(31) [draw,fill=white,circle,inner sep=1.2pt]{3};
\path 
(10) --+(-1.5,1)  node(00) [draw,isosceles triangle,isosceles triangle apex angle=60,inner sep=0.5pt]{\,0}
(10) --+(-1.5,0)  node(01) [draw,isosceles triangle,isosceles triangle apex angle=60,inner sep=0.5pt]{\,0}
(10) --+(-1.5,-1) node(02) [draw,isosceles triangle,isosceles triangle apex angle=60,inner sep=0.5pt]{\,0};
\begin{scope}[thick,decoration={
    markings,
    mark=at position 0.5 with {\arrow{>}}},
    ] 
\draw [postaction={decorate}] (00)--(10);
\draw [postaction={decorate}] (01)--(10);
\draw [postaction={decorate}] (02)--(10);
\draw [postaction={decorate}] (10)--(20);
\draw [postaction={decorate}] (20)--(30);
\draw [postaction={decorate}] (30)--(11);
\draw [postaction={decorate}] (11)--(21);
\draw [postaction={decorate}] (21)--(31);
\draw [postaction={decorate}] (31)--(10);
\end{scope}

\path (-100:2.25) node{\( \vec{G} \)}; 

\path (11) --+(1,0) coordinate(arrowBeg);
\path (arrowBeg) --+(6.5,0) coordinate(arrowEnd);
\draw [->] (arrowBeg) --node[above,align=center]{Out-neighbourhood bijective\\ homomorphism (OBH)} (arrowEnd);

\path (arrowEnd) ++(1,0) node(0) [draw,isosceles triangle,isosceles triangle apex angle=60,inner sep=0.5pt]{\,0} ++(1.5,0) node(1) [draw,rectangle]{1} ++(1,1) node(2) [draw,regular polygon, regular polygon sides=5,inner sep=1.0pt]{2};
\path (1) ++(1,-1) node(3) [draw,circle,inner sep=1.2pt]{3};
\begin{scope}[thick,decoration={
    markings,
    mark=at position 0.5 with {\arrow{>}}},
    ] 
\draw [postaction={decorate}] (0)--(1);
\draw [postaction={decorate}] (1)--(2);
\draw [postaction={decorate}] (2)--(3);
\draw [postaction={decorate}] (3)--(1);
\end{scope}

\path (1) +(0,-1.5) node{\( \vec{H} \)}; 
\end{tikzpicture}
\caption{An out-neighbourhood bijective homomorphism from an oriented graph \( 
\vec{G} \) to an oriented graph \( \vec{H} \). 
The vertices in \( \vec{H} \) are labelled distinct and are drawn by distinct shapes. 
For each vertex \( w \) of \( \vec{H} \), each copy of \( w \) in \( \vec{G} \) is drawn in the same shape as \( w \) (and labelled the same). 
}
\label{fig:eg obh}
\end{figure}

Similar to the notion of locally constrained homomorphisms between (undirected) graphs, one can define locally constrained homomorphisms between orientations (i.e., oriented graphs) and between directed graphs. 
Such notions can be defined in a variety of ways \cite{bard_etal2018}. 
A notion of locally injective homomorphism between directed graphs is defined and studied by MacGillivray and Swarts~\cite{macgillivray}. 
We are interested in a similar notion of locally constrained homomorphisms between orientations. 
Let \( G \) and \( H \) be graphs with orientations \( \vec{G} \) and \( \vec{H} \), respectively. 
Recall that a homomorphism from \( \vec{G} \) to \( \vec{H} \) is a mapping \( \psi\colon V(\vec{G})\to V(\vec{H}) \) such that \( (\psi(u),\psi(v)) \) is an arc in \( \vec{H} \) whenever \( (u,v) \) is an arc in \( \vec{G} \). 
We define an \emph{out-neighbourhood bijective homomorphism} from \( \vec{G} \) to \( \vec{H} \) as a mapping \( \psi\colon V(\vec{G})\to V(\vec{H}) \) such that for every vertex \( v \) of \( \vec{G} \), the restriction of \( \psi \) to the out-neighbourhood \( N_{\vec{G}}^+(v) \) is a bijection from \( N_{\vec{G}}^+(v) \) to \( N_{\vec{H}}^+(\psi(v)) \). 
Observe that out-neighbourhood bijective homomorphisms from \( \vec{G} \) to \( \vec{H} \) are indeed homomorphisms from \( \vec{G} \) to \( \vec{H} \). 
A homomorphism \( \psi \) from \( \vec{G} \) to \( \vec{H} \) is \emph{out-neighbourhood injective} is if for every vertex \( v \) of \( \vec{G} \), the restriction of \( \psi \) to the out-neighbourhood \( N_{\vec{G}}^+(v) \) is an injection from \( N_{\vec{G}}^+(v) \) to \( N_{\vec{H}}^+(\psi(v)) \). 
In other words, a homomorphism \( \psi \) from \( \vec{G} \) to \( \vec{H} \) is out-neighbourhood bijective (resp.\ injective) if for each vertex \( w \) of \( \vec{H} \) and each out-neighbour \( x \) of \( w \) in \( \vec{H} \), each copy of \( w \) in \( \vec{G} \) has exactly one (resp.\ at most one) copy of \( x \) in \( \vec{G} \) as its out-neighbour. 
\todo{
Reviewer G:\\
1 provide reference where these concepts were defined earlier.  
\\
2. use ``mapping'' instead of ``function''. Though these are often considered as synonyms, ``functions'' should be reserved to real/complex-valued ones. \colcol{(done)}
}

\todo[color=gray!10!white]{
We introduce out-neighbourhood bijective/injective homomorphisms in this paper, motivated by (in-neighbourhood) injective homomorphism of MacGillivray and Swarts~\cite{macgillivray}. 
}

\inlinenew{\textbf{``}}
Bard et al.~\cite{bard_etal2018} defined various notions of locally injective homomorphisms between oriented graphs. 
A homomorphism \( \psi \) from a oriented graph \( \vec{G} \) to a oriented graph \( \vec{H} \) is called a locally bijective homomorphism if the following hold for every vertex \( v \) of \( \vec{G} \): (i)~\( \psi \) maps the in-neighbours of \( v \) bijectively to in-neighbours of \( \psi(v) \), and (ii)~\( \psi \) maps the out-neighbours of \( v \) bijectively to out-neighbours of \( \psi(v) \) (in the terminology of Bard et al.~\cite{bard_etal2018}, it is an ios-bijective homomorphism, where `ios' is short for `in and out separately'). 
In contrast, an iot-bijective homomorphism between two oriented graphs (where `iot' stands for `in and out together') is precisely a locally bijective homomomorphism between the underlying undirected graphs. 
It is a folklore result that for two graphs \( G \) and \( H \), a mapping \( \psi\colon V(G)\to V(H) \) is an LBH from \( G \) to \( H \) if and only if \( \psi \) is an LBH from an orientation of \( G \) to an orientation of \( H \). 
Ne\v{s}et\v{r}il and Mendez~\cite{nesetril_mendez2006} introduced \( d \)\nobreakdash-folding, a version of homomorphism between directed graphs which is a much stronger notion than in-neighbourhood injective homomorphism. 

It is well-known that if there is an LBH from a graph \( G \) to a graph \( H \), then every eigenvalue of (adjacency matrix of) \( H \) with geometric multiplicity \( t \) is an eigenvalues of \( G \) with geometric multiplicity at least \( t \)~\cite{sachs}, and as a result the characteristic polynomial of (adjacency matrix of) \( H \) divides the characteristic polynomial of \( G \) (because adjacency matrices of \( G \) and \( H \) are diagonalizable); see \cite{godsil_royle} for an alternate proof. 
\inlinenew{\textbf{''}}
Observe that, in general, a locally bijective homomorphism need not preserve subgraphs (or induced subgraphs for that matter). 
Nevertheless, they preserve subgraphs of diameter~2. 
\begin{observation}\label{obs:lbh preseves diameter 2 subgraphs}
Let \( J \) be a graph of diameter~2, and let \( G \) be a graph that contains \( J \) as subgraph. 
If \( G \) admits an LBH to a graph \( H \), then \( H \) contains \( J \) as subgraph. 
\end{observation}
\begin{proof}
Suppose that \( G \) admits an LBH to a graph \( H \). 
Since \( \psi \) is a homomorphism from \( G \) to \( H \),\\ 
\( \psi(x)\psi(y) \) is an edge in \( H \) for every edge \( xy \) of \( J \). 
Hence, \( H[\psi(V(J))] \) contains \( J \) as subgraph, provided that \( \psi \) maps distinct vertices of \( J \) to distinct vertices in \( H \). 
Thus, to prove that \( H \) contains~\( J \) as subgraph, it suffices to show that \( \psi \) maps distinct vertices of \( J \) to distinct vertices in~\( H \). 
On the contrary, assume that \( u,v\in V(J) \) and \( \psi(u)=\psi(v) \). 
Hence, \( uv\notin E(G) \) (because \( \psi \) is a homomorphism from \( G \) to \( H \)). 
Since \( J \) is a graph of diameter~2, the vertices \( u \) and \( v \) of \( J \) have a common neighbour \( x \). 
Since \( \psi \) is an LBH from \( G \) to \( H \), the restriction of \( \psi \) to \( N_G(x) \) is a bijection from \( N_G(x) \) onto \( N_H(\psi(x)) \). 
Since \( u \) and \( v \) are distinct members of \( N_G(x) \), the mapping \( \psi \) maps \( u \) and \( v \) to distinct vertices in \( H \). 
This is a contradiction to \( \psi(u)=\psi(v) \). 
\end{proof}

\subsection{Star Colourings and Orientations}
\todo[color=blue!5!white]{
Subsection is re-written. 
The notion of coloured in-orientation in the literature is used more prominently (in place of in-orientation), and we define coloured MINI-orientation similarly. 
}
%
Star colourings are known to be closely related with a particular type of graph orientations called in-orientations. 
In this section, we define a special type of in-orientations. 
Let \( G \) be an (undirected) graph. 
A \emph{coloured orientation of \( G \)} is an ordered pair \( (\vec{G},f) \) such that \( \vec{G} \) is an orientation of \( G \) and \( f \) is a colouring of \( G \). 
A coloured orientation \( (\vec{G},f) \) is called a \emph{coloured in-orientation} if the edges in each bicoloured 3-vertex path in \( (\vec{G},f) \) are oriented towards the middle vertex~\cite{dvorak_esperet}. 
If \( (\vec{G},f) \) is a coloured in-orientation of a \( G \), then \( f \) is a star colouring of \( G \). 
On the other hand, if \( f \) is a star colouring of \( G \), then \( (\vec{G},f) \) is a coloured in-orientation of \( G \), where \( \vec{G} \) is obtained by orienting edges in each bicoloured 3-vertex path in \( (G,f) \) towards the middle vertex, and then orienting the remaining edges arbitrarily; let us call \( \vec{G} \) an \emph{in-orientation of \( G \) induced by \( f \)}. 
\begin{observation}[\cite{albertson}]
\label{obs:in-orientation underlying colouring is star}
\( (\vec{G},f) \) is a coloured in-orientation of a graph \( G \) if and only if \( f \) is a star colouring of \( G \) and \( \vec{G} \) an in-orientation of \( G \) induced by \( f \). 
\qed
\end{observation}

In-orientation of \( G \) induced by \( f \) is unique if and only if no bicoloured component of \( (G,f) \) is (isomorphic to) \( K_{1,1} \).

An coloured orientation \( (\vec{G},f) \) of \( G \) is a \emph{\( q \)-coloured in-orientation} if \( f \) is a \( q \)-colouring of \( \vec{G} \) and the following hold for every vertex \( v \) of \( (\vec{G},f) \):\\
(i)~no out-neighbour of \( v \) has the same colour as an in-neighbour of \( v \), and\\
(ii)~out-neighbours of \( v \) have pairwise distinct colours.\\
A coloured in-orientation of \( G \) is a \( q \)-coloured in-orientation of \( G \) for some \( q\in \mathbb{N} \). 
Clearly, a graph \( G \) admits a \( q \)-star colouring if and only if \( G \) admits a \( q \)-coloured in-orientation. 
This gives a characterization of the star chromatic number of \( G \) which is equivalent to the characterization of Ne\v{s}et\v{r}il and Mendez~\cite{nesetril_mendez2003}: the star chromatic number of \( G \) is equal to the least integer \( q \) such that \( G \) admits a \( q \)-coloured in-orientation. 

We say that a coloured orientation \( (\vec{G},f) \) of \( G \) is a \emph{\( q \)-coloured Monochromatic In-Neighbourhood In-orientation} (in short, \emph{\( q \)-coloured MINI-orientation}) if \( f \) is a \( q \)-colouring of \( \vec{G} \) and the following hold for every vertex \( v \) of \( (\vec{G},f) \):\\
(i)~no out-neighbour of \( v \) has the same colour as an in-neighbour of \( v \),\\
(ii)~out-neighbours of \( v \) have pairwise distinct colours, and\\
(iii)~all in-neighbours of \( v \) have the same colour.

Observe that compared to \( q \)-coloured in-orientation, the only new condition is Condition~(iii), which says that the in-neighbourhood of \( v \) is monochromatic (hence the name monochromatic in-neighbourhood in-orientation). 
Note that every graph \( G \) admits a \( q \)-coloured in-orientation with \( q=|V(G)| \) (assign pairwise distinct colours on vertices). 
In contrast, due to the addition of Condition~(iii), there exist graphs that do not admit a \( q \)-coloured MINI-orientation for any~\( q \) (for instance, see Theorem~\ref{thm:CL2r+1 obstruction to CEO}). 
%
%

\begin{figure}[hbtp]
\centering
\tikzset{
>={Straight Barb[length=1.0mm]}
}
\begin{tikzpicture}[scale=0.75]
\path (-2,-2) node(01)[evenbiggerdot]{}--(-2,2) node(13)[evenbiggerdot]{}--(2,2) node(32)[evenbiggerdot]{}--(2,-2) node(20)[evenbiggerdot]{}--cycle;
\path (01)--node(30)[evenbiggerdot]{} (13)--node(21)[evenbiggerdot]{} (32)--node(03)[evenbiggerdot]{} (20)--node(12)[evenbiggerdot]{} (01);
\path (30)--node(02)[evenbiggerdot]{} (21)--node(10)[evenbiggerdot]{} (03)--node(31)[evenbiggerdot]{} (12)--node(23)[evenbiggerdot]{} (30);
\begin{scope}[thick,decoration={
    markings,
    mark=at position 0.5 with {\arrow{>}}}
    ] 
\draw [postaction={decorate}] (01)--(30);
\draw [postaction={decorate}] (30)--(13);
\draw [postaction={decorate}] (13)--(21);
\draw [postaction={decorate}] (21)--(32);
\draw [postaction={decorate}] (32)--(03);
\draw [postaction={decorate}] (03)--(20);
\draw [postaction={decorate}] (20)--(12);
\draw [postaction={decorate}] (12)--(01);

\draw [postaction={decorate}] (12)--(31);
\draw [postaction={decorate}] (31)--(03);
\draw [postaction={decorate}] (03)--(10);
\draw [postaction={decorate}] (10)--(21);
\draw [postaction={decorate}] (21)--(02);
\draw [postaction={decorate}] (02)--(30);
\draw [postaction={decorate}] (30)--(23);
\draw [postaction={decorate}] (23)--(12);

\draw [postaction={decorate}] (02)--(10);
\draw [postaction={decorate}] (10)--(31);
\draw [postaction={decorate}] (31)--(23);
\draw [postaction={decorate}] (23)--(02);

\draw [postaction={decorate}] (01) to[bend right=35] (20);
\draw [postaction={decorate}] (20) to[bend right=35] (32);
\draw [postaction={decorate}] (32) to[bend right=35] (13);
\draw [postaction={decorate}] (13) to[bend right=35] (01);
\end{scope}
\node at (01) {0};
\node at (02) {0};
\node at (03) {0};
\node at (10) {1};
\node at (12) {1};
\node at (13) {1};
\node at (20) {3};
\node at (21) {3};
\node at (23) {3};
\node at (30) {2};
\node at (31) {2};
\node at (32) {2};
\end{tikzpicture}
\caption{A 4-coloured MINI-orientation of \( L(Q_3) \) (the cuboctahedral graph).
}\label{fig:4-CEO of G4}
\end{figure}

A 4-coloured MINI-orientation of the cuboctahedral graph is shown in Figure~\ref{fig:4-CEO of G4}. 
For each \( q\in \mathbb{N} \), admitting a \( q \)-coloured MINI-orientation is a hereditary property: if a graph \( G \) admits a \( q \)-coloured MINI-orientation  \( (\vec{G},f) \), then every subgraph \( H \) of \( G \) admits a \( q \)-coloured MINI-orientation (of the form \( (\vec{H},f_{\restriction V(H)}) \)).

\subsection{Oriented line graph}\label{sec:oriented line graph def}
\todo[color=blue!5!white]{
Newly added subsection. 
}
\todo[color=blue!5!white]{
\inline{Figure to be added: example for oriented line graph} 

[as in my presentation: \( G=\mathrm{diamond} \), \( \vec{D}(G) \), \( \vec{L}(G) \) top of \( \vec{D}(G) \), and \( \vec{L}(G) \)]
}
\begin{figure}[hbtp]
\centering
\tikzset{
>={Straight Barb[length=1.0mm]},
circnode/.style={draw,circle,minimum size=5pt},
}
\begin{subfigure}[c]{0.5\textwidth}
\centering
\begin{tikzpicture}[thick]
\path (0:2) node(1)[circnode]{};
\path (120:2) node(2)[circnode]{};
\path (-120:2) node(3)[circnode]{};
\path (1) --+(3.5,0) node(0)[circnode]{};

\draw (0)--(1)--(2)--(3)--(1);
\end{tikzpicture}
\caption{\( G \)}
\end{subfigure}%
\begin{subfigure}[c]{0.5\textwidth}
\centering
\begin{tikzpicture}[thick]
\path (0:2) node(1)[circnode]{};
\path (120:2) node(2)[circnode]{};
\path (-120:2) node(3)[circnode]{};
\path (1) --+(3.5,0) node(0)[circnode]{};

\begin{scope}[thick,decoration={
    markings,
    mark=at position 0.525 with {\arrow{>}}}
    ] 
\draw [postaction={decorate}] (0) to[bend left=20] node(01)[pos=0.5]{} (1);
\draw [postaction={decorate}] (1) to[bend left=20] node(10)[pos=0.5]{} (0);
\draw [postaction={decorate}] (1) to[bend left=20] node(12)[pos=0.5]{} (2);
\draw [postaction={decorate}] (2) to[bend left=20] node(21)[pos=0.5]{} (1);
\draw [postaction={decorate}] (1) to[bend left=20] node(13)[pos=0.5]{} (3);
\draw [postaction={decorate}] (3) to[bend left=20] node(31)[pos=0.5]{} (1);
\draw [postaction={decorate}] (2) to[bend left=20] node(23)[pos=0.5]{} (3);
\draw [postaction={decorate}] (3) to[bend left=20] node(32)[pos=0.5]{} (2);
\end{scope}
\end{tikzpicture}
\caption{\( \vec{D}(G) \)}
\end{subfigure}%

\begin{subfigure}[c]{0.5\textwidth}
\centering
\begin{tikzpicture}[thick]
\begin{scope}[opacity=0]
\path (0:2) node(1)[circnode]{};
\path (120:2) node(2)[circnode]{};
\path (-120:2) node(3)[circnode]{};
\path (1) --+(3.5,0) node(0)[circnode]{};
\end{scope}

\begin{scope}[thick,opacity=0,decoration={
    markings,
    mark=at position 0.525 with {\arrow{>}}}
    ] 
\draw [postaction={decorate}] (0) to[bend left=20] node(01)[pos=0.5]{} (1);
\draw [postaction={decorate}] (1) to[bend left=20] node(10)[pos=0.5]{} (0);
\draw [postaction={decorate}] (1) to[bend left=20] node(12)[pos=0.5]{} (2);
\draw [postaction={decorate}] (2) to[bend left=20] node(21)[pos=0.5]{} (1);
\draw [postaction={decorate}] (1) to[bend left=20] node(13)[pos=0.5]{} (3);
\draw [postaction={decorate}] (3) to[bend left=20] node(31)[pos=0.5]{} (1);
\draw [postaction={decorate}] (2) to[bend left=20] node(23)[pos=0.5]{} (3);
\draw [postaction={decorate}] (3) to[bend left=20] node(32)[pos=0.5]{} (2);
\end{scope}

\node (01-) at (01) [circnode]{};
\node (10-) at (10) [circnode]{};
\node (12-) at (12) [circnode]{};
\node (21-) at (21) [circnode]{};
\node (13-) at (13) [circnode]{};
\node (31-) at (31) [circnode]{};
\node (23-) at (23) [circnode]{};
\node (32-) at (32) [circnode]{};

\begin{scope}[thick,decoration={
    markings,
    mark=at position 0.5 with {\arrow{>}}}
    ] 
\draw [postaction={decorate}] (12-)--(23-);
\draw [postaction={decorate}] (23-)--(31-);
\draw [postaction={decorate}] (31-)--(12-);
\draw [postaction={decorate}] (21-)--(13-);
\draw [postaction={decorate}] (13-)--(32-);
\draw [postaction={decorate}] (32-)--(21-);
\draw [postaction={decorate}] (01-) to[bend right](12-);
\draw [postaction={decorate}] (31-) to[bend right] (10-);
\draw [postaction={decorate}] (21-) to[bend left] (10-);
\draw [postaction={decorate}] (01-) to[bend left](13-);
\end{scope}
\end{tikzpicture}
\caption{\( \vec{L}(G) \)}
\end{subfigure}%
\begin{subfigure}[c]{0.5\textwidth}
\centering
\begin{tikzpicture}[thick]
\begin{scope}[draw=gray,fill=gray]
\path (0:2) node(1)[circnode]{};
\path (120:2) node(2)[circnode]{};
\path (-120:2) node(3)[circnode]{};
\path (1) --+(3.5,0) node(0)[circnode]{};
\end{scope}

\begin{scope}[thick,draw=gray,fill=gray,decoration={
    markings,
    mark=at position 0.525 with {\arrow{>}}}
    ] 
\draw [postaction={decorate}] (0) to[bend left=20] node(01)[pos=0.5]{} (1);
\draw [postaction={decorate}] (1) to[bend left=20] node(10)[pos=0.5]{} (0);
\draw [postaction={decorate}] (1) to[bend left=20] node(12)[pos=0.5]{} (2);
\draw [postaction={decorate}] (2) to[bend left=20] node(21)[pos=0.5]{} (1);
\draw [postaction={decorate}] (1) to[bend left=20] node(13)[pos=0.5]{} (3);
\draw [postaction={decorate}] (3) to[bend left=20] node(31)[pos=0.5]{} (1);
\draw [postaction={decorate}] (2) to[bend left=20] node(23)[pos=0.5]{} (3);
\draw [postaction={decorate}] (3) to[bend left=20] node(32)[pos=0.5]{} (2);
\end{scope}

\node(01-) at (01) [circnode]{};
\node(10-) at (10) [circnode]{};
\node(12-) at (12) [circnode]{};
\node(21-) at (21) [circnode]{};
\node(13-) at (13) [circnode]{};
\node(31-) at (31) [circnode]{};
\node(23-) at (23) [circnode]{};
\node(32-) at (32) [circnode]{};

\begin{scope}[thick,decoration={
    markings,
    mark=at position 0.5 with {\arrow{>}}}
    ] 
\draw [postaction={decorate}] (12-)--(23-);
\draw [postaction={decorate}] (23-)--(31-);
\draw [postaction={decorate}] (31-)--(12-);
\draw [postaction={decorate}] (21-)--(13-);
\draw [postaction={decorate}] (13-)--(32-);
\draw [postaction={decorate}] (32-)--(21-);
\draw [postaction={decorate}] (01-) to[bend right](12-);
\draw [postaction={decorate}] (31-) to[bend right] (10-);
\draw [postaction={decorate}] (21-) to[bend left] (10-);
\draw [postaction={decorate}] (01-) to[bend left](13-);
\end{scope}
\end{tikzpicture}
\caption{\( \vec{L}(G) \) drawn on top of \( \vec{D}(G) \)}
\end{subfigure}%
\caption{An example of the oriented line graph operation.}
\label{fig:eg Larrow}
\end{figure}

Oriented line graphs were introduced by Kotani and Sunada~\cite{kotani_sunada} to study the Ihara-Zeta functions of graphs, and they are also used to study Ramanujan graphs~\cite{murty} and Ramanujan digraphs~\cite{parzanchevski}. 
For instance, it is known that a regular graph is a Ramanujan graph if and only if its oriented line graph is a Ramanujan digraph \cite{parzanchevski}. 
Let \( G \) be a (simple undirected) graph, and let \( \vec{D}(G) \) be the oriented graph obtained from \( G \) be replacing each edge \( uv \) of \( G \) by two arcs \( (u,v) \) and \( (v,u) \). 
The \emph{oriented line graph} of \( G \) is the oriented graph with the set of arcs of \( \vec{D}(G) \) as its vertex set, and there is an arc in it from a vertex \( (u,v) \) to a vertex \( (v,w) \) when \( u\neq w \) (see Figure~\ref{fig:eg Larrow}). 
Another way to obtain the oriented line graph of \( G \) is to take the line digraph of \( \vec{D}(G) \) as defined in \cite{beineke_bagga2021a}, and then remove all parallel edges in that graph (i.e, if \( t\geq 2 \) edges exist between two vertices, then remove all those \( t \) edges). 
Note that this definition is not artificial since two parallel edges in opposite directions `cancel each other' in some related contexts such as signed graphs. 
We denote the oriented line graph of \( G \) by \( \vec{L}(G) \) and its underlying undirected graph by \( L^*(G) \).

\section{Tools}\label{sec:prelims}
\todo[color=blue!5!white]{
Section name is changed (previously `Preliminaries').
}

\subsection{Oriented line graph}\label{sec:oriented line graph}
\todo[color=blue!5!white]{
Newly added subsection. 
}
In this subsection, we show that \( L^*(H) \) is a 2-lift of \( L(H) \) for every \( H \) (i.e., \( L^*(H) \) has double the number of vertices and admits an LBH to \( L(H) \)). 
\begin{theorem}\label{thm:LstarH has LBH to LH}
For every graph \( H \), there is an LBH from \( L^*(H) \) to \( L(H) \). 
\end{theorem}
\begin{proof}
Let \( H \) be a graph. 
Note that the vertex set of the graph \( L^*(H) \) is \( \cup_{uv\in E(H)}\{(u,v),(v,u)\} \). 
Define \( \psi\colon V(L^*(H)) \to E(H) \) as \( \psi((u,v))=\{u,v\} \) for every \( (u,v)\in V(L^*(H)) \). 
For each edge \( \{(u,v),(v,w)\} \) of \( L^*(H) \) where \( u,v,w\in V(H) \), we have \( \{\psi((u,v)),\psi((v,w))\}=\{uv,vw\} \) is an edge in \( L(H) \). 
Hence, \( \psi \) is a homomorphism from \( L^*(H) \) to \( L(H) \). 
It remains to prove that \( \psi \) is locally bijective. 
To prove this, it suffices to show that for an arbitrary vertex \( x \) of \( L(H) \), and an arbitrary copy \( w \) of \( x \) in \( L^*(H) \) under \( \psi \) (i.e., \( \psi(w)=x \)), the members of \( N_{L^*(H)}(w) \) are precisely copies of members of \( N_{L(H)}(x) \) in \( L^*(H) \) in a bijective fashion. 
To this end, consider an arbitrary vertex \( u_1v_1 \) of \( L(H) \), where \( u_1,v_1\in V(H) \). 
Let \( u_1,u_2,\dots,u_k \) be the neighbours of \(v_1 \) in \( H \), and let \( v_1,v_2,\dots,v_\ell \) be the neighbours of \( u_1 \) in \( H \) (where \( k,\ell\in \mathbb{N} \)). 
The neighbours of \( u_1v_1 \) in \( L(H) \) are \( u_1v_2,\dots,u_1v_\ell \) (provided \( \ell>1 \)) and \( v_1u_2,\dots,v_1u_k \) (provided \( k>1 \)). 
By the definition of \( \psi \), for each vertex \( yz \) in \( L(H) \) (where \( y,z\in V(H) \)), the copies of \( yz \) in \( L^*(H) \) (under \( \psi \)) are \( (y,z) \) and \( (z,y) \). 
In particular, the copies of \( u_1v_1 \) in \( L^*(H) \) are \( (u_1,v_1) \) and \( (v_1,u_1) \). 
The neighbours of \( (u_1,v_1) \) in \( L^*(H) \) are \( (v_1,u_2),\dots,(v_1,u_k), \)\( (v_2,u_1),\dots,(v_\ell,u_1) \), which are precisely copies of \( v_1u_2,\dots,v_1u_k, \)\( u_1v_2,\dots,u_1v_\ell \) in \( L^*(H) \) respectively in a bijective fashion. 
Similarly, the neighbours of \( (v_1,u_1) \) in \( L^*(H) \) are \( (u_1,v_2),\dots,(u_1,v_\ell), \)\allowbreak \( (u_2,v_1),\dots,(u_k,v_1)  \), which are precisely copies of \( u_1v_2,\dots,u_1v_\ell, \)\( v_1u_2,\dots,v_1u_k \) in \( L^*(H) \) respectively in a bijective fashion. 
That is, for each copy \( w \) of \( u_1v_1 \) in \( L^*(H) \), the members of \( N_{L^*(H)}(w) \) are precisely copies of members of \( N_{L(H)}(u_1v_1) \) in \( L^*(H) \) in a bijective fashion.
Since \( u_1v_1\in V(L(H)) \) is arbitrary, \( \psi \) in an LBH from \( L^*(H) \) to \( L(H) \).
\end{proof}

\subsection{Locally Bijective Homomorphisms and the line graph operation}
In this subsection, we prove that locally bijective homomorphisms between 3-regular graphs, in a sense, behave well with respect to the line graph operation.
%
%
\todo[color=blue!5!white]{
Changed from observations to lemmas
}
\begin{lemma}\label{lem:lbh leads to lbh btwn line graphs}
If \( \psi \) is a locally bijective homomorphism from a graph \( G \) to a graph \( H \), then the mapping \( \psi'\colon E(G)\to E(H) \), defined as \( \psi'(uv)=\psi(u)\psi(v) \) for all \( uv\in E(G) \) \( ( \)where \( u,v\in V(G) \)\( ) \) is a locally bijective homomorphism from \( L(G) \) to \( L(H) \). 
\end{lemma}
\begin{proof}
Let \( \psi \) be an LBH from \( G \) to \( H \). 
Define \( \psi'\colon E(G)\to E(H) \) as \( \psi'(uv)=\psi(u)\psi(v) \) for all \( uv\in E(G) \), where \( u,v\in V(G) \). 
Consider an arbitrary edge \( u_1v_1 \) of \( G \), where \( u_1,v_1\in V(G) \). 
To prove the observation, it suffices to establish the following claim (since \( u_1v_1\in E(G) \) is arbitrary). 
\begin{myclaim}\label{clm:line graph operation keeps lbh}
    the restriction of \( \psi' \) to the set \( N_{L(G)}(u_1v_1) \) is a bijection from \( N_{L(G)}(u_1v_1) \) onto \( N_{L(H)}(\psi'(u_1v_1)) \) \( \big( \)i.e., \( N_{L(H)}(\psi(u_1)\psi(v_1)) \)\( \big) \). 
\end{myclaim}
Let \( u_1,u_2,\dots,u_k \) be the neighbours of \( v_1 \) in~\( G \), and let \( v_1,v_2,\dots,v_\ell \) be the neighbours of \( u_1 \) in~\( G \), where \( k,\ell\in \mathbb{N} \).
The neighbours of \( u_1v_1 \) in \( L(G) \) are \( u_1v_2,\dots,u_1v_\ell,u_2v_1,\dots,u_kv_1 \). 
Since \( \psi \) is an LBH from \( G \) to \( H \), the restriction of \( \psi \) to \( N_G(u_1) \) is a bijection from \( N_G(u_1) \) onto \( N_H(\psi(u_1)) \). 
Hence, \( \psi(v_1), \psi(v_2), \dots, \psi(v_\ell) \) are pairwise distinct. 
Similarly, \( \psi(u_1), \psi(u_2), \dots, \psi(u_k) \) are pairwise distinct. 

Since \( \psi \) is a homomorphism from \( G \) to \( H \), the vertices \( \psi(u_1), \psi(u_2), \dots, \psi(u_k) \) are neighbours of \( \psi(v_1) \) in \( H \). 
Similarly, \( \psi(v_1), \psi(v_2), \dots, \psi(v_\ell) \) are neighbours of \( \psi(u_1) \) in \( H \). 
Hence, \( \psi(u_2)\psi(v_1), \dots \psi(u_k)\psi(v_1), \psi(u_1)\psi(v_2), \dots, \psi(u_1)\psi(v_\ell) \) are neighbours of \( \psi(u_1)\psi(v_1) \) in \( L(H) \). 
Since \( \psi'(u_iv_1)=\psi(u_i)\psi(v_1) \) for \( 2\leq i\leq k \) and \( \psi'(u_1v_j)=\psi(u_1)\psi(v_j) \) for \( 2\leq j\leq \ell \), the mapping \( \psi' \) maps each member of \( N_{L(G)}(u_1v_1) \) to a member of \( N_{L(H)}(\psi(u_1)\psi(v_1)) \). 
Hence, to prove Claim~\ref{clm:line graph operation keeps lbh}, it suffices to show that no two members of \( N_{L(G)}(u_1v_1) \) are mapped to the same vertex in \( L(H) \) by \( \psi' \). 
\inline{``}
To produce a contradiction, assume on the contrary that one of the following holds:\\ (Case~1)~\( \psi'(u_iv_1)=\psi'(u_jv_1) \) for some distinct \( i,j\in \{2,\dots,k\} \);\\ (Case~2)~\( \psi'(u_iv_1)=\psi'(u_1v_j) \) for some \( i\in \{2,\dots,k\} \) and \( j\in \{2,\dots,\ell\} \); or\\ (Case~3)~\( \psi'(u_1v_i)=\psi'(u_1v_j) \) for some distinct \( i,j\in \{2,\dots,\ell\} \). 

~\\
\noindent \emph{Case~1:} \( \psi'(u_iv_1)=\psi'(u_jv_1) \) for some distinct \( i,j\in \{2,\dots,k\} \), say \( i=2 \) and \( j=3 \). 

\noindent Since \( \psi'(u_2v_1)=\psi'(u_3v_1) \), we have \( \psi(u_2)\psi(v_1)=\psi(u_3)\psi(v_1) \) (as edges in \( L(H) \)). 
That is, \( \{\psi(u_2),\psi(v_1)\}=\{\psi(u_3),\psi(v_1)\} \), or in other words, \( \psi(u_2)=\psi(u_3) \). 
This is a contradiction since \( \psi(u_2),\psi(u_3),\dots,\psi(u_k) \) are pairwise distinct. 

~\\
\noindent \emph{Case~2:} \( \psi'(u_iv_1)=\psi'(u_1v_j) \) for some \( i\in \{2,\dots,k\} \) and\( j\in \{2,\dots,\ell\} \), say \( i=j=2 \). 

\noindent Since \( \psi'(u_2v_1)=\psi(u_1v_2) \), we have \( \psi(u_2)\psi(v_1)=\psi(u_1)\psi(v_2) \) (as edges in \( L(H) \)). 
That is, \( \{\psi(u_2),\psi(v_1)\}=\{\psi(u_1),\psi(v_2)\} \). 
But, \( \psi(v_1)\notin \{\psi(u_1),\psi(v_2)\} \) because \( \psi(v_1)\neq \psi(u_1) \) (since \( \psi \) is a homomorphism) and \( \psi(v_1)\neq \psi(v_2) \) (since \( \psi(v_1),\psi(v_2),\dots,\psi(v_\ell) \) are pairwise distinct); a contradiction. 

~\\
\noindent By symmetry, Case~3 leads to a contradiction like Case~1. 
Since all three cases lead to contradictions, Claim~\ref{clm:line graph operation keeps lbh} is proved. 
\inline{''}
\todo{
Reviewer G:\\
Simplify by use of Obs 2 as the image of a \( K_3 \) in \( G \) is a \( K_3 \) in \( H \), which is forbidden.
}
\todo[color=gray!10!white]{
Lemma~\ref{lem:lbh leads to lbh btwn line graphs} is for general graphs \( G \) and \( H \); that is, not necessarily 3-regular, and \( H \) need not be triangle-free. 
}

\end{proof}

\begin{lemma}\label{lem:lbh sometimes leads to lbh between clique graphs}
Let \( q\in \mathbb{N} \), and let \( G \) and \( H \) be two graphs such that all maximal cliques in \( G \) and \( H \) are \( q \)-cliques. 
If \( G \) admits an LBH to \( H \), then \( K(G) \) admits an LBH to \( K(H) \). 
\end{lemma}
\todo{
Reviewer G:\\
Consider to use the well established term ``\( G \) covers \( H \)''. 
}
\todo[color=gray!10!white]{
The usage is changed from `there is an LBH from \( G \) to \( H \)' to `\( G \) admits an LBH to \( H \)'. 
We prefer the homomorphism terminology to have fewer definitions. 
To make comparison with OBH easier, it is better to keep the term `locally bijective homomorphism (LBH)' instead of `covering projection'. 
Since the term covering projection is not used, it feels more consistent to use `\( G \) admits an LBH to \( H \)' and `\( G \) admits an LBH \( \psi \) to \( H \)' (instead of `\( G \) covers \( H\)' and `\( G \) admits an LBH \( \psi \) to \( H \)'). 
Besides, we feel that the direction of the covering projection is not immediately clear with the usage `\( G \) covers \( H \)'; for instance, although unrelated, in the edge cover terminology, \( G \) covers \( H \) means that copies of \( G \) form an edge cover of \( H \). 
}
%
\begin{proof}
Suppose that \( \psi \) is an LBH from \( G \) to \( H \). 
Recall that for each \( S\subseteq V(G) \), the set-image \( \psi(S)=\{\psi(x)\colon x\in S\} \). 
Since \( \psi \) is a homomorphism from \( G \) to \( H \), for each clique \( K \) of \( G \), the set-image \( \psi(K) \) is a clique of \( H \). 
Since maximal cliques of \( H \) are \( q \)-cliques, for each maximal clique \( K \) of \( G \), the set-image \( \psi(K) \) is a maximal clique of \( H \). 
Define \( \psi^*\colon V(K(G))\to V(K(H)) \) as \( \psi^*(K)=\psi(K) \) for each maximal clique \( K \) of \( G \) (here, \( \psi^*(K) \) is the image of \( \psi^* \) at \( K \), whereas  \( \psi(K) \) is a set-image). 
Observe that \( \psi^* \) is well-defined. 
\setcounter{myclaim}{0}
\begin{myclaim}\label{clm:psi star is homomorphism}
    \( \psi^* \) is a homomorphism from \( K(G) \) to \( K(H) \). 
\end{myclaim}
\noindent Consider an arbitrary edge \( c_1c_2 \) of \( K(G) \), where \( c_1 \) and \( c_2 \) are maximal cliques in \( G \). 
We need to show that \( \psi^*(c_1)\psi^*(c_2) \) is an edge in \( K(H) \) \( \big( \)i.e., \( \psi(c_1)\psi(c_2)\in E(K(H)) \)\( \big) \). 
Since \( c_1c_2 \) is an edge in \( K(G) \), we have \( c_1\cap c_2\neq \emptyset \). 
Let \( v\in c_1\cap c_2 \) (where \( v\in V(G) \)). 
Clearly, \( \psi(v)\in \psi(c_1) \) and \( \psi(v)\in \psi(c_2) \). 
Since \( \psi(v)\in \psi(c_1)\cap \psi(c_2) \), it follows that \( \psi(c_1)\cap \psi(c_2)\neq \emptyset \). 
Since \( \psi(c_1) \) and \( \psi(c_2) \) are maximal cliques in \( H \) that intersect, \( \psi(c_1)\psi(c_2) \) is an edge in \( K(H) \). 
This proves Claim~\ref{clm:psi star is homomorphism}; that is, \( \psi^* \) is a homomorphism from \( K(G) \) to \( K(H) \). 

\begin{myclaim}\label{clm:psi star is LBH}
    \( \psi^* \) is an LBH from \( K(G) \) to \( K(H) \). 
\end{myclaim}
Consider an arbitrary vertex \( K \) of \( K(G) \) (i.e., \( K \) is a maximal clique in \( G \)). 
Since \( K \) is arbitrary, to prove Claim~\ref{clm:psi star is LBH}, it suffices to show that the restriction of \( \psi^* \) to \( N_{K(G)}(K) \) is a bijection from \( N_{K(G)}(K) \) onto \( N_{K(H)}(\psi^*(K)) \) (i.e, \( N_{K(H)}(\psi(K)) \)). 
Neighbours of \( K \) in \( K(G) \) are maximal cliques of \( G \) that intersect with \( K \). 
Each such clique \( \Tilde{K} \) of \( G \) satisfies \( \psi(K)\cap \psi(\Tilde{K})\neq \emptyset \) since \( K\cap \Tilde{K}\neq \emptyset \) and \( \psi(v)\in \psi(K)\cap \psi(\Tilde{K}) \) for all \( v\in K\cap \Tilde{K} \). 
Hence, for each neighbour \( \Tilde{K} \) of \( K \) in \( K(G) \), its image \( \psi^*(\Tilde{K}) \) (=\( \psi(\Tilde{K}) \)) is a neighbour of \( \psi(K) \) in \( K(H) \). 
That is, \( \psi^* \) maps each member of \( N_{K(G)}(K) \) to a member of \( N_{K(H)}(\psi(K)) \). 
Hence to prove Claim~\ref{clm:psi star is LBH}, it suffices to show that \( \psi^* \) maps distinct members of \( N_{K(G)}(K) \) to distinct members of \( N_{K(H)}(\psi(K)) \). 
On the contrary, assume that \( \psi^* \) maps distinct members \( c_1 \) and \( c_2 \) of \( N_{K(G)}(K) \) to the same member of \( N_{K(H)}(\psi(K)) \); 
that is, \( c_1,c_2\in N_{K(G)}(K) \) and \( \psi^*(c_1)=\psi^*(c_2) \) (i.e., \( \psi(c_1)=\psi(c_2) \)). 
Let \( x\in c_1\cap K \). 

~\\
\emph{Case~1:} \( x\in c_2 \).

\noindent Since \( c_1 \) and \( c_2 \) are distinct \( q \)-cliques in \( G \), there exists a vertex \( y\in c_1\setminus c_2 \). 
Note that \( \psi(u)\neq \psi(v) \) for each \( u\in c_1\setminus \{x\} \) and \( v\in c_2\setminus \{x\} \) because \( \psi \) restricted to \( N_G(x) \) is a bijection from \( N_G(x) \) onto \( N_H(\psi(x)) \) (and \( u,v\in N_G(x) \)). 
Hence, \( \psi(y)\notin \psi(c_2) \), a contradiction to \( \psi(c_1)=\psi(c_2) \). 

~\\
\emph{Case~2:} \( x\notin c_2 \). 

\noindent Since \( \psi(c_1)=\psi(c_2) \), there exists a vertex \( u\in c_2 \) such that \( \psi(x)=\psi(u) \). 
Since \( c_2 \) and \( K \) intersect, there exists a vertex \( y\in K\cap c_2 \). 
Since \( \psi \) restricted to \( N_G(y) \) is a bijection from \( N_G(y) \) onto \( N_H(\psi(y)) \), the LBH \( \psi \) maps the distinct neighbours \( x \) and \( u \) of \( y \) to distinct neighbours of \( \psi(y) \). 
That is, \( \psi(x)\neq \psi(u) \); a contradiction. 

Since we have a contradiction in both cases, \( \psi^* \) maps distinct members of \( N_{K(G)}(K) \) to distinct members of \( N_{K(H)}(\psi(K)) \). 
This proves Claim~\ref{clm:psi star is LBH}. 
\end{proof}

%

Thanks to the above lemmas, we have the following theorem.

\begin{theorem}\label{thm:lbh goes with line graph operation}
Let \( G \) be a 3-regular graph, and let \( H \) be a triangle-free 3-regular graph. 
Then, there exists an LBH from \( G \) to \( H \) if and only if there exists an LBH from \( L(G) \) to \( L(H) \). 
\end{theorem}
\begin{proof}
If there is an LBH from \( G \) to \( H \), then there is an LBH from \( L(G) \) to \( L(H) \) by Lemma~\ref{lem:lbh leads to lbh btwn line graphs}. 

Conversely, suppose that there is an LBH from \( L(G) \) to \( L(H) \). 
Since \( H \) is a triangle-free 3-regular graph, \( L(H) \) is \( K_4 \)-free and \( K(H)=L(H) \). 
Since \( H \) is a triangle-free 3-regular graph, \( H \) is a locally-\( 3K_1 \) graph. 
Hence, \( K(H) \) (\( =L(H) \)) is a locally-\( 2K_2 \) graph~\cite[Theorem~1.4]{devillers}. 
Thus, \( L(H) \) is a \( K_4 \)-free locally-\( 2K_2 \) graph. 
Hence, every maximal clique in \( L(H) \) is a 3-clique. 
Since \( L(H) \) is \( K_4 \)-free and \( L(G) \) admits an LBH to \( L(H) \), by Observation~\ref{obs:lbh preseves diameter 2 subgraphs}, \( L(G) \) is \( K_4 \)-free. 
Since \( G \) is 3-regular, every edge of \( L(G) \) lies in a triangle. 
Also, there are no isolated vertices in \( G \). 
Hence, every maximal clique in \( L(G) \) is a 3-clique. 
Thus, by Lemma~\ref{lem:lbh sometimes leads to lbh between clique graphs}, there exists an LBH from \( K(L(G)) \) to \( K(L(H)) \). 

\setcounter{myclaim}{0}
\begin{myclaim}\label{clm:G is traingle-free}
    \( G \) is triangle-free. 
\end{myclaim}

\noindent On the contrary, assume that \( G \) contains a triangle. 
Then, \( L(G) \) contains diamond (as subgraph). 
Hence, \( L(H) \) contains diamond by Observation~\ref{obs:lbh preseves diameter 2 subgraphs}. 
But, \( L(H) \) does not contain diamond since it is locally-\( 2K_2 \). 
This contradiction proves Claim~\ref{clm:G is traingle-free}. 
\todo{
Reviewer G:\\
Again, use Obs 2 directly.
}
\todo[color=gray!10!white]{
For this direction, we need to show that \( G \) covers \( H \) if \( L(G) \) covers \( L(H) \). 
Since we have an LBH from \( L(G) \) to \( L(H) \) (rather than from \( G \) to \( H \)), we cannot use Obs 2 directly to argue that \( G \) is triangle-free. 
}

Since \( G \) and \( H \) are triangle-free 3-regular graphs, \( K(G)=L(G) \) and \( K(H)=L(H) \). 
Due to the same reason, \( G \) and \( H \) are locally-\( 3K_1 \). 
Hence, \( K(L(G))=K(K(G))\cong G \) and \( K(L(H))=K(K(H))\cong H \)\cite[Theorem~1.4]{devillers}. 
Moreover, there is an LBH from \( K(L(G)) \) to \( K(L(H)) \) by Lemma~\ref{lem:lbh sometimes leads to lbh between clique graphs}. 
Therefore, there is an LBH from \( G \) to \( H \). 
\end{proof}

\subsection{Homomorphisms That Carry Back Star Colouring}\label{sec:homs that keep star colouring}
It is well-known that homomorphisms carry \( k \)-colourability of target graphs backwards in the following sense: if there is a homomorphism \( \psi \) from a graph \( G \) to a \( k \)-colourable graph~\( H \), then \( G \) is \( k \)-colourable as well~\cite{hell_nesetril2004}. 
In this subsection, we show that out-neighbourhood injective homomorphisms carry \( k \)-star colourability of target graphs backwards. 
We prove that if \( (\vec{H},h) \) is a \( q \)-coloured in-orientation (resp.\ \( q \)-coloured MINI-orientation) of \( H \) and \( \psi \) is an out-neighbourhood injective homomorphism from \( \vec{G} \) to \( \vec{H} \), then \( (\vec{G},h\circ \psi) \) is a \( q \)-coloured in-orientation (resp.\ \( q \)-coloured MINI-orientation) of \( G \). 
\todo{
Reviewer G:\\
For better flow include the proofs in this section, after the statements 
\colcol{(done)}
}
\begin{theorem}\label{thm:out-nbd I Hom preserve in-orientation}
Let \( (\vec{H},h) \) be a \( q \)-coloured in-orientation of a graph \( H \), where \( q\in \mathbb{N} \). 
Let \( \vec{G} \) be an orientation of a graph \( G \), and let \( \psi \) be is an OBH from \( \vec{G} \) to \( \vec{H} \). 
Then, \( (\vec{G},h\circ \psi) \) is a \( q \)-coloured in-orientation of \( G \). 
\end{theorem}
\begin{proof}
Recall that an out-neighbourhood injective homomorphism from \( \vec{G} \) to \( \vec{H} \) is a homomorphism from \( \vec{G} \) to \( \vec{H} \). 
Hence, for every arc \( (u,v) \) in \( \vec{G} \), we know that \( \big(\psi(u),\psi(v)\big) \) is an arc in \( \vec{H} \), and thus \( h(\psi(u))\neq h(\psi(v)) \). 
Thus, \( h\circ \psi \) is a \( q \)-colouring of \( \vec{G} \). 
We need to prove that \( (\vec{G},h\circ \psi) \) is a \( q \)-coloured in-orientation of \( G \). 
That is, we need to prove the following:\\
(i)~for each vertex \( v \) of \( \vec{G} \) with an in-neighbour \( w \) and an out-neighbour \( x \), we have \( h(\psi(w))\neq h(\psi(x)) \) (i.e., each in-neighbour and each out-neighbour of \( v \) have different colours under \( h\circ \psi \)); and\\
(ii)~for each vertex \( v \) of \( \vec{G} \) with two out-neighbours \( x_1 \) and \( x_2 \), we have  \( h(\psi(x_1))\neq h(\psi(x_2)) \) (i.e., no two out-neighbours of \( v \) have the same colour under \( h\circ \psi \)).\\

To prove (i), assume that \( v \) is a vertex in \( \vec{G} \) with an in-neighbour \( w \) and an out-neighbour \( x \). 
That is, \( (w,v) \) and \( (v,x) \) are arcs in \( \vec{G} \). 
Since \( \psi \) is a homomorphism from \( \vec{G} \) to \( \vec{H} \), it follows that \( (\psi(w),\psi(v)) \) and \( (\psi(v),\psi(x)) \) are arcs in \( \vec{H} \). That is, \( \psi(v) \) is a vertex in \( \vec{H} \) with \( \psi(w) \) as an in-neighbour and \( \psi(x) \) as an out-neighbour. 
Hence, \( h(\psi(w))\neq h(\psi(x)) \) since \( (\vec{H},h) \) is a \( q \)-coloured in-orientation. 
This proves~(i).

To prove (ii), assume that \( v \) is a vertex in \( \vec{G} \) with two out-neighbours \( x_1 \) and \( x_2 \). 
Since \( \psi \) is a homomorphism from \( \vec{G} \) to \( \vec{H} \), vertex \( \psi(x_i) \) is an out-neighbour of \( \psi(v) \) for \( i\in \{1,2\} \). 
By definition, \( v \) is a copy of \( \psi(v) \) in \( \vec{G} \), and \( x_i \) is a copy of \( \psi(x_i) \) in \( \vec{G} \) for \( i\in \{1,2\} \). 
Since \( \psi \) is an out-neighbourhood injective homomorphism, the copy \( v \) of \( \psi(v) \) in \( \vec{G} \) has at most one copy of \( \psi(x_1) \) in \( \vec{G} \) as its out-neighbour (in \( \vec{G} \)). 
Thus, \( x_2 \) is not a copy of \( \psi(x_1) \) in \( \vec{G} \). 
That is, \( \psi(x_2)\neq \psi(x_1) \). 
Hence, \( \psi(x_1) \) and \( \psi(x_2) \) are distinct out-neighbours of \( \psi(v) \) in \( \vec{H} \). 
since \( (\vec{H},h) \) is a \( q \)-coloured in-orientation, \( h(\psi(x_1))\neq h(\psi(x_2)) \). 
This proves~(ii). 
\end{proof}

\begin{theorem}\label{thm:out-nbd I Hom preserve star colouring}
Let \( \psi \) be a locally injective homomorphism from a graph \( G \) to a graph \( H \), and let \( h \) be a \( q \)-star colouring of \( H \). 
Then, \( h\circ \psi \) is a \( q \)-star colouring of \( G \). 
\end{theorem}
\begin{proof}
Let \( \vec{H} \) be an in-orientation of \( H \) induced by the \( q \)-star colouring \( h \). 
Since \( \psi \) is a locally injective homomorphism from \( G \) to \( H \), there exists an orientation \( \vec{G} \) of \( G \) such that \( \psi \) is an out-neighbourhood injective homomorphism from \( \vec{G} \) to \( \vec{H} \). 
Since \( (\vec{H},h) \) is a \( q \)-coloured in-orientation of \( H \), it follows from Theorem~\ref{thm:out-nbd I Hom preserve in-orientation} that \( (\vec{G},h\circ \psi) \) is a \( q \)-coloured in-orientation of \( G \). 
By Observation~\ref{obs:in-orientation underlying colouring is star}, \( h\circ \psi \) is a \( q \)-star colouring of \( G \). 
\end{proof}


\begin{theorem}\label{thm:out-nbd I Hom preserve colourful out-nbd orientation}
Let \( (\vec{H},h) \) be a \( q \)-coloured MINI-orientation of a graph \( H \), where \( q\in \mathbb{N} \). 
Let \( \vec{G} \) be an orientation of a graph \( G \), and let \( \psi \) be an out-neighbourhood injective homomorphism from \( \vec{G} \) to \( \vec{H} \). 
Then, \( (\vec{G},h\circ \psi) \) is a \( q \)-coloured MINI-orientation of \( G \). 
\end{theorem}
\todo{
Reviewer G:\\
``such that \( \psi \) is'' \( \to \) ``and let \( \psi \) be'' \colcol{(done)} 
}
\begin{proof}
Clearly, \( (\vec{H},h) \) be a \( q \)-coloured in-orientation of \( H \). 
Since \( \psi \) is an out-neighbourhood injective homomorphism from \( \vec{G} \) to \( \vec{H} \), it follows from Theorem~\ref{thm:out-nbd I Hom preserve in-orientation} that \( (\vec{G},h\circ \psi) \) is a \( q \)-coloured in-orientation of \( G \). 
To complete the proof of the theorem, it suffices to show that for each vertex \( v \) of \( \vec{G} \) with two in-neighbours \( w_1 \) and \( w_2 \), we have \( h(\psi(w_1))=h(\psi(w_2)) \) (i.e., all in-neighbours of \( v \) have the same colour under \( h\circ \psi \)).

Suppose that \( v \) is a vertex in \( \vec{G} \) with two in-neighbours \( w_1 \) and \( w_2 \). Since \( \psi \) is a homomorphism from \( \vec{G} \) to \( \vec{H} \), it follows that \( \psi(w_i) \) is an in-neighbour of \( \psi(v) \) in \( \vec{H} \) for \( i\in \{1,2\} \). 
Since \( (\vec{H},h) \) is a coloured MINI-orientation of \( H \), all in-neighbours of \( \psi(v) \) have the same colour in \( (\vec{H},h) \). 
Thus, \( h(\psi(w_1))=h(\psi(w_2)) \). 
This completes the proof since \( v, w_1 \) and \( w_2 \) are arbitrary. 
\end{proof}

%
%

\iftoggle{extended}
{\section{A Lower Bound for the Star Chromatic Number}
Xie et al.~\cite{xie} proved that at least four colours are needed to star colour a 3-regular graph. 
It follows from a result of Fertin et al.~\cite{fertin2003} that at least \( \lceil (d+3)/2 \rceil \) colours are needed to star colour a \( d \)-regular graph. 
For \( d\geq 2 \), we improve this lowerbound to \( \lceil (d+4)/2 \rceil \), by generalizing the proof of Xie et al~\cite{xie}.
In the next section, we show that this lowerbound is tight for each \( d \), and characterise the extremal graphs.

\begin{theorem}\label{thm:lb chi_s}
Let \( G \) be a \( d \)-regular graph with \( d\geq 2 \). Then, \( \chi_s(G)\geq \raisebox{1pt}{\big\lceil}\frac{d+4}{2}\raisebox{1pt}{\big\rceil} \).
\end{theorem}
\begin{proof}
We have two cases: (i)~\( d \) is even, and (ii)~\( d \) is odd. 
If \( d \) is even, say \( d=2k \), then at least \( \lceil (d+3)/2 \rceil = (d+4)/2 = k+2 \) colours are needed to acyclic colour \( G \) \cite[Proposition~1]{fertin2003}; hence, at least \( k+2=\lceil (d+4)/2 \rceil \) colours are needed to star colour \( G \). Next, we consider the case when \( d \) is odd, say \( d=2k+1 \). 
To prove that \( \chi_s(G)\geq \lceil (d+4)/2 \rceil=k+3 \), it suffices to show that \( G \) is not \( (k+2) \)-star colourable. 
Assume that \( G \) admits a \( (k+2) \)-star colouring \( f\colon V(G)\to\{0,1,\dots,k+1\} \). 
Recall that \( V_i=f^{-1}(i)=\{v\in V(G)\ :\ f(v)=i \} \) for every colour \( i \).\\[5pt]
\noindent \textbf{Claim~1:} For every bicoloured component \( H \) of \( G \), \( |E(H)|\leq \frac{d}{d+1}|V(H)| \), and equality holds only when \( H \) is isomorphic to \( K_{1,d} \).\\[5pt]
Since \( H\cong K_{1,q} \) where \( 0\leq q\leq d \), we have \( |E(H)|/|V(H)|=q/(q+1)\leq d/(d+1) \) and equaltiy holds only when \( q=d \). This proves Claim~1.\\[5pt]
\noindent \textbf{Claim~2:} \( G \) has a bicoloured component \( H \) not isomorphic to \( K_{1,d} \).\\[5pt]
On the contrary, assume that every bicoloured component of \( G \) is \( K_{1,d} \). 
Consider an arbitrary bicoloured component \( H \) of \( G \). 
Let \( u \) be the centre of the star \( H\cong K_{1,d} \), and let \( v_1,v_2,\dots,v_d \) be the remaining vertices in \( H \). 
Without loss of generality, assume that \( u \) is coloured~0, and each \( v_i \) is coloured~1 for \( 1\leq i\leq d \). 
Let \( N_G(v_1)=\{u,w_1,w_2,\dots,w_{d-1}\} \). 
Clearly, \( f(w_i)\in \{2,3,\dots,k+1\} \) for \( 1\leq i\leq d-1 \) (if \( f(w_i)=0 \), then path \( v_2,u,v_1,w_i \) is a bicoloured \( P_4 \)). 
Hence, at least two of vertices \( w_1,w_2,\dots,w_{d-1} \) should receive the same colour, say colour~2. 
Since \( |\{w_1,\dots,w_{d-1}\}|=d-1 \), vertex \( v_1 \) has \( q \) neighbrours coloured~2 where \( 2\leq q\leq d-1 \).
Thus, the component of \( G[V_1\cup V_2] \) containing vertex \( v_1 \) is \( K_{1,q}\not\cong K_{1,d} \). This contradiction proves Claim~2. 

For every pair of distinct colours \( i \) and \( j \), let \( \mathbb{G}_{ij} \) denote the set of components of~\( G[V_i\cup V_j] \). By Claims~1 and 2, we have
\[
\sum_{\substack{i,j\\ i<j}}\sum_{H\in\mathbb{G}_{ij}}|E(H)|<\sum_{\substack{i,j\\ i<j}}\sum_{H\in\mathbb{G}_{ij}}\frac{d}{d+1}|V(H)|=\frac{d}{d+1}\sum_{\substack{i,j\\ i<j}}\left(\,|V_i|+|V_j|\,\right).
\]
Since the set of bicoloured components of \( G \) forms an (edge) decomposition of \( G \), the sum on the left side is \( m\coloneqq |E(G)| \). 
The sum on the right side is \( (k+1)n \) where \( n\coloneqq |V(G)| \) (because \( |V_i| \) appears exactly \( (k+1) \) times in the sum for each \( i \)). 
Therefore, the above inequality simplifies to \( m<\frac{d}{d+1}(k+1)n \). 
Since \( G \) is \( d \)-regular, \( m=\frac{d}{2}n \) and thus the inequality reduces to \( \frac{d}{2}n<\frac{d}{d+1}(k+1)n \). 
That is, \( \frac{d+1}{2}<k+1 \). 
This is a contradiction because \( d=2k+1 \). 
This completes the proof when \( d=2k+1 \). 
\end{proof}
}{}

\subsection{Out-neighbourhood Bijective Homomorphism}
\todo[color=blue!5!white]{
Newly added subsection. 
}
For some properties of locally bijective homomorphism (LBH), analogous propeties hold for out-neighbourhood bijective homomorphism (OBH). 
But, there are a few notable differences between LBH and OBH. 

An LBH \( \psi \) from a graph \( G \) to a graph \( H \) preserves degrees; that is, \( \deg_G(v)=\deg_H(\psi(v)) \) for every vertex \( v \) of \( G \). 
Similarly, every OBH preserves out-degrees. 
\hypertarget{lnk:obh differ from lbh}{ 
However, an OBH need not preserve degrees. 
For instance, if \( \vec{G} \) admits an OBH \( \varphi \) to \( \vec{H} \) and \( u,v,w \) are vertices in \( \vec{G} \) with \( (u,v)\in E(\vec{G}) \) and \( \varphi(v)=\varphi(w) \), then deleting the arc \( (u,v) \) of \( \vec{G} \) and adding an arc \( (u,w) \) results in an oriented graph \( \vec{J} \) such that \( \varphi \) itself is an OBH from \( \vec{J} \) to \( \vec{H} \). 
Note that when \( \vec{G} \) and \( \vec{H} \) are regular graphs, \( \vec{J} \) will not be a regular graph. 
}

The following property of LBH is apparent. 
If \( uv \) is an edge in \( H \), then \( |\psi^{-1}(u)|=|\psi^{-1}(v)| \). 
In constrast, OBH does not have this property. 
If \( \varphi \) is an OBH from an oriented graph \( \vec{G} \) to an oriented graph \( \vec{H} \), the existence of an arc \( (u,v) \) in \( \vec{H} \) does not establish a relationship between \( |\varphi^{-1}(u)| \) and \( |\varphi^{-1}(v)| \); although each copy of \( u \) has exactly one copy of \( v \) as an out-neighbour, a copy of \( v \) can have 0, 1 or more copies of \( u \) as its in-neighbour. 
For LBH to a connected graph \( H \), it follows from the cardinality relationship that \( |\psi^{-1}(u)|=|\psi^{-1}(v)| \) for all \( u,v\in V(H) \), thereby giving an alternative definition of LBH in terms of the notion of `lifts'. 
For \( q\in\mathbb{N} \), a \emph{\( q \)-lift} of a graph \( H \) is a graph \( G \) with \( q \) copies of \( H \) forming its vertex set and for each edge \( uv \) of \( H \), the subgraph of \( G \) induced by copies of \( u \) and \( v \) in \( G \) forming a matching (where each edge is from a copy of \( u \) to a copy of \( v \)). 
For a connected graph \( H \), graphs \( G \) admitting LBH to \( H \) are precisely \( q \)-lifts of \( H \) for some \( q\in \mathbb{N} \) (clearly, \( q=|V(G)|/|V(H)| \)).

There is no lift-like definition for OBH. 
%
Neverthless, if cardinalities of pre-images of vertices under an OBH are the same, then a construction approach similar to lift is indeed possible. 
For an oriented graph \( \vec{G} \) and four vertices \( u,v,x,y \) of \( \vec{G} \) with exactly two arcs \( (u,v) \) and \( (x,y) \) in \( G[\{u,v,x,y\}] \), let us define an \emph{arc 2-switch operation} of arcs \( (u,v) \) and \( (x,y) \) in \( \vec{G} \) as the graph operation of removing arcs \( (u,v) \) and \( (x,y) \), and adding arcs \( (u,y) \) and \( (x,v) \). 
For the context, it is worthwhile to first look at locally bijective homomorphisms between oriented graphs. 
\begin{observation}\label{obs:single arc 2switch lbh}
Let \( \varphi \) be an LBH from an orientated graph \( \vec{G} \) to an oriented graph \( \vec{H} \). 
Let \( (u,v) \) and \( (x,y) \) be two arcs in \( \vec{G} \) with \( \varphi(u)=\varphi(x) \) and \( \varphi(v)=\varphi(y) \). 
Let \( \vec{J} \) be  the oriented graph obtained from \( \vec{G} \) by doing arc 2-switch operation on the arcs \( (u,v) \) and \( (x,y) \). 
Then, \( \varphi \) itself is an LBH from \( \vec{J} \) to \( \vec{H} \). 
\qed
\end{observation}
\begin{observation}\label{obs:single arc 2switch obh}
Let \( \varphi \) be an OBH from an orientated graph \( \vec{G} \) to an oriented graph \( \vec{H} \). 
Let \( (u,v) \) and \( (x,y) \) be two arcs in \( \vec{G} \) with \( \varphi(v)=\varphi(y) \). 
Let \( \vec{J} \) be  the oriented graph obtained from \( \vec{G} \) by doing arc 2-switch operation on the arcs \( (u,v) \) and \( (x,y) \). 
Then, \( \varphi \) itself is an OBH from \( \vec{J} \) to \( \vec{H} \). 
\qed 
\end{observation}
Note that in Observations~\ref{obs:single arc 2switch lbh} and \ref{obs:single arc 2switch obh}, the reverse of the arc 2-switch operation is also an arc 2-switch operation of the same type. 
\begin{observation}
Let \( \psi \) be an LBH from a orientation \( \vec{G} \) of a graph \( G \) to an orientation \( \vec{H} \) of a graph \( H \) such that \( |\psi^{-1}(u)|=|\psi^{-1}(v)| \) for all \( u,v\in V(H) \) \( ( \)e.g., \( H \) is connected\( ) \). 
Then, \( |V(G)|=q|V(H)| \) for some \( q\in \mathbb{N} \), and \( \vec{G} \) can be constructed by \( (i) \) starting with \( q \vec{H} \) and an LBH \( \varphi \) from \( q\,\vec{H} \) to \( \vec{H} \), and \( (ii) \) repeating arc 2-switch operations on arcs of the form \( (u,v) \) and \( (x,y) \) with \( \varphi(u)=\varphi(x) \) and \( \varphi(v)=\varphi(y) \) an arbitrary number of times. 
\qed
\end{observation}
\begin{corollary}
An oriented graph \( \vec{G} \) admits LBH to an orientation \( \vec{H} \) of a connected graph if and only if \( |V(\vec{G})|=q|V(\vec{H})| \) for some \( q\in \mathbb{N} \) and \( \vec{G} \) can be constructed by \( (i) \) starting with \( q \vec{H} \) and an LBH \( \varphi \) from \( q\,\vec{H} \) to \( \vec{H} \), and \( (ii) \) repeating arc 2-switch operations on arcs of the form \( (u,v) \) and \( (x,y) \) with \( \varphi(u)=\varphi(x) \) and \( \varphi(v)=\varphi(y) \) an arbitrary number of times. 
\qed
\end{corollary}

A similar result indeed holds for OBH, by the same reasoning. 
\begin{lemma}\label{lem:obh same premiage size}
Let \( G \) and \( H \) be two graphs with orientations \( \vec{G} \) and \( \vec{H} \), respectively. 
Let \( \psi \) be an OBH from \( \vec{G} \) to \( \vec{H} \) such that \( |\psi^{-1}(u)|=|\psi^{-1}(v)| \) for all \( u,v\in V(H) \). 
Then, \( |V(G)|=q|V(H)| \) for some \( q\in \mathbb{N} \), and \( \vec{G} \) can be constructed by \( (i) \) starting with \( q \vec{H} \) and an LBH \( \varphi \) from \( q\,\vec{H} \) to \( \vec{H} \), and \( (ii) \) repeating arc 2-switch operations on arcs of the form \( (u,v) \) and \( (x,y) \) with \( \varphi(v)=\varphi(y) \) an arbitrary number of times. 
\qed
\end{lemma}


There are some special cases where an OBH will be an LBH. 
\begin{observation}\label{obs:when obh is lbh simple case1}
Let \( \vec{G} \) be an oriented graph that admits an OBH \( \psi \) to a strongly connected oriented graph \( \vec{H} \). 
If every vertex in \( \vec{H} \) has in-degree exactly 1, then \( \psi \) is an LBH from \( \vec{G} \) to \( \vec{H} \). 
\qed
\end{observation}
\begin{observation}\label{obs:when obh is lbh simple case2}
Let \( \vec{G} \) be an oriented graph that admits an OBH \( \psi \) to a strongly connected oriented graph \( \vec{H} \). 
If \( |V(\vec{G})|=|V(\vec{H})| \), then \( \psi \) is an isomorphism from \( \vec{G} \) to \( \vec{H} \) 
(thus, \( \psi \) is an automorphism). 
\qed
\end{observation}

Next, we show that the cardinalities of pre-images of vertices are the same under a degree-preserving OBH to a strongly connected oriented graph. 
\begin{theorem}\label{thm:strongly connected and degree-preserving}
Let \( \vec{G} \) be an oriented graph that admits a degree-preserving OBH \( \psi \) to a strongly connected oriented graph \( \vec{H} \). 
Then, \( |\psi^{-1}(u)|=|\psi^{-1}(v)| \) for all \( u,v\in V(\vec{H}) \). 
\end{theorem}
We prove this theorem with the help of the following lemma. 
\begin{lemma}
Let \( \vec{G} \) be an oriented graph that admits a degree-preserving OBH \( \psi \) to a strongly connected oriented graph \( \vec{H} \). 
Then, \( \vec{G} \) can be constructed by \( (i) \)~starting with \( q \vec{H} \) and an LBH \( \varphi \) from \( q\,\vec{H} \) to \( \vec{H} \), and \( (ii) \) repeating arc 2-switch operations on arcs of the form \( (u,v) \) and \( (x,y) \) with \( \varphi(v)=\varphi(y) \) an arbitrary number of times. 
\end{lemma}
\begin{proof}
Let \( n=|V(\vec{H})| \). 
By Observation~\ref{obs:single arc 2switch obh}, it is possible to start with \( \vec{G} \) and perform arc 2-switch operation on arcs of the form \( (u,v) \) and \( (x,y) \) with \( \psi(v)=\psi(y) \), and thereby obtain oriented graphs such that \( \psi \) itself is a (degree-preserving) OBH from it to \( \vec{H} \). 
Let \( q \) be the highest integer such that an oriented graph constructed this way contains \( q\vec{H} \). 
It suffices to show that \( q=|V(\vec{G})|/n \) to prove the lemma (because the reverse of such an arc 2-switch operation is an arc 2-switch operation of the same type).
On the contrary, assume that \( q<|V(\vec{G})|/n \). 
Consider an oriented graph \( \vec{D} \) containing \( q\vec{H} \) constructed this way, and let \( \vec{J} \) be the oriented graph obtained from \( \vec{D} \) by removing the subgraph \( q\vec{H} \). 
Clearly, the restriction of \( \psi \) to \( V(\vec{J}) \) is an OBH from \( \vec{J} \) to \( \vec{H} \); let us call it \( \psi_0 \). 
The oriented graph \( \vec{J} \) cannot be empty since \( q<|V(\vec{G})|/n \). 
Since \( \vec{H} \) is strongly connected and \( \psi_0 \) is an OBH from \( \vec{J} \) to \( \vec{H} \), each vertex of \( \vec{H} \) has at least one copy in \( \vec{J} \). 
For each vertex \( z \) of \( \vec{H} \), choose a copy of \( z \) in \( \vec{J} \), and call it~\( z_0 \). 
For each arc \( uv \) of \( \vec{H} \), the vertex \( u_0 \) has a copy \( v' \) of \( v \) as its out-neighbour in \( \vec{J} \). 
If \( v'\neq v_0 \), check whether \( v_0 \) has at least one copy of \( u \) as its in-neighbour in \( \vec{J} \). 
If yes, call one those in-neighbours \( y^* \). 
Otherwise, there exists an in-neighbour \( w \) of \( v \) in \( \vec{H} \) such that at least two copies of \( w \) appear as in-neighbours of \( v_0 \) in \( \vec{J} \) (because \( \psi_0 \) is degree-preserving), and call of one those in-neighbours as \( y^* \). 
Perform arc 2-switch operation on the arcs \( (u_0,v') \) and \( (y^*,v_0) \). 
It is easy to see that after completing such arc 2-switch operations considering every arc of \( \vec{H} \), the set \( \{z_0\colon z\in V(\vec{H})\} \) induces a copy of \( \vec{H} \) in the resultant oriented graph. 
This means that starting with \( \vec{G} \) and performing the sequence of arc 2-switch operations used to construct \( \vec{D} \) and then performing the above sequence of arc 2-switch operations gives an oriented graph that contains \( (q+1)\vec{H} \). 
This contradicts the choice of \( q \), and thus completes the proof.
\end{proof}

\begin{corollary}\label{cor:strongly connected and degree-preserving}
An oriented graph \( \vec{G} \) admits a degree-preserving OBH to a strongly connected orientated graph \( \vec{H} \) if and only if \( |V(\vec{G})|=q|V(\vec{H})| \) for some \( q\in \mathbb{N} \) and \( \vec{G} \) can be constructed by \( (i) \)~starting with \( q \vec{H} \) and an LBH \( \varphi \) from \( q\,\vec{H} \) to \( \vec{H} \), and \( (ii) \) repeating arc 2-switch operations on arcs of the form \( (u,v) \) and \( (x,y) \) with \( \varphi(v)=\varphi(y) \) an arbitrary number of times. 
\qed
\end{corollary}

\section[\( 2p \)-Regular \( (p+2) \)-Star Colourable Graphs]{\boldmath \( 2p \)-Regular \( (p+2) \)-Star Colourable Graphs}\label{sec:2p-regular p+2-star colourable}
We start with a simple, but important observation. 
Refer to Section~\ref{sec:oriented line graph def} for the definitions of \( \vec{L}(K_q) \) and \( L^*(K_q) \).  
In this section, we write \( \mathrm{proj}_2 \) to denote of the 2nd projection map on \( \mathbb{R}^2 \) (i.e., \( (x,y)\mapsto y \)) domain and co-domain restricted suitably. 
%
%
\begin{observation}\label{obs:star col of  Larrow Kq}
For \( q\geq 2 \), \( (\vec{L}(K_q),\mathrm{proj}_2) \) is a \( q \)-coloured MINI-orientation of \( L^*(K_q) \), where \( \mathrm{proj}_2 \) is the 2nd projection map restricted to the domain \( V(\vec{L}(K_q)) \) and the co-domain \( \mathbb{Z}_q \). 
\end{observation}
\begin{proof}
Consider an arbitrary vertex \( (i,j) \) of \( \vec{L}(K_q) \). 
Clearly, \( \{(k,i)\in \mathbb{Z}_q^2\colon k\neq i, k\neq j\} \) is the set of in-neighbours of \( (i,j) \) and \( \{(j,\ell)\in \mathbb{Z}_q^2\colon \ell\neq i, \ell\neq j\} \) is the set of out-neighbours of \( (i,j) \) in \( \vec{L}(K_q) \). 
Hence, \( \mathrm{proj}_2 \) assigns colour \( i \) on in-neighbours of \( (i,j) \) and pairwise distinct colours from \( \mathbb{Z}_q\setminus\{i\} \) on \( (i,j) \) and its out-neigbours in \( \vec{L}(K_q) \). 
This completes the proof since all three conditions in the definition of coloured MINI-orientation are satisfied, and \( (i,j) \) is arbitrary. 
\end{proof}
\todo{
Reviewer G:\\
Remind the reader where these (\( G_{2p} \) and \( \overrightarrow{G_{2p}} \)) were defined.
}
\todo[color=gray!10!white]{
The notation is changed to \( L^*(K_{p+2}) \) and \( \vec{L}(K_{p+2}) \); the reader is referred to the definitions as suggested.
}
\todo{
(omitted) [rewritten] 

\textbf{Observation.}
\emph{
For \( p\geq 2 \), the orientation \( \overrightarrow{G_{2p}} \) of \( G_{2p} \) is a \( (p+2) \)-CEO of \( G_{2p} \). 
}

\begin{proof}
Theorem~4 in \cite{shalu_cyriac3} showed that the vertex set of \( G_{2p} \) can be partitioned into sets \( V_i^j\coloneqq \{(i,j)\} \) satisfying certain conditions, and as a result, the function \( f\colon V(G_{2p})\to \{0,1,\dots,p+1\} \) defined as \( f(v)=i \) for all \( v\in V_i^j \) and \( i,j\in \{0,1,\dots,p+1\} \) is a \( (p+2) \)-star colouring of \( G_{2p} \) by \cite[Theorem~3]{shalu_cyriac3}. 
Explicitly written, \( f((i,j))=i \) for all \( (i,j)\in V(G_{2p}) \). 
Hence, each vertex \( (i,j) \) of \( G_{2p} \) has exactly \( p \) neighbours coloured \( j \) under \( f \) (namely, vertices \( (j,k) \) where \( k\in \{0,1,\dots,p+1\}\setminus \{i,j\} \)) and exactly one neighbour coloured~\( k \) (namely, \( (k,i) \)) for each \( k\in \{0,1,\dots,p+1\}\setminus \{i,j\} \).

Since \( f \) is a \( (p+2) \)-star colouring of \( G_{2p} \), every bicoloured component of \( G_{2p} \) under \( f \) is \( K_{1,p} \) (by \cite[Theorem~2]{shalu_cyriac3}). 
Hence, the in-orientation of \( G_{2p} \) induced by \( f \) orients edges in each bicoloured component \( H\cong K_{1,p} \) of \( G \) towards the centre of the star \( H \). 
By the proof of Theorem~8 in \cite{shalu_cyriac3}, the in-orientation of \( G_{2p} \) induced by \( f \) is a \( (p+2) \)-CEO of \( G_{2p} \). 
This \( (p+2) \)-CEO of \( G_{2p} \) must orient each edge \( \{(j,k), (i,j)\} \) of \( G_{2p} \) as \( (j,k)\to (i,j) \) because the subgraph of \( G_{2p} \) induced by \( \{(i,j)\}\cup \{(j,\ell)\colon 0\leq \ell \leq p+1, \ell \neq i, \ell \neq j\} \) is a bicoloured component \( H \) of \( G \) under \( f \), and \( (i,j) \) is the centre of the star \( H\cong K_{1,p} \). 
In other words, this \( (p+2) \)-CEO of \( G_{2p} \) orients each edge \( \{(i,j),(k,i)\} \) of \( G_{2p} \) as \( (i,j)\to (k,i) \), which is precisely the definition of \( \overrightarrow{G_{2p}} \). 
Hence, \( \overrightarrow{G_{2p}} \) is a \( (p+2) \)-CEO of~\( G_{2p} \).
\end{proof}
}

\begin{theorem}\label{thm:characterise 2p-regular p+2-star colourable}
For \( p\geq 2 \), the following are equivalent for every \( 2p \)-regular graph \( G \). 
\begin{enumerate}[label={\Roman*.},itemsep=0pt,topsep=3pt]
\item \( G \) admits a \( (p+2) \)-star colouring. 
\item \( G \) admits a \( (p+2) \)-coloured MINI-orientation. 
\item \( G \) has an orientation that admits an OBH to \( \vec{L}(K_{p+2}) \). 
\item \( G \) admits a \( (p+2) \)-colouring \( f \) such that every bicoloured component of \( (G,f) \) is \( K_{1,p} \). 
\end{enumerate}
\end{theorem}

We prove the theorem with the help of the following lemma. 

\begin{lemma}\label{lem:star col iff mini iff obh}
Let \( (\vec{G},f) \) be a \( (p+2) \)-coloured orientation of a \( 2p \)-regular graph \( G \), where \( p\geq 2 \). 
Then, \( f \) is a star coloring of \( G \) if and only if every bicoloured component of \( (G,f) \) is \( K_{1,p} \). 
Moreover, the following are equivalent, where \( \mathrm{proj}_2 \) is the 2nd projection map restricted to the domain \( V(\vec{L}(K_{p+2})) \) and the co-domain \( \mathbb{Z}_{p+2} \). 
\begin{enumerate}[label={\Roman*.},itemsep=0pt,topsep=3pt]
\item \( f \) is a star colouring of \( G \), and \( \vec{G} \) is the in-orientation of \( G \) induced by \( f \). 
\item \( (\vec{G},f) \) is a \( (p+2) \)-coloured MINI-orientation of \( G \). 
\item \( \vec{G} \) admits an OBH \( \psi \) to \( \vec{L}(K_{p+2}) \) such that \( \mathrm{proj}_2\circ\psi=f \). 
\end{enumerate}
\end{lemma}
\begin{proof}
If every bicoloured component of \( (G,f) \) is \( K_{1,p} \), then \( f \) is evidently a star coloring. 

Suppose that \( f \) is a star colouring of \(  G \).
Let \( \vec{G} \) be an in-orientation of \( G \) induced by \( f \). 
Obviously, \( (\vec{G},f) \) is a \( (p+2) \)-coloured in-orientation. 
To prove that it is a \( (p+2) \)-coloured MINI-orientation, it suffices to show that for each vertex of \( (\vec{G},f) \), all in-neighbours have the same colour. 
Let \( v \) be a vertex in \( \vec{G} \), and let \( q \) denote the out-degree of \( v \) in \( \vec{G} \). 
Clearly, \( v \) has at least one in-neighbour since all \( 2p \) neighbours of \( v \) cannot be out-neighbours (else, \( v \) and its out-neighours need \( 2p+1 \) colours, but we have only \( p+2 \)). 
In \( (\vec{G},f) \), the number of colours in the closed neighbourhood of \( v \) is at least \( q+2 \) (i.e., \( q \) colors for out-neighbours of \( v \), an extra colour for \( v \), and an extra colour for some in-neighbour of \( v \)), and equality holds if and only if all in-neighbours of \( v \) have the same colour. 
Since at most \( p+2 \) colours are available, \( q\leq p \). 
Hence, the out-degree of \( v \) is at most \( p \), and therefore the in-degree of \( v \) is at least \( p \). 
Since the sum of out-degrees of all vertices equals the sum of in-degrees, every vertex of \( \vec{G} \) has out-degree and in-degree exactly \( p \) (i.e., \( \vec{G} \) is an Eulerian orientation), and in particular, \( q=p \). 
This proves that every bicoloured component of \( (G,f) \) is \( K_{1,p} \). 
Hence, \( \vec{G} \) is the unique in-orientation of \( G \) induced by \( f \). 
Since all \( q+2 \) colours are needed to colour the closed neighbourhood of \( v \), all in-neighbours of \( v \) have the same colour. 
Since \( v \) is arbitrary, \( (\vec{G},f) \) is a \( (p+2) \)-coloured MINI-orientation. 
This proves \( I\implies II \). 
It is also established that (i)~\( f \) is a star coloring of \( G \) if and only if every bicoloured component of \( (G,f) \) is \( K_{1,p} \), and (ii)~if \( f \) is a star colouring of \( G \) and \( \vec{G} \) is the in-orientation of \( G \) induced by \( f \), then \( \vec{G} \) is an Eulerian orientation. 

To prove \( II \implies I \), assume \( II \). 
Since \( (\vec{G},f) \) is a \( (p+2) \)-coloured MINI-orientation of \( G \), it is a \( (p+2) \)-coloured in-orientation of \( G \). 
By Observation~\ref{obs:in-orientation underlying colouring is star}, \( f \) is a star colouring of \( G \) and \( \vec{G} \) is an in-orientation of \( G \) induced by \( f \). 
Since \( f \) is a star coloring, every bicoloured component of \( (G,f) \) is \( K_{1,p} \). 
Hence, \( \vec{G} \) is the unique in-orientation of \( G \) induced by \( f \). 

Finally, we prove \( II\iff III \). 
To prove \( II\implies III \), assume \( II \). 
Each vertex \( v \) of \( \vec{G} \) has at least one in-neighbour (in fact, exactly \( p \) in-neighbours since every bicoloured component of \( (\vec{G},f) \) is \( K_{1,p} \)). 
Define a function \( h\colon V(\vec{G})\to\mathbb{Z}_{p+2} \) that assigns the colour of in-neighbours of \( v \) on each vertex \( v \) of \( \vec{G} \). 
Define a function \( \psi\colon V(\vec{G})\to V(\vec{L}(K_{p+2}))  \) as \( \psi(v)=(h(v),f(v)) \) for all \( v\in V(\vec{G}) \). 
By definition of \( \psi \), we have \( \mathrm{proj}_2\circ \psi=f \). 

To prove that \( \psi \) is a homomorphism from \( \vec{G} \) to \( \vec{L}(K_{p+2}) \), consider an arbitrary arc \( (u,v) \) of \( \vec{G} \). 
Let \( \psi(u)=(i,j) \) and \( \psi(v)=(k,\ell) \), where \( i,j,k,\ell \in \mathbb{Z}_{p+2} \). 
Let \( w \) be an in-neighbour of \( u \) in \( \vec{G} \). 
We know that \( u \) is an in-neighbour of \( v \) in \( \vec{G} \). 
Note that \( f(w)\neq f(v) \) since \( (\vec{G},f) \) is a colourful MINI-orientation. 
Also, \( h(v)=f(u) \) and \( h(u)=f(w) \) by the definition of \( h \).  
That is, \( k=j \), and \( i=h(u)=f(w)\neq f(v)=\ell \). 
Thus, \( k=j \) and \( i\neq \ell \), there is an arc in \( \vec{L}(K_{p+2}) \) from \( (i,j) \) to \( (k,\ell) \). 
Since there is an arc from \( \psi(u) \) to \( \psi(v) \) in \( \vec{L}(K_{p+2}) \) for an arbitrary arc \( (u,v) \) of \( \vec{G} \), indeed \( \psi \) is a homomorphism from \( \vec{G} \) to \( \vec{L}(K_{p+2}) \). 

We know that \( \vec{G} \) and \( \vec{L}(K_{p+2}) \) are both Eulerian orientations. 
Hence, every vertex in \( \vec{G} \) and \( \vec{L}(K_{p+2}) \) has exactly \( p \) out-neighbours. 
To prove that the homomorphism \( \psi \) from \( \vec{G} \) to \( \vec{L}(K_{p+2}) \) is an OBH, it suffices to show that for each vertex \( v \) of \( \vec{G} \), no two out-neighbours of \( v \) have the same image under \( \psi \). 
Let \( v \) be a vertex in \( \vec{G} \) with distinct out-neighbours \( w_1 \) and \( w_2 \). 
Then, \( f(w_1)\neq f(w_2) \), and hence \( \psi(w_1)\neq \psi(w_2) \). 
This proves that \( \psi \) is an OBH from \( \vec{G} \) to \( \vec{L}(K_{p+2}) \) with \( \textrm{proj}_2\circ \psi=f \) since \( v, w_1 \) and \( w_2 \) are aribtrary. 

To prove \( III\implies II \), assume \( III \). 
By Observation~\ref{obs:star col of  Larrow Kq},  \( (\vec{L}(K_{p+2}),\mathrm{proj}_2) \) is a \( (p+2) \)-coloured MINI-orientation of \( L^*(K_{p+2}) \). 
Hence, by Theorem~\ref{thm:out-nbd I Hom preserve colourful out-nbd orientation}, \( (\vec{G},\mathrm{proj}_2\circ \psi) \) is a \( (p+2) \)-coloured MINI-orientation of \( G \); and it is the same as \( (\vec{G},f) \). 
This proves \( II \). 
\end{proof}

\todo[color=white]{
(omitted) [rewritten]

Theorem~2 in \cite{shalu_cyriac3} proved that \( I\implies II \). 
Since \( K_{1,p} \) is a star, we have \( II\implies I \). 
Theorem~8 in \cite{shalu_cyriac3} proved that \( I\iff III \). 
To complete the proof of Theorem~\ref{thm:characterise 2p-regular p+2-star colourable}, it suffices to show that \( III\iff IV \). 

\textbf{Theorem (\( III\iff IV \)).}
\emph{
Let \( p\geq 2 \), and let \( G \) be a \( 2p \)-regular graph with an orientation \( \vec{G} \). 
Then, \( \vec{G} \) is a \( (p+2) \)-CEO of \( G \) if and only if \( \vec{G} \) admits an out-neighbourhood bijective homomorphism to \( \overrightarrow{G_{2p}} \).
}

\emph{Proof.} 
Suppose that \( \vec{G} \) is a \( (p+2) \)-CEO of \( G \) with an underlying \( (p+2) \)-colouring \( f\colon V(G)\to \{0,1,\dots,p+1\} \). 
We know that \( f \) is a \( (p+2) \)-star colouring of \( G \). 
We also know that \( \vec{G} \) is the in-orientation of \( G \) induced by \( f \). 
Let \( V_i=f^{-1}(i) \) for \( 0\leq i\leq p+1 \). 
For distinct \( i,j\in \{0,1,\dots,p+1\} \), let \( V_i^j \) denote the set of vertices \( x\in V_i \) such that \( x \) has exactly \( p \) neighbours in \( V_j \). 
For each \( q\in \mathbb{N} \), let \( \mathbb{Z}_q \) denote the set \( \{0,1,\dots,q-1\} \). 
By (the proof of) Theorem~3 in \cite{shalu_cyriac3}, (i)~\( \{V_i^j\colon i,j\in \mathbb{Z}_{p+2}, i\neq j\} \) is a partition of \( V(G) \), and (ii)~for distinct \( i,j\in \mathbb{Z}_{p+2} \), each vertex in \( V_{i}^{j} \) has exactly \( p \) neighbours in \( \bigcup_{k\notin\{i,j\}} V_{j}^{k} \) and exactly one neighbour in \( V_{k}^{i} \) for each \( k\notin\{i,j\} \) (where \( k\notin \{i,j\} \) is an abbreviation for \( k\in \mathbb{Z}_{p+2}\setminus \{i,j\} \)). 
\todo{
Reviewer F:\\
I think that the proof of (i) and (ii) should be included, at least in the appendix
}
\begin{tcolorbox}[colorback=gray!10!white,boxrule=0.5pt,boxsep=0pt,colframe=black]
The proof of the theorem is re-written, and it is combined with a lemma to keep it short and self-contained. 
\end{tcolorbox}

Fix distinct indices \( i,j\in \mathbb{Z}_{p+2} \) and a vertex \( v\in V_i^j \). 
By point~(ii) above, \( v \) has exactly \( p \) neighbours in \( \bigcup_{k\notin\{i,j\}} V_{j}^{k} \) and exactly one neighbour in \( V_{k}^{i} \) for each \( k\notin\{i,j\} \). 
Since \( \bigcup_{k\notin\{i,j\}} V_{j}^{k}\subseteq V_j=f^{-1}(j) \) and \( v \) has no neighbour in \( V_j^i=V_j\setminus \bigcup_{k\notin\{i,j\}} V_{j}^{k} \), vertex \( v \) together with its \( p \) neighbours in \( \bigcup_{k\notin\{i,j\}} V_{j}^{k} \) induce a bicoloured component \( H \) of \( G \). 
Since \( \vec{G} \) is the in-orientation of \( G \) induced by \( f \), the edges in \( H \) are oriented towards its centre \( v \) in \( \vec{G} \). 
That is, the \( p \) neighbours of \( v \) in \( \bigcup_{k\notin\{i,j\}} V_{j}^{k} \) are in-neighbours of \( v \) in \( \vec{G} \). 
On the other hand, for each \( k\notin \{i,j\} \) and each neighbour \( x \) of \( v \) in \( V_k^i \), the bicoloured component \( H \) of \( G \) containing edge \( vx \) has \( x \) as its centre, and thus \( x \) is an out-neighbour of \( v \) in \( \vec{G} \) (if \( v \) is the centre of \( H \), then \( v \) has \( p \) neighbours in \( V_k=\bigcup_{\substack{0\leq \ell\leq p+1\\ \ell\neq k}} V_k^\ell \), a contradiction). 
}

\todo{
(omitted) [continues]

\begin{proof}[Proof (cont.).]
Since \( i,j \), and \( v \) are arbitrary, we have the following. 
\setcounter{myclaim}{0}
\begin{myclaim}\label{clm:arcs between Vijs}
For distinct \( i,j\in \mathbb{Z}_{p+2} \), each vertex in \( V_i^j \) has exactly \( p \) in-neighbours in \( \bigcup_{k\notin\{i,j\}} V_{j}^{k} \) and exactly one out-neighbour in \( V_k^i \) for each \( k\notin \{i,j\} \). 
\end{myclaim}
Define \( \psi \colon V(\vec{G})\to V(\overrightarrow{G_{2p}}) \) as \( \psi(v)=(i,j) \) for all \( v\in V_i^j \), where \( i,j\in \mathbb{Z}_{p+2} \) and \( i\neq j \). 
Note that for distinct \( i,j\in \mathbb{Z}_{p+2} \), the set of copies of \( (i,j) \) in \( \vec{G} \) is precisely \( V_i^j \) (i.e., \( V_i^j=\psi^{-1}((i,j)) \)). 
To prove that \( \psi \) is a homomorphism from \( \vec{G} \) to \( \overrightarrow{G_{2p}} \), consider an arbitrary arc \( (u,v) \) in \( \vec{G} \). 
Due to Claim~\ref{clm:arcs between Vijs}, there exist pairwise distinct indices \( i,j,k\in \mathbb{Z}_{p+2} \) such that \( u\in V_i^j \) and \( v\in V_k^i \). 
As a result, \( \psi(u)=(i,j) \) and \( \psi(v)=(k,i) \). 
Hence, \( \big( (i,j), (k,i) \big) = \big( \psi(u),\psi(v) \big) \) is an arc in \( \overrightarrow{G_{2p}} \). 
Since \( (u,v) \) is an arbitrary arc in \( \vec{G} \), the mapping \( \psi \) is a homomorphism from \( \vec{G} \) to \( \overrightarrow{G_{2p}} \). 

To prove that \( \psi \) is an out-neighbourhood bijective homomorphism from \( \vec{G} \) to \( \overrightarrow{G_{2p}} \), it suffices to show that for every vertex \( (i,j) \) of \( \overrightarrow{G_{2p}} \) and every out-neighbour \( (k,i) \) of \( (i,j) \) in \( \overrightarrow{G_{2p}} \), each copy of \( (i,j) \) in \( \vec{G} \) has exactly one copy of \( (k,i) \) in \( \vec{G} \) as its out-neighbour. 
Consider a vertex \( (i,j) \) in \( \overrightarrow{G_{2p}} \), an out-neighbour \( (k,i) \) of \( (i,j) \) in \( \overrightarrow{G_{2p}} \), and a copy \( v \) of \( (i,j) \) in \( \vec{G} \). 
Since \( v \) is a copy of \( (i,j) \) in \( \vec{G} \), we have \( \psi(v)=(i,j) \), and thus \( v\in V_i^j \). 
Since \( V_k^i \) is precisely the set of copies of \( (k,i) \) in \( \vec{G} \), it follows from Claim~\ref{clm:arcs between Vijs} that \( v \) has exactly one copy of \( (k,i) \) in \( \vec{G} \) as its out-neighbour. 
Since \( (i,j) \), \( (k,i) \) and \( v \) are arbitrary, \( \psi \) is an out-neighbourhood bijective homomorphism from \( \vec{G} \) to \( \overrightarrow{G_{2p}} \).

Conversely, suppose that there exists an out-neighbourhood bijective homomorphism \( \psi \) from \( \vec{G} \) to \( \overrightarrow{G_{2p}} \). 
Since \( G \) and \( G_{2p} \) are \( 2p \)-regular, each out-neighbourhood bijective homomorphism from \( \vec{G} \) to \( \overrightarrow{G_{2p}} \) is degree-preserving. 
Since \( \overrightarrow{G_{2p}} \) is an Eulerian orientation and \( \psi \) is a degree-preserving out-neighbourhood bijective homomorphism from \( \vec{G} \) to \( \overrightarrow{G_{2p}} \), the orientation \( \vec{G} \) is Eulerian. 
By Observation~\ref{obs:G2p arrow is p+2-CEO}, \( \overrightarrow{G_{2p}} \) is a \( (p+2) \)-CEO of \( G_{2p} \). 
In particular, \( \overrightarrow{G_{2p}} \) is a \( (p+2) \)-coloured MINI-orientation of \( G_{2p} \) (with an underlying \( (p+2) \)-colouring \( h \)). 
We also know that \( \psi \) is an out-neighbourhood injective homomorphism from \( \vec{G} \) to \( \overrightarrow{G_{2p}} \). 
Therefore, by Theorem~\ref{thm:out-nbd I Hom preserve colourful out-nbd orientation}, \( \vec{G} \) is a \( (p+2) \)-coloured MINI-orientation of \( G \) (with \( h\circ \psi \) as its underlying \( (p+2) \)-colouring). 
Since \( \vec{G} \) is an Eulerian orientation as well as a \( (p+2) \)-coloured MINI-orientation, \( \vec{G} \) is a \( (p+2) \)-CEO of~\( G \). 
\end{proof}
}

The next theorem follows from Theorem~\ref{thm:strongly connected and degree-preserving} since \( \vec{L}(K_{p+2}) \) is strongly connected (for \( p\geq 2 \)). 
\begin{theorem}
Let \( G \) be a \( 2p \)-regular graph with \( p\geq 2 \). 
Let \( \vec{G} \) be an orientation of \( G \) that admits an OBH \( \psi \) to \( \vec{L}(K_{p+2}) \). 
Then, \( |V(G)|=q(p+1)(p+2) \) for some positive integer \( q \), and \( \vec{G} \) can be constructed by \( (i) \) starting with \( q\,\vec{L}(K_{p+2}) \) and an LBH \( \varphi \) from \( q \vec{L}(K_{p+2}) \) to \( \vec{L}(K_{p+2}) \), and \( (ii) \)~repeating arc 2-switch operations on arcs of the form \( (u,v) \) and \( (x,y) \) with \( \varphi(v)=\varphi(y) \) an arbitrary number of times. 
\qed
\end{theorem}

Before proceeding further, we restate a few results in \cite{shalu_cyriac3} which are closely related to the notions we discuss here, but were expressed with different terminology and notation. 
After this point, we never use those terminology or notation from \cite{shalu_cyriac3}. 
Let \( p\geq 2 \). 
The oriented graph \( \vec{L}(K_{p+2}) \) and the graph \( L^*(K_{p+2}) \) were introduced in \cite{shalu_cyriac3} under the names \( \overrightarrow{G_{2p}} \) and \( G_{2p} \), respectively. 
It is easy observe that \( \vec{L}(K_{p+2}) \) is isomorphic to \( \overrightarrow{G_{2p}} \) (the map \( (x,y)\mapsto (y,x) \) is an isomorphism). 
\begin{theorem}[\cite{shalu_cyriac3}]\label{thm:Lstarkq ham}
\( L^*(K_q) \) is vertex-transitive and edge-transitive for \( q\in \mathbb{N} \) and Hamiltonian for  \( q\geq 4 \). 
\end{theorem}
Due to Observation~\ref{obs:when obh is lbh simple case2}, for each \( p\geq 2 \), \( L^*(K_{p+2}) \) is the unique \( 2p \)-regular \( (p+2) \)-star colourable graph on \( (p+1)(p+2) \) vertices (see \cite{shalu_cyriac3} for an alternate proof).

In Notes 1 and 2 below, We mention some terminology and notation from \cite{shalu_cyriac3} to facilitate comparison with proofs in \cite{shalu_cyriac3}, in case a reader wishes to do so. 

\noindent
\emph{Note 1:} 
Let \( G \) be a \( 2p \)-regular graph with an orientation \( \vec{G} \). 
In \cite{shalu_cyriac3}, \( \vec{G} \) is called a \( (p+2) \)-Colourful Eulerian Orientation (abbr.\ \( (p+2) \)-CEO) if there exists a \( (p+2) \)-colouring \( f \) of \( G \) such that \( (\vec{G},f) \) is a \( (p+2) \)-colourful MINI-orientation, and \( \vec{G} \) is Eulerian. 
From Lemma~\ref{lem:star col iff mini iff obh}, it follows that any orientation of a \( 2p \)-regular graph that admits an OBH to \( \vec{L}(K_{p+2}) \) must be Eulerian.  
Hence, the condition ``\( \vec{G} \) is Eulerian'' in the definition of CEO is redundant for \( (p+2) \)-CEOs of \( 2p \)-regular graphs. 

\noindent
\emph{Note 2:} 
Let \( f \) be a \( (p+2) \)-star colouring of a \( 2p \)-regular graph \( G \). 
In \cite{shalu_cyriac3}, \( V_i^j \)  denotes the set of vertices with colour \( i \) whose in-neighbours are coloured \( j \) by \( f \); with the notation in the proof of Lemma~\ref{lem:star col iff mini iff obh}, \( V_i^j=\{v\in V(G)\colon f(v)=i \mathrm{ and } h(v)=j\}=\psi^{-1}((j,i)) \). 
~\\

Some properties of \( 2p \)-regular \( (p+2) \)-star colourable graphs appear in \cite[Corollary~2]{shalu_cyriac3}, and we explore more properties here. 
\begin{theorem}[\cite{shalu_cyriac3}]
For \( p\geq 2 \), every \( 2p \)-regular \( (p+2) \)-star colourable graph has the following properties: \( (i) \) \( G \) is \( (\mathrm{diamond},K_4) \)-free \( ( \)i.e., \( \mathrm{diamond} \)-subgraph-free\( ) \), \( (ii) \) the independence number of \( G \) is greater than \( |V(G)|/4 \), and \( (iii) \) the chromatic number of \( G \) is \( O(\log p) \). 
\end{theorem}

Dvo{\v{r}}{\'a}k et al.~\cite{dvorak} proved that for every 3-regular graph \( G \), the line graph of \( G \) is 4\nobreakdash-star colourable if and only if \( G \) admits a locally bijective homomorphism to \( Q_3 \). 
In Theorem~\ref{thm:4-star colourable in terms of distance-two colouring} below, we show that for every 3-regular graph \( G \), the line graph of \( G \) is 4\nobreakdash-star colourable if and only if \( G \) is bipartite and distance-two 4-colourable. 
In Theorem~\ref{thm:p+2 star colourable iff LBH} below, we prove that for every integer \( p\geq 2 \), a \( K_{1,p+1} \)-free \( 2p \)-regular graph \( G \) is \( (p+2) \)\nobreakdash-star colourable if and only if \( G \) admits a locally bijective homomorphism to \( G_{2p} \). 
In particular, a claw-free 4-regular graph \( G \) is 4-star colourable if and only if \( G \) admits a locally bijective homomorphism to the line graph of \( Q_3 \) (recall that \( G_4=L(Q_3) \)). 

\begin{theorem}\label{thm:4-star colourable in terms of distance-two colouring}
For every 3-regular graph \( G \), the line graph of \( G \) is 4-star colourable if and only if \( G \) is bipartite and distance-two 4-colourable. 
\end{theorem}
\begin{proof}
Let \( G \) be a 3-regular graph. 
By the characterisation of Dvo{\v{r}}{\'a}k et al.~\cite{dvorak}, the line graph of \( G \) is 4-star colourable if and only if \( G \) admits a locally bijective homomorphism to \( Q_3 \). 
For a fixed graph \( H \), a graph \( G \) admits a locally bijective homomorphism to the bipartite double \( H\times K_2 \)  if and only if \( G \) is bipartite and \( G \) admits a locally bijective homomorphism to \( H \) \cite{fiala2000}. 
\todo{
Reviewer G:\\
Define the product.}
\todo[color=gray!10!white]{
The definition of bipartite double is added in Section~\ref{sec:def}. 
We prefer not to define the tensor product to have fewer definitions. 
}
\todo{
Reviewer G:\\
Either  https://doi.org/10.7151/dmgt.1159 or the PhD. Thesis should be considered as [14] is not reviewed. BTW the same reference shall be used for the next marked "if and only if"
\colcol{(done)}
}
We know that \( Q_3 \) is the bipartite double of \( K_4 \); that is \( Q_3\cong K_4\times K_2 \). 
Hence, the line graph of \( G \) is 4-star colourable if and only if \( G \) is bipartite and \( G \) admits a locally bijective homomorphism to \( K_4 \). 
A 3-regular graph \( G \) admits a locally bijective homomorphism to \( K_4 \) if and only if \( G \) admits a locally injective homomorphism to \( K_4 \) \cite[Theorem 2.4]{fiala2000}. 
\todo{
Reviewer G:\\
add reference, e.g. as in the previous comment \colcol{(done)}
}
Besides, a 3-regular graph \( G \) admits a locally injective homomorphism to \( K_4 \) if and only if \( G \) admits a distance-two 4-colouring \cite{fiala_etal2001,kratochvil_siggers}. 
\todo{
Reviewer G:\\
This should be observed decades earlier, check e.g. Tiziana Calamoneri: The L(h, k)-Labelling Problem: An Updated Survey and Annotated Bibliography
\colcol{(done)}
}
%
Thus, \( G \) admits a locally bijective homomorphism to \( K_4 \) if and only if \( G \) admits a distance-two 4-colouring.
Therefore, the line graph of \( G \) is 4-star colourable if and only if \( G \) is bipartite and distance-two 4-colourable. 
\end{proof}

\todo{
(omitted) [from previous version]

The following observation helps us to prove Theorem~\ref{thm:K1p+1free out-nbd BH same as LBH} below.  

\textbf{Observation.}
\emph{
Let \( p\geq 2 \), and let \( G \) be a \( 2p \)-regular graph with an orientation \( \vec{G} \). 
Let \( \psi \colon V(\vec{G})\to V\left(\overrightarrow{G_{2p}}\right) \) be an out-neighbourhood bijective homomorphism from \( \vec{G} \) to \( \overrightarrow{G_{2p}} \). 
If \( (v,w,x) \) is a triangle in \( G \), then there exist pairwise distinct integers \allowbreak \( i,j,k\in \{0,1,\dots,p+1\} \) such that either (i)~\( \psi(v)=(i,j) \), \( \psi(w)=(j,k) \) and \( \psi(x)=(k,i) \), or (ii)~\( \psi(v)=(k,i) \), \( \psi(w)=(j,k) \) and \( \psi(x)=(i,j) \). 
}
\todo{
Reviewer G:\\
The goal of this statement is twofold:\\ 
1. the cycle \( (v,w,x) \) is cyclically directed in one of the two possible directions, and\\
2. the labels form sequence \( (i,j),(k,i),(j,k) \).
Indeed \( (i,j),(j,k),(k,i) \) sounds more natural as I have suggested earlier.

Rephrase accordingly.
}
\begin{tcolorbox}[colorback=gray!10!white,boxrule=0.5pt,boxsep=0pt,colframe=black]
The observation is removed. 
\( \psi \) being a homomorphism between undirected graphs was enough (showing that \( \psi \) maps traingles to traingles was sufficient for our purpose).
\end{tcolorbox}
\begin{proof}
Let \( (v,w,x) \) be a triangle in \( G \). 
We know that \( \psi \) is a homomorphism from \( \vec{G} \) to~\( \overrightarrow{G_{2p}} \). 
As a result, \( \psi \) is a homomorphism from \( G \) to \( G_{2p} \). 
Consider the edge \( vx \) of \( G \). 
Since \( vx\in E(G) \) and \( \psi \) is a homomorphism from \( G \) to \( G_{2p} \), we have \( \psi(v)\psi(x)\in E(G_{2p}) \). 
Since every edge in \( G_{2p} \), and in particular \( \psi(v)\psi(x) \), is of the form \( \{(i,j),(k,i)\} \), there exist distinct integers \( i,j,k\in \{0,1,\dots,p+1\} \) such that either (i)~\( \psi(v)=(i,j) \) and \( \psi(x)=(k,i) \), or (ii)~\( \psi(x)=(i,j) \) and \( \psi(v)=(k,i) \). 
Since the vertex \( w \) of \( G \) is adjacent to both \( v \) and \( x \), and \( \psi \) is a homomorphism from \( G \) to \( G_{2p} \), the vertex \( \psi(w) \) of \( G_{2p} \) is adjacent to both \( \psi(v) \) and \( \psi(x) \). 
That is, \( \psi(w) \) is adjacent to \( (i,j) \) and \( (k,i) \). 
Since \( (j,k) \) is the only vertex adjacent to both \( (i,j) \) and \( (k,i) \) in \( G_{2p} \), we have \( \psi(w)=(j,k) \). 
Therefore, there exist pairwise distinct integers \( i,j,k\in \{0,1,\dots,p+1\} \) such that either (i)~\( \psi(v)=(i,j) \), \( \psi(w)=(j,k) \) and \( \psi(x)=(k,i) \), or (ii)~\( \psi(v)=(k,i) \), \( \psi(w)=(j,k) \) and \( \psi(x)=(i,j) \). 
\end{proof}
}

By Theorem~\ref{thm:4-star colourable in terms of distance-two colouring}, for every 3-regular graph \( G \), the line graph of \( G \) is 4-star colourable if and only if \( G \) is bipartite and distance-two 4-colourable. 
Feder et al.\ proved that if all faces of a plane 3-regular graph \( G \) are of length divisible by~4, then \( G \) is distance-two 4\nobreakdash-colourable \cite[Corollary 2.4]{feder}. 
Since such graphs \( G \) are bipartite as well, \( L(G) \) is 4-star colourable by Theorem~\ref{thm:4-star colourable in terms of distance-two colouring}. 
\todo{
Reviewer G:\\
Define this graph (\( CL_{4r} \)). \colcol{(done)}
}
As a special case, \( L(CL_{4r}) \) 
is 4-star colourable for each \( r\in \mathbb{N} \) (where \( CL_t \) denotes the circular ladder graph on \( 2t \) vertices, which is the cartesian product of \( C_t \) with \( K_2 \)); as a result, \( L(CL_{4r}) \) has an orientation that admits an out-neighbourhood bijective homomorphism to \( \vec{G_4} \) by Theorem~\ref{thm:characterise 2p-regular p+2-star colourable}. 
It is easy to show that for every \( r\in \mathbb{N} \), the graph \( L(CL_{4r}) \) admits a locally bijective homomorphism to \( L(CL_4) \) (we prove a more general statement below). 
Since \( Q_3\cong CL_4 \) and \( L^*(K_4)\cong L(Q_3) \), the above statement is equivalent to ``\( L(CL_{4r}) \) admits a locally bijective homomorphism to \( L^*(K_4) \) for every \( r\in \mathbb{N} \)''. 
We show that a \( K_{1,3} \)-free 4-regular graph \( G \) has an orientation that admits an out-neighbourhood bijective homomorphism \( \psi \) to \( \vec{L}(K_4) \) if and only if \( \psi \) is a locally bijective homomorphism from \( G \) to \( L^*(K_4) \). 
The next theorem proves a more general statement: for a \( K_{p+1} \)-free \( 2p \)-regular graph \( G \) with \( p\geq 2 \), a mapping \( \psi \colon V(G)\to V(L^*(K_{p+2})) \) is an OBH from some orientation of \( G \) to \( \vec{L}(K_{p+2}) \) if and only if \( \psi \) is an LBH from \( G \) to \( L^*(K_{p+2}) \). 
%
\todo{
Reviewer G:\\
Thm 9 is indeed a consequence of a more general argument: If \( \vec{H} \) is strongly connected then every locally out-neighbourhood bijective homomorphism \( f \) is locally bijective.

Proof idea: for \( (u,v)\in E(H) \), from out- follows \( |f^{-1}(u)|\le |f^{-1}(v)| \). Along directed cycles we get \( = \). Therefore \( f^{-1}(u,v) \) always induces a matching. QED.
}
\todo[color=gray!10!white]{
The generalisation proposed is not correct. 
Suppose that \( G \) and \( H \) are graphs with orientations \( \vec{G} \) and \( \vec{H} \), respectively. 
Also, suppose that \( \vec{H} \) is strongly connected, and \( \psi \) is an OBH from \( \vec{G} \) to \( \vec{H} \). 
It does not follow that \( |\psi^{-1}(u)|\leq |\psi^{-1}(v)| \) for every arc \( (u,v) \) of \( \vec{H} \). 
By a different proof, we show that if \( \psi \) is also degree-preserving, then \( |\psi^{-1}(u)|=|\psi^{-1}(v)| \) for all \( u,v\in V(H) \) (included as Theorem~\ref{thm:strongly connected and degree-preserving}). 
Yet, \( |\psi^{-1}(u)|=|\psi^{-1}(v)| \) for all \( u,v\in V(H) \) does not imply that \( \psi \) is an LBH from \( G \) to \( H \) (because \( G[\psi^{-1}(u)\cup \psi^{-1}(v)] \) need not be a matching, although in it, each copy of \( u \) has exactly one incident edge; see Corollary~\ref{cor:strongly connected and degree-preserving}). 
}

\begin{theorem}\label{thm:K1p+1free out-nbd BH same as LBH}
Let \( G \) be a \( K_{1,p+1} \)-free \( 2p \)-regular graph with \( p\geq 2 \). 
Let \( \vec{G} \) be an orientation of \( G \) that admits an OBH \( \psi \) to \( \vec{L}(K_{p+2}) \). 
Then, \( \psi \) is an LBH from \( \vec{G} \) to \( \vec{L}(K_{p+2}) \). 
In particular, \( \psi \) is an LBH from \( G \) to \( L^*(K_{p+2}) \). 
\end{theorem}
\noindent
\emph{Remark:} If \( \psi \) is an LBH from an aribtrary graph \( G \) to \( L^*(K_{p+2}) \), then \( \psi \) is obviously an LBH from an orientation of \( G \) to \( \vec{L}(K_{p+2}) \). 
\begin{proof}
We know that \( \vec{G} \) is an Eulerian orientation. 
Hence, each vertex of \( G \) has excatly \( p \) in-neighbours. 
To prove that the OBH \( \psi \) is an LBH from \( \vec{G} \) to \( \vec{L}(K_{p+2}) \), it suffices to show that for each vertex \( v \) of \( \vec{G} \), no two in-neighbours of \( v \) have the same image under \( \psi \). 

Define a function \( f\colon V(\vec{G})\to \mathbb{Z}_{p+2} \) as \( f(v)=\mathrm{proj}_2(\psi(v)) \) for each vertex \( v \) of \( G \). 
Since \( \psi \) is a homomorphism from \( \vec{G} \) to \( \vec{L}(K_{p+2}) \), for each arc \( (u,v) \) of \( \vec{G} \), we have \( \psi(u)=(i,j) \) and \( \psi(v)=(j,k) \) for some \( i,j,k\in \mathbb{Z}_{p+2} \). 
Hence, \( f(u)\neq f(v) \) for every edge \( uv \) of \( G \). 
That is, \( f \) is a \( (p+2) \)-colouring of \( G \). 
By Lemma~\ref{lem:star col iff mini iff obh}, \( (\vec{G},f) \) is a \( (p+2) \)-coloured MINI-orientation of \( G \) (note that \( \mathrm{proj}_2\circ \psi=f \)). 

Let \( v \) be an aribtrary vertex in \( \vec{G} \). 
Without loss of generality, let \( \psi(v)=(p+1,0) \). 
Clearly, the homomorphism \( \psi \) from \( \vec{G} \) to \( \vec{L}(K_{p+2}) \) maps each in-neighbour of \( v \) to \( (i,p+1) \) for some \( i\in\mathbb{Z}_{p+2}\setminus\{0,p+1\} \) (i.e., \( i\in \{1,2,\dots,p\} \)). 
Let \( W=\{w_1,w_2,\dots,w_p\} \) be the set of in-neighbours and \( X=\{x_1,x_2,\dots,x_p\} \) be the set of out-neighbours of \( v \) (in \( \vec{G} \)). 
Clearly, \( f(w_1)=f(w_2)=\dots=f(w_p)=p+1 \). 
Hence, \( W \) is an independent set in \( G \). 
As it is an OBH, \( \psi \) maps the set \( X \) of out-neighbours of \( v \) to \( \{(0,1),(0,2),\dots,(0,p)\} \). 
Without loss of generality, let \( \psi(x_j)=(0,j) \) for \( 1\leq j\leq p \). 
Since \( G \) is \( K_{1,p+1} \)-free and \( W\subseteq N_G(v) \) is an independent set of size \( p \), each out-neighbour \( x_j \) of \( v \) is adjacent to some in-neighbour \( w_k\in W \) of \( v \). 
Thus, there is a function \( \sigma\colon \{1,2,\dots,p\}\to \{1,2,\dots,p\} \) such that \( w_{\sigma(j)} \) is a neighbour of \( x_j \) for \( 1\leq j\leq p \). 

The homomorphism \( \psi \) from \( \vec{G} \) to \( \vec{L}(K_{p+2}) \) is apparently a homomorphism from \( G \) to \( L^*(K_{p+2}) \). 
For \( 1\leq j\leq p \), since \( (v,x_j,w_{\sigma(j)}) \) is a triangle in \( G \), the homomorphism \( \psi \) from \( G \) to \( L^*(K_{p+2}) \) maps \( (v,x_j,w_{\sigma(j)}) \) to \( \big((p+1,0),(0,j),(j,p+1)\big) \), the only triangle in \( L^*(K_{p+2}) \) that contains both vertices \( (p+1,0) \) and \( (0,j) \). 
Hence, \( \psi(w_{\sigma(j)})=(j,p+1) \) for each \( j \). 
Therefore, \( \psi \) maps \( w_1,w_2,\dots,w_p \) to \( (1,p+1),(2,p+1),\dots,(p,p+1) \) in some order. 
In particular, the images of distinct in-neighbours of \( v \) under \( \psi \) are distinct. 
Since \( v \) is arbitrary, this proves that \( \psi \) is an LBH from \( \vec{G} \) to \( \vec{L}(K_{p+2}) \). 
\end{proof}


\todo{
(omitted) [Rewritten]

\noindent
\textbf{Theorem~\ref{thm:K1p+1free out-nbd BH same as LBH}.}
Let \( G \) be a \( K_{1,p+1} \)-free \( 2p \)-regular graph with \( p\geq 2 \). 
Then, \( \psi \colon V(G)\to V(G_{2p}) \) is an out-neighbourhood bijective homomorphism from an orientation \( \vec{G} \) of \( G \) to \( \overrightarrow{G_{2p}} \) if and only if \( \psi \) is a locally bijective homomorphism from \( G \) to \( G_{2p} \). 
\todo{
Reviewer G:\\
Rephrase: Change ``and let \( \psi \colon V(G)\to V(G_{2p}) \). Then, \dots iff it is \dots'' to\\ ``Then any \( \psi \colon V(G)\to V(G_{2p}) \) is \dots iff it is \dots''
}
\begin{tcolorbox}[colorback=gray!10!white,boxrule=0.5pt,boxsep=0pt,colframe=black]
The theorem statement is re-written for emphasis on the main direction of interest. 
The suggested change is applied on the sentence before the theorem statement. 
\end{tcolorbox}
\todo{``\( K_{1,p+1} \)''. 
Even though it is often used in literature,  define properly that it means `has no induced subgraph' to avoid any confusion.\\
}
\begin{tcolorbox}[colorback=gray!10!white,boxrule=0.5pt,boxsep=0pt,colframe=black]
The term \( H \)-free is defined in Section~\ref{sec:def}. 
\end{tcolorbox}

\emph{Proof.}
Suppose that \( \psi \) is an out-neighbourhood bijective homomorphism from an orientation \( \vec{G} \) of \( G \) to \( \overrightarrow{G_{2p}} \). 
As a result, \( \psi \) is a homomorphism from \( \vec{G} \) to \( \overrightarrow{G_{2p}} \). 
For distinct \( i,j\in \{0,1,\dots,p+1\} \), let \( V_i^j \) denote the set of copies of \( (i,j) \) in \( G \); that is, \( V_i^j=\psi^{-1}\big((i,j)\big) \). 
Consider an arbitrary vertex \( (i,j) \) of \( G_{2p} \) and a copy \( v \) of \( (i,j) \) in \( \vec{G} \). 
In \( \overrightarrow{G_{2p}} \), the vertex \( (i,j) \) has \( (k,i) \) as an out-neighbour for each \( k\in \mathbb{Z}_{p+2}\setminus \{i,j\} \), and each in-neighbour of \( (i,j) \) is of the form \( (j,k) \) for some \( k\in \mathbb{Z}_{p+2}\setminus \{i,j\} \). 
Since \( \psi \) is an out-neighbourhood bijective homomorphism from \( \vec{G} \) to \( \overrightarrow{G_{2p}} \), the copy \( v \) of \( (i,j) \) in \( \vec{G} \) has exactly one copy of \( (k,i) \) in \( \vec{G} \) as its out-neighbour in \( \vec{G} \) for \( k\in \mathbb{Z}_{p+2}\setminus \{i,j\} \). 
The remaining \( p \) neighbours of \( v \) in \( \vec{G} \) are in-neighbours of \( v \) in \( \vec{G} \). 
Since \( \psi \) is a homomorphism from \( \vec{G} \) to \( \overrightarrow{G_{2p}} \), for each in-neighbour \( w \) of \( v \) in \( \vec{G} \), \( \psi(w) \) is an in-neighbour of  \( \psi(v)=(i,j) \) in \( \overrightarrow{G_{2p}} \). 
That is, for each in-neighbour \( w \) of \( v \) in \( \vec{G} \), \( \psi(w)=(j,k) \) for some \( k\in \mathbb{Z}_{p+2}\setminus \{i,j\} \), and thus \( w\in V_j^k \). 
Hence, \( v \) has exactly one copy of \( (k,i) \) in \( \vec{G} \) as its neighbour for each \( k\in \mathbb{Z}_{p+2}\setminus \{i,j\} \), and each of the remaining \( p \) neighbours of \( v \) in \( \vec{G} \) is a copy of \( (j,k) \) in \( \vec{G} \) for some \( k\in \mathbb{Z}_{p+2}\setminus \{i,j\} \). 
In other words, in \( G \), \( v \) has exactly one neighbour in \( V_k^i \) for each \( k\in \mathbb{Z}_{p+2}\setminus \{i,j\} \) and the remaining \( p \) neighbours of \( v \) are in \( \bigcup_{k\in \mathbb{Z}_{p+2}\setminus \{i,j\}} V_{j}^{k} \). 
Since the copy \( v \) of \( (i,j) \) in \( \vec{G} \) is arbitrary, each vertex \( v \) in~\( V_i^j \) (i.e., a copy of \( (i,j) \) in \( \vec{G} \)) has exactly one neighbour in \( V_k^i \) for each \( k\in \mathbb{Z}_{p+2}\setminus \{i,j\} \) and the remaining \( p \) neighbours of \( v \) are in \( \bigcup_{k\in \mathbb{Z}_{p+2}\setminus \{i,j\}} V_{j}^{k} \). 
Since the vertex \( (i,j) \) of \( G_{2p} \) is arbitrary, we have the following claim. 
\setcounter{myclaim}{0}
\begin{myclaim}\label{clm:psi is LBH by copy def step0}
For distinct \( i,j\in \mathbb{Z}_{p+2} \), each vertex in \( V_i^j \) has exactly one neighbour in \( V_k^i \) for each \( k\in \mathbb{Z}_{p+2}\setminus \{i,j\} \), and exactly \( p \) neighbours in \( \bigcup_{k\in \mathbb{Z}_{p+2}\setminus \{i,j\}} V_{j}^{k} \). 
\end{myclaim}

Since \( \psi \) is a homomorphism from \( \vec{G} \) to \( \overrightarrow{G_{2p}} \), it is obviously a homomorphism from \( G \) to \( G_{2p} \). 
To prove that \( \psi \) is a locally bijective homomorphism from \( G \) to \( G_{2p} \), it suffices to show that for distinct \( i,j\in \{0,1,\dots,p+1\} \), each copy of \( (i,j) \) in \( G \) has exactly one copy of \( (j,k) \) in \( G \) and exactly one copy of \( (k,i) \) in \( G \) at its neighbour. 
Since \( V_i^j \) is the set of copies of \( (i,j) \) in \( G \) for each \( (i,j)\in V(G_{2p}) \), it suffices to prove the following claim. 
}

\todo{
(omitted) [continues]

\begin{myclaim}\label{clm:psi is LBH by copy def}
For distinct \( i,j\in \mathbb{Z}_{p+2} \), each vertex in \( V_i^j \) has exactly one neighbour in \( V_j^k \) and exactly one neighbour in \( V_k^i \) for each \( k\in\mathbb{Z}_{p+2}\setminus \{i,j\} \). 
\end{myclaim}
We prove the claim for \( i=0 \) and \( j=1 \) (the proof is similar for other values of \( i \) and~\( j \)). 
Let \( v\in V_0^1 \). 
By Claim~\ref{clm:psi is LBH by copy def step0}, \( v \) has exactly \( p \) neighbours in \( \bigcup_{k=2}^{p+1} V_1^k \) and exactly one neighbour in \( V_{r+1}^0 \) for \( 1\leq r\leq p \). 
Let \( w_1,\dots,w_p,x_1,\dots,x_p \) be the neighbours of \( v \) in \( G \) where \( w_1,\dots,w_p\in \bigcup_{k=2}^{p+1} V_1^k \) and \( x_r\in V_{r+1}^0 \) for \( 1\leq r\leq p \). 
Let \( W=\{w_1,\dots,w_p\} \). 
Since \( \psi \) is a homomorphism from \( G \) to \( G_{2p} \), the set \( \psi(W)=\{\psi(w)\colon w\in W\} \) is a subset of \( \{(1,2),\dots,(1,p+1)\} \), and thus \( W \) is an independent set in \( G \) (because \( \{(1,2),\dots,(1,p+1)\} \) is an independent set in \( G_{2p} \)). 
Hence, the subgraph of \( G \) induced by \( \{v,w_1,\dots,w_p\} \) is isomorphic to \( K_{1,p} \). 
Since \( x_1 \) is a neighbour of \( v \), the subgraph of \( G \) induced by \( \{v,w_1,\dots,w_p,x_1\} \) is isomorphic to \( K_{1,p+1} \) unless \( x_1 \) has a neighbour in \( W \). 
As a result, \( x_1 \) has a neighbour in \( W \) because \( G \) is \( K_{1,p+1} \)-free. 
Similarly, for \( 1\leq r\leq p \), the vertex \( x_r \) has a neighbour in \( W \). 
Note that \( \psi(v)=(0,1) \) and \( \psi(x_1)=(2,0) \). 
Suppose that \( w \) is a neighbour of \( x_1 \) in \( W \). 
Since \( (v,w,x_1) \) is a triangle in \( G \), \( \psi(w)=(1,2) \) by Observation~\ref{obs:triangle in extremal graphs}, and thus \( w\in V_1^2 \). 
Similarly, for \( 1\leq r\leq p \), if \( w \) is a neighbour of \( x_r \) in \( W \), then \( \psi(w)=(1,r+1) \), and thus \( w\in V_1^{r+1} \). 
Since \( v \) has exactly \( p \) neighbours in \( \bigcup_{k=2}^{p+1} V_1^k \) and \( v \) has a neighbour in \( V_1^{r+1} \) for each \( r\in \{1,2,\dots,p\} \), the vertex \( v \) has exactly one neighbour in \( V_1^{r+1} \) for each \( r\in \{1,2,\dots,p\} \). 
Therefore, \( v \) has exactly one neighbour in \( V_1^k \) and exactly one neighbour in \( V_k^0 \) for each \( k\in \{2,3,\dots,p+1\} \). 
This proves Claim~\ref{clm:psi is LBH by copy def}. 
Consequently, \( \psi \) is a locally bijective homomorphism from \( G \) to \( G_{2p} \).

Conversely, suppose that \( \psi \) is a locally bijective homomorphism from \( G \) to \( G_{2p} \). 
Then, there exists a unique orientation \( \vec{G} \) of \( G \) such that \( \psi \) is an out-neighbourhood bijective homomorphism from \( \vec{G} \) to \( \overrightarrow{G_{2p}} \). 
This proves the converse part. 
\hfill
\( \square \)
}

By Theorem~\ref{thm:characterise 2p-regular p+2-star colourable}, for all \( p\geq 2 \), a \( 2p \)-regular graph \( G \) is  \( (p+2) \)-star colourable if and only if \( G \) has an orientation \( \vec{G} \) that admits an OBH \( \psi \) to \( \vec{L}(K_{p+2}) \). 
Hence, Theorem~\ref{thm:K1p+1free out-nbd BH same as LBH} implies the following. 

\begin{theorem}\label{thm:p+2 star colourable iff LBH}
Let \( p\geq 2 \), and let \( G \) be a \( K_{1,p+1} \)-free \( 2p \)-regular graph. 
Then, \( G \) is \( (p+2) \)-star colourable if and only if \( G \) admits an LBH to \( L^*(K_{p+2}) \). 
\qed
\end{theorem}
\todo{
Reviewer F:\\
Theorem 10 (currently Theorem~\ref{thm:p+2 star colourable iff LBH}) is in fact a restatement of Theorem 1, I would suggest to make it clear in the same way as with Theorems 2 and 3
}
\todo[color=gray!10!white]{
The theorem statements are moved so that proofs are close to them. 
Restatement of theorems is avoided. 
}

Let \( G \) be a graph that admits an LBH to \( L^*(K_{p+2}) \), where \( p\geq 2 \). 
By Theorem~\ref{thm:LstarH has LBH to LH}, \( L^*(K_{p+2}) \) admits an LBH to \( L(K_{p+2}) \), and thus \( G \) admits an LBH to \( L(K_{p+2}) \) (since a function composition of LBHs is an LBH). 
As a result, the characteristic polynomial of (adjacency matrix of) \( G \) is divisible by that of \( L(K_{p+2}) \)~\cite{fiala_kratochvil}. 
Hence, the characteristic polynomial of \( L(K_{p+2}) \) in \( x \), that is \( (x-2p) \)\( (x-p+2)^{p+1} \)\( (x+2)^{(p-1)(p+2)/2} \)~\cite[Table~4.1]{beineke_bagga2021a}, divides the characteristic polynomial of \( G \) in~\( x \). 
Since each \( K_{1,p+1} \)-free \( 2p \)-regular \( (p+2) \)-star colourable graph admits an LBH to \( L^*(K_{p+2}) \) (by Theorem~\ref{thm:p+2 star colourable iff LBH}), we have the following.
\begin{theorem}\label{thm:star col eigenvalue}
Let \( G \) be a \( K_{1,p+1} \)-free \( 2p \)-regular graph with \( p\geq 2 \). 
If \( G \) is \( (p+2) \)-star colourable, then \( {-2} \) and \( p-2 \) are eigenvalues of \( G \) with multiplicities at least \( (p-1)(p+2)/2 \) and \( p+1 \), respectively. 
\qed
\end{theorem}

Next, we show that the structure of \( K_{1,p+1} \)-free \( 2p \)-regular \( (p+2) \)-star colourable graphs is even more limited when \( p=2 \). 
\begin{theorem}\label{thm:4reg 4-starcol clawfree imply line}
Let \( G \) be a claw-free 4-regular 4-star colourable graph. 
Then, \( G \) is the line graph of a bipartite graph, and in particular, \( G \) is odd-hole-free. 
\end{theorem}
\begin{proof}
Since \( G \) is 4-regular and 4-star colourable, \( G \) is \( (\mathrm{diamond},K_4) \)-free~\cite[Corollary~2]{shalu_cyriac3}. 
Hence, \( G \) is \( (\mathrm{claw}, \mathrm{diamond}, K_4) \)-free. 
But, \( (\mathrm{claw}, \mathrm{diamond}) \)-free graphs are line graphs~\cite{beineke} (in fact, they are precisely the line graphs of triangle-free graphs~\cite{beineke_bagga2021a}).  
Since \( G \) is a \( K_4 \)-free line graph, \( G \) is the line graph of a 3-regular graph \( H \). 
Since \( H \) is 3-regular and its line graph \(  G\) is 4-star colourable, \( H \) is bipartite by Theorem~\ref{thm:4-star colourable in terms of distance-two colouring}. 
Hence, \( G \) is the line graph of a bipartite graph. 
That is, \( G \) is \( (\mathrm{claw}, \mathrm{diamond},\mathrm{odd-hole}) \)-free~\cite{beineke_bagga2021a}. 
\end{proof}

For \( p\in \mathbb{N} \), the \emph{friendship graph} \( F_p \) is the graph obtained from \( p\,C_3 \) by `gluing' the \( p \) triangles together at a single vertex. 
For \( p\geq 2 \), the friendship graph \( F_p \) has a unique universal vertex. 
For \( p\geq 2 \), let \( F_p^- \) denote the graph obtained from \( F_p \) by removing an edge not incident on the universal vertex of \( F_p \). 
For each vertex \( v \) of a diamond-free \( 2p \)-regular graph \( G \) with \( p\geq 2 \), the neighbourhood \( N_G(v) \) induces one of the graphs \( pK_2, (p-1)K_2+2K_1, \) \( (p-2)K_2+4K_1, \dots, 2pK_1 \). 
Thus, we have the following corollary. 
\begin{corollary}
Let \( G \) be a \( (K_{1,p+2}, F_p^-) \)-free 2p-regular graph with \( p\geq 2 \). 
Then, \( G \) is \( (p+2) \)-star colourable if and only if \( G \) admits an LBH to \( L^*(K_{p+2}) \). 
\qed
\end{corollary}

For integers \( p,q\in \mathbb{N} \) with \( 2\leq q\leq p \), the same result holds for \( 2p \)-regular graphs that are \( K_{1,p+q} \)-free as well as \( (F_p^{-r}) \)-free for all \( r\in \{1,2,\dots,q-1\} \), where \( F_p^{-r} \) denotes the graph obtained from \( F_p \) by removing \( r \) edges not incident on the universal vertex of \( F_p \). 

Observe that for each \( p\geq 2 \), the set of neighbours of \( (0,1) \) in \( L^*(K_{p+2}) \) is \( \{(1,2),\dots,(1,p+1) \), \( (2,0),\dots,(p+1,0)\} \) and the subgraph of \( L^*(K_{p+2}) \) induced by this set is isomorphic to \( pK_2 \), where the edges in the subgraph are \( \{(1,2),(2,0)\} \), \dots, \( \{(1,p+1),(p+1,0)\} \). 
Similarly, for each \( p\geq 2 \) and each vertex \( v \) of \( L^*(K_{p+2}) \), the subgraph of \( L^*(K_{p+2}) \) induced by the neighbourhood of \( v \) is \( pK_2 \). 
For \( p\geq 2 \), since \( L^*(K_{p+2}) \) is a connected graph (in fact, Hamiltonian by Theorem~\ref{thm:Lstarkq ham}), 
 \( L^*(K_{p+2})\in F(p,2) \). 
Hence, for each \( p\geq 2 \), \( L^*(K_{p+2})\in F(p,2) \) and thus it is locally linear by Observation~\ref{obs:locally linear iff Fp2}. 

%
\begin{lemma}\label{lem:locally linear iff diamond K4 K1p+1 free}
Let \( p\geq 2 \), and let \( G \) be a connected \( 2p \)-regular graph. 
Then, \( G \) is locally linear if and only if \( G \) is \( ( \)\textup{diamond}, \( K_4 \), \( K_{1,p+1} \)\( ) \)-free. 
\end{lemma}
\begin{proof}
Suppose that \( G \) is locally linear. 
Since each edge in \( G \) is in exactly one triangle in \( G \), we know that \( G \) is \( ( \)\textup{diamond}, \( K_4 \)\( ) \)-free. 
Since \( G \) is a connected locally linear \( 2p \)-regular graph, \( G\in F(p,2) \) by Observation~\ref{obs:locally linear iff Fp2}. 
Consider an arbitrary vertex \( v \) in~\( G \). 
We know that \( G[N_G(v)]\cong pK_2 \). 
Hence, the independence number of \( G[N_G(v)] \) is exactly \( p \). That is, \( v \) has at most \( p \) neighbours in \( G  \) that are pairwise non-adjacent. 
Thus, there is no induced subgraph \( H \) of \( G \) such that (i)~\( H\cong K_{1,p+1} \) and (ii)~\( v \) is the centre of~\( H \). 
Since \( v \) is arbitrary, there is no induced subgraph of \( G \) isomorphic to~\( K_{1,p+1} \). 
Hence, \( G \) is (\textup{diamond}, \( K_4 \), \( K_{1,p+1} \))-free.

Conversely, suppose that \( G \) is (\textup{diamond}, \( K_4 \), \( K_{1,p+1} \))-free. 
Consider an arbitrary vertex \( v \) of \( G \). 
We know that \( G[N_G(v)] \) contains exactly  \( 2p \) vertices. 
Since \( G \) is (\textup{diamond}, \( K_4 \))-free, each component of \( G[N_G(v)] \) contains at most one edge. 
Suppose that \( G[N_G(v)] \) contains exactly  \( q \) edges, where \( q\leq p \). 
Clearly, \( G[N_G(v)] \) contains exactly \( 2p-2q \) isolated vertices, and thus \( G[N_G(v)]\cong qK_2+(2p-2q)K_1 \). 
Hence, the independence number of \( G[N_G(v)] \) is exactly \( q+(2p-2q)=2p-q \). 
Thus, \( v \) has \( 2p-q \) neighbours in \( G \) that are pairwise non-adjacent. 
Since \( G \) is \( K_{1,p+1} \)\nobreakdash-free, we have \( 2p-q\leq p \); that is, \( p\leq q \). 
Since \( q\leq p \) and \( p\leq q \), we have \( q=p \). 
That is, \( G[N_G(v)]\cong pK_2 \). 
Since \( v \) is arbitrary and \( G \) is connected, \( G\in F(p,2) \). 
Hence, \( G \) is locally linear by Observation~\ref{obs:locally linear iff Fp2}. 
\end{proof}
\begin{theorem}\label{thm:larger subclass of locally linear}
Let \( p\geq 2 \), and let \( G \) be a connected \( K_{1,p+1} \)-free \( 2p \)-regular graph. 
If \( G \) is \( (p+2) \)-star colourable, then \( G \) is a locally linear graph as well as a clique graph, and \( K(K(G))\cong G \). 
\end{theorem}
\begin{proof}
Suppose that \( G \) is \( (p+2) \)-star colourable. 
Then, \( G \) is (\textup{diamond}, \( K_4 \))-free by \cite[Corollary~2]{shalu_cyriac3}. 
Since \( G \) is a (\textup{diamond}, \( K_4 \), \( K_{1,p+1} \))-free graph, \( G \) is locally linear by Lemma~\ref{lem:locally linear iff diamond K4 K1p+1 free}. 
Since \( G \) is a connected graph as well, \( G\in F(p,2) \) by Observation~\ref{obs:locally linear iff Fp2}. 
Thus, \( K(K(G))\cong G \) by \cite[Theorem~1.4]{devillers}, and hence \( G \) itself is a clique graph. 
\end{proof}

Theorem~\ref{thm:4-star colourable in terms of distance-two colouring} proved that for every 3-regular graph \( G \), the line graph of \( G \) is 4-star colourable if and only if \( G \) is bipartite and distance-two 4-colourable. 
Feder et al.~\cite{feder} proved that given a planar 3-regular 3-connected bipartite graph \( G \), it is NP-complete to check whether \( G \) is distance-two 4-colourable. Thus, the following is a corollary of Theorem~\ref{thm:4-star colourable in terms of distance-two colouring}. 
\begin{corollary}\label{cor:4-star colouring 4-regular NPC}
Given a planar 3-regular 3-connected bipartite graph \( G \), it is NP-complete to check whether the line graph of \( G \) is 4-star colourable. 
In particular, (i)~it is NP-complete to check whether a planar 4-regular 3-connected graph is 4-star colourable, and (ii)~it is NP-complete to check whether a \( K_{1,3} \)-free 4-regular graph is 4-star colourable. 
\qed
\end{corollary}

If a \( K_{1,3} \)-free 4-regular graph \( G \) is 4-star colourable, then \( G \) is a locally-\( 2K_2 \) graph by Theorem~\ref{thm:larger subclass of locally linear}. 
Given a graph \( G \), one can test in polynomial time whether \( G \) is locally-\( 2K_2 \). 
Moreover, locally-\( 2K_2 \) graphs are clique graphs~\cite[Theorem~1.4]{devillers}. 
Thus, we have the following. 
\begin{corollary}\label{cor:locally 2K2 4-star col npc}
It is NP-complete to check whether a planar locally-\( 2K_2 \) graph is 4-star colourable.  
Hence, it is NP-complete to check whether a planar 4-regular clique graph is 4-star colourable. 
\qed
\end{corollary}
Thanks to Theorem~\ref{thm:characterise 2p-regular p+2-star colourable}, we also have the following corollary. 
\begin{corollary}
Given a \( ( \)planar bipartite\( ) \) graph \( G \) and a \( ( \)strongly connected\( ) \) oriented graph \( \vec{H} \), it is NP-complete to check whether \( G \) has an  orientation that admits a \( ( \)degree-preserving\( ) \) OBH to~\( \vec{H} \). 
\qed
\end{corollary}

Corollary~2 in \cite{shalu_cyriac3} proved that for \( p\geq 2 \), a \( 2p \)-regular \( (p+2) \)-star colourable graph does not contain diamond as a subgraph. 
The next theorem shows that for \( p\geq 2 \), a \( 2p \)-regular \( (p+2) \)-star colourable graph does not contain the circular ladder graph \( CL_{2r+1} \) as a subgraph for any \( r\in \mathbb{N} \). 
\emph{Remark:} Using Lemmas~\ref{lem:C3 in CEO} and \ref{lem:C4 in CEO} below, it is easy to show that the diamond graph does not admit a \( q \)\nobreakdash-coloured MINI-orientation for any \( q\in \mathbb{N} \). 

\begin{theorem}\label{thm:CL2r+1 obstruction to CEO}
For \( r\in \mathbb{N} \), the circular ladder graph \( CL_{2r+1} \) does not admit a \( q \)\nobreakdash-coloured MINI-orientation for any \( q\in \mathbb{N} \). 
Thus, for \( p\geq 2 \), a \( 2p \)-regular \( (p+2) \)-star colourable graph does not contain \( CL_{2r+1} \) as a subgraph. 
\qed
\end{theorem}

Theorem~\ref{thm:CL2r+1 obstruction to CEO} follows from Corollary~\ref{cor:opposite sides in C4 under CEO} below. 
The proofs of the following results are straightforward and hence moved to Appendix~\ref{sec:A.2}. 

\begin{lemma}\label{lem:C3 in CEO}
Let \( G \) be a graph, and let \( (\vec{G},f) \) be a \( q \)-coloured MINI-orientation of \( G \) for some \( q\in \mathbb{N} \). 
Then, \( \vec{G} \) orients each triangle in \( G \) as a directed cycle. 
\qed
\end{lemma}
%
%
%
\begin{lemma}\label{lem:C4 in CEO}
Let \( G \) be a graph, and let \( (\vec{G},f) \) be a \( q \)-coloured MINI-orientation of \( G \) for some \( q\in \mathbb{N} \). 
Then, for every 4-vertex cycle in \( G \), not necessarily induced, edges in the cycle are oriented by \( \vec{G} \) either as a directed cycle (as in Figure~\ref{fig:C4 orientation 1}) or as a direction-alternating cycle (as in Figure~\ref{fig:C4 orientation 2}). 
\qed
\end{lemma}
\begin{figure}[hbt]
\centering
\begin{subfigure}[b]{0.3\textwidth}
\centering
\begin{tikzpicture}[decoration={
    markings,
    mark=at position 0.5 with {\arrow{>}}},
    ]
\path (0,0) node(u)[dot]{}--++(1,-1) node(v)[dot]{}--++(-1,-1) node(w)[dot]{}--++(-1,1) node(x)[dot]{}--(u); 
\draw [postaction=decorate,thick]
(u)--(v);
\draw [postaction=decorate,thick]
(v)--(w);
\draw [postaction=decorate,thick]
(w)--(x);
\draw [postaction=decorate,thick]
(x)--(u);
\end{tikzpicture}
\caption{}
\label{fig:C4 orientation 1}
\end{subfigure}\begin{subfigure}[b]{0.3\textwidth}
\centering
\begin{tikzpicture}[decoration={
    markings,
    mark=at position 0.5 with {\arrow{>}}},
    ]
\path (0,0) node(u)[dot]{}--++(1,-1) node(v)[dot]{}--++(-1,-1) node(w)[dot]{}--++(-1,1) node(x)[dot]{}--(u); 
\draw [postaction=decorate,thick]
(u)--(v);
\draw [postaction=decorate,thick]
(w)--(v);
\draw [postaction=decorate,thick]
(w)--(x);
\draw [postaction=decorate,thick]
(u)--(x);
\end{tikzpicture}
\caption{}
\label{fig:C4 orientation 2}
\end{subfigure}\caption{Possible orientations of a \( C_4 \) in a coloured MINI-orientation of \( G \).}\label{fig:C4 orientations}
\end{figure}
\begin{corollary}\label{cor:opposite sides in C4 under CEO}
Let \( G \) be a graph with a 4-vertex cycle \( (u,v,w,x) \), and let \( (\vec{G},f) \) be a \( q \)-coloured MINI-orientation of \( G \). 
If \( (u,v)\in E(\vec{G}) \), then \( (w,x)\in E(\vec{G}) \). 
\qed
\end{corollary}

Finally, we point out that given an orientation \( \vec{G} \) of a graph \( G \), one can test in polynomial-time whether \( \vec{G} \) is a MINI-orientation of \( G \). 
The proof of the theorem is straightforward and hence moved to Appendix~\ref{sec:A.3}. 
\begin{theorem}\label{thm:orientation mini test in P}
Given an orientation \( \vec{G} \) of a graph \( G \), we can test in polynomial time whether there exist an integer \( q \) and a \( q \)-colouring \( f \) of \( G \) such that \( (\vec{G},f) \) is a \( q \)-coloured MINI-orientation of~\( G \). 
\qed
\end{theorem}

\section{Conclusion and Open Problems}\label{sec:open problems}
For \( d\geq 3 \), at least \( \lceil (d+4)/2 \rceil \) colours are required to star colour a \( d \)-regular graph~\cite{shalu_cyriac3}. 
In particular, at least \( (p+2) \) colours are required to star colour \( 2p \)-regular graphs \( G \) with \( p\geq 2 \), and graphs \( G \) for which \( (p+2) \) colours suffice are characterised in terms of graph orientations in \cite{shalu_cyriac3} and in terms of graph homomorphisms in the current paper. 
The following is a natural follow-up question since at least \( (p+2) \) colours are required to star colour \( (2p-1) \)-regular graphs with \( p\geq 2 \). 
\begin{problem}
For \( p\geq 2 \), characterise \( (2p-1) \)-regular \( (p+2) \)-star colourable graphs. 
\end{problem}

For a fixed \( k\in \mathbb{N} \), the problem \textsc{\( k \)-Star Colourability} takes a graph \( G \) as input and asks whether \( G \) is \( k \)-star colourable. 
The problem \textsc{4-Star Colourability} is NP-complete even when restricted to \( K_{1,3} \)-free (planar) 4-regular graphs by Corollary~\ref{cor:4-star colouring 4-regular NPC}. 

\begin{problem}[\cite{shalu_cyriac3}]
For \( p\geq 3 \), is \textsc{\( (p+2) \)-Star Colourability} restricted to \( 2p \)-regular graphs NP-complete?
\end{problem}
\todo{
Reviewer G:\\
This is already answered in this paper by Thms 2 and 9, and the hardness result of [14], if I am not mistaken as \( G_{2p} \) is \( 2p \)-regular, isn't it?
}
\todo[color=gray!10!white]{
A clarification note was indeed necessary. 
It is added after Problem 2. 
}
Observe that this problem is indeed open. 
By Theorem~\ref{thm:characterise 2p-regular p+2-star colourable}, a \( 2p \)-regular graph \( G \) is \( (p+2) \)-star colourable if and only if \( G \) admits an OBH to \( L^*(K_{p+2}) \). 
But, the complexity of deciding whether an input  graph \( G \) admits an OBH to \( L^*(K_{p+2}) \) is open. 
By Lemma~\ref{lem:star col iff mini iff obh}, every OBH from an orientation of a \( K_{1,p+1} \)-free \( 2p \)-regular graph to \( \vec{L}(K_{p+2}) \) is an LBH. 
For each \( d \)-regular graph \( H \) with \( d\geq 3 \), it is NP-complete to test whether an input (\( d \)-regular) graph \( G \) admits an LBH to \( H \). 
In particular, it is NP-complete to test whether a \( 2p \)-regular graph \( G \) admits an LBH to \( L^*(K_{p+2}) \). 
Neverthless, the complexity status is unknown when \( G \) is guaranteed to be \( K_{1,p+1} \)-free (even though \( L^*(K_{p+2}) \) is \( K_{1,p+1} \)-free, a graph admitting LBH to \( L^*(K_{p+2}) \) need not be \( K_{1,p+1} \)-free).

\begin{problem}\label{prob:p+2 star col large star free 2p-regular}
For \( p\geq 3 \), is \textsc{\( (p+2) \)-Star Colourability} restricted to \( K_{1,p+1} \)-free \( 2p \)-regular graphs NP-complete?
\end{problem}
\noindent If the answer to Problem~\ref{prob:p+2 star col large star free 2p-regular} is `yes', then \textsc{\( (p+2) \)-Star Colourability} restricted to locally-\( pK_2 \) graphs is NP-complete (by the same arguments as in Corollary~\ref{cor:locally 2K2 4-star col npc}).\\

By Theorem~\ref{thm:CL2r+1 obstruction to CEO}, diamond and circular ladder graph \( CL_{2r+1} \)~(\( r\in \mathbb{N} \)) does not admit a \( q \)\nobreakdash-coloured MINI-orientation for any \( q\in \mathbb{N} \). 
\begin{problem}
Characterise graphs that do not admit a \( q \)\nobreakdash-coloured MINI-orientation for any \( q\in \mathbb{N} \). 
\end{problem}

Theorem~\ref{thm:4-star colourable in terms of distance-two colouring} proved that for every 3-regular graph \( G \), the line graph of \( G \) is 4-star colourable if and only if \( G \) is bipartite and distance-two 4-colourable. 
Generalisation of this result to larger graph classes is an interesting future direction. 
Determining the spectrum of \( L^*(G) \) for each graph \( G \) is another future direction we are interested in. 
Theorem~\ref{thm:LstarH has LBH to LH} revealed some information on the spectrum of \( L^*(G) \). 

\section*{Declaration of competing interest}
The authors declare that they have no known competing financial interests or personal relationships that could have appeared to influence the work reported in this paper.


\section*{Acknowledgement}
We thank two anonymous referees for their careful reading, and their suggestions that improved the presentation of the paper. 
Cyriac Antony reports financial support from IIITDM Kancheepuram, Chennai; IIT Madras, Chennai; and Luiss University, Rome. 

%
 \bibliographystyle{splncs04}
 \bibliography{myRefs13}

\appendix
\section{Appendix: Omitted Proofs}

\subsection{Proofs of Lemma~\ref{lem:C3 in CEO}, Lemma~\ref{lem:C4 in CEO}, Corollary~\ref{cor:opposite sides in C4 under CEO} and Theorem~\ref{thm:CL2r+1 obstruction to CEO}}\label{sec:A.2}
\begin{customlem}{\ref{lem:C3 in CEO}}
Let \( G \) be a graph, and let \( (\vec{G},f) \) be a \( q \)-coloured MINI-orientation of \( G \) for some \( q\in \mathbb{N} \). 
Then, \( \vec{G} \) orients each triangle in \( G \) as a directed cycle. 
\end{customlem}
\begin{proof}
Let \( (u,v,w) \) be an arbitrary triangle in \( G \). 
Either \( (u,v) \) or \( (v,u) \) is an arc in \( \vec{G} \). 
We show that if \( (u,v) \) is an arc in \( \vec{G} \), then \( (u,v,w) \) is oriented by \( \vec{G} \) as \( u\to v\to w\to u \). 
Suppose that \( (u,v) \) is an arc in \( \vec{G} \). 
\setcounter{myclaim}{0}
\begin{myclaim}\label{clm:triangle in CEO}
\( (v,w) \) is an arc in \( \vec{G} \). 
\end{myclaim}
\noindent If not, then \( u \) and \( w \) are in-neighbours of \( v \) and thus they must get the same colour under~\( f \), which is a contradiction since \( uw \) is an edge of \( G \) and \( f \) is a colouring of \( G \). 
This proves Claim~\ref{clm:triangle in CEO}.

Similarly, if \( (u,w) \) is an arc in \( \vec{G} \), then the in-neighbours \( u \) and \( v \) of \( w \) must get the same colour under \( f \), a contradiction. 
Hence, \( (w,u) \) is an arc in \( \vec{G} \). 
Hence, \( (u,v,w) \) is oriented by \( \vec{G} \) as \( u\to v\to w\to u \). 
Therefore, if \( (u,v) \) is an arc in \( \vec{G} \), then \( (u,v,w) \) is oriented by \( \vec{G} \) as \( u\to v\to w\to u \). 
Similarly, we can show that if \( (v,u) \) is an arc in \( \vec{G} \), then \( (u,v,w) \) is oriented by \( \vec{G} \) as \( v\to u\to w\to v \). 
\end{proof}

\begin{customlem}{\ref{lem:C4 in CEO}}
Let \( G \) be a graph, and let \( (\vec{G},f) \) be a \( q \)-coloured MINI-orientation of \( G \) for some \( q\in \mathbb{N} \). 
Then, for every 4-vertex cycle in \( G \), not necessarily induced, edges in the cycle are oriented by \( \vec{G} \) either as a directed cycle (as in Figure~\ref{fig:C4 orientation 1}) or as a direction-alternating cycle (as in Figure~\ref{fig:C4 orientation 2}). 
\end{customlem}
\begin{figure}[hbt]
\centering
\begin{subfigure}[b]{0.3\textwidth}
\centering
\begin{tikzpicture}[decoration={
    markings,
    mark=at position 0.5 with {\arrow{>}}},
    ]
\path (0,0) node(u)[dot]{}--++(1,-1) node(v)[dot]{}--++(-1,-1) node(w)[dot]{}--++(-1,1) node(x)[dot]{}--(u); 
\draw [postaction=decorate,thick]
(u)--(v);
\draw [postaction=decorate,thick]
(v)--(w);
\draw [postaction=decorate,thick]
(w)--(x);
\draw [postaction=decorate,thick]
(x)--(u);
\end{tikzpicture}
\caption{}
\label{fig:C4 orientation 1}
\end{subfigure}\begin{subfigure}[b]{0.3\textwidth}
\centering
\begin{tikzpicture}[decoration={
    markings,
    mark=at position 0.5 with {\arrow{>}}},
    ]
\path (0,0) node(u)[dot]{}--++(1,-1) node(v)[dot]{}--++(-1,-1) node(w)[dot]{}--++(-1,1) node(x)[dot]{}--(u); 
\draw [postaction=decorate,thick]
(u)--(v);
\draw [postaction=decorate,thick]
(w)--(v);
\draw [postaction=decorate,thick]
(w)--(x);
\draw [postaction=decorate,thick]
(u)--(x);
\end{tikzpicture}
\caption{}
\label{fig:C4 orientation 2}
\end{subfigure}\caption{Possible orientations of a \( C_4 \) in a MINI-orientation \( \protect\vec{G} \).}\label{fig:C4 orientations}
\end{figure}
\begin{proof}
Let \( (u,v,w,x) \) be an arbitrary \( C_4 \) in \( G \). 
Without loss of generality, assume that \( (u,v) \) is an arc in \( \vec{G} \). 
Let us consider the possible orientations of the edge \( vw \) as two cases.\\

\noindent \textit{Case~1:} \( (v,w) \) is an arc in \( \vec{G} \).\\[5pt]
We have, \( u \) is an in-neighbour of \( v \) and \( w \) is an out-neighbour of \( v \). 
Since \( (\vec{G},f) \) is a coloured MINI-orientation, we have \( f(u)\neq f(w) \).\\[5pt]
\noindent \textbf{Claim~1 (of Case~1):} \( (w,x) \) is an arc in \( \vec{G} \).\\[5pt]
On the contrary, assume that \( (x,w) \) is an arc in \( \vec{G} \). 
Since \( (\vec{G},f) \) is a coloured MINI-orientation, the in-neighbours \( v \) and  \( x \) of \( w \) should get the same colour under \( f \). 
That is, \( f(v)=f(x) \). 
Since \( v \) and \( x \) are neighbours of \( u \), the only possibility of \( v \) and \( x \) getting the same colour is that \( v \) and \( x \) are in-neighbours of \( u \). 
This is a contradiction since \( v \) is an out-neighbour of \( u \). 
This proves Claim~1.

By Claim~1, \( (w,x) \) is an arc in \( \vec{G} \). 
Since \( f(u)\neq f(w) \), it follows that \( (u,x) \) is not an arc in \( \vec{G} \) (if not, the in-neighbours \( u \) and \( w \) of \( x \) should get the same colour under \( f \); a contradiction). 
Therefore, the orientation on \( (u,v,w,x) \) in \( \vec{G} \) is \( u\to v\to w\to x\to u \); that is, as in Figure~\ref{fig:C4 orientation 1}.\\

\noindent \textit{Case 2:} \( (w,v) \) is an arc in \( \vec{G} \).\\[5pt]
Then, the in-neighbours \( u \) and \( w \) of \( v \) should get the same colour under \( f \). 
That is, \( f(u)=f(w) \). 
Since \( u \) and \( w \) are neighbours of \( x \), the only possibility of \( u \) and \( w \) getting the same colour is that \( u \) and \( w \) are in-neighbours of \( x \). 
Therefore, the orientation on \( (u,v,w,x) \) in \( \vec{G} \) is \( u\to v, w\to v, u\to x \) and \( w\to x \); that is, as in Figure~\ref{fig:C4 orientation 2}.

In both Case~1 and Case~2, the orientation on \( (u,v,w,x) \) is as shown in Figure~\ref{fig:C4 orientations}. 
This completes the proof since the 4-vertex cycle \( (u,v,w,x) \) is arbitrary. 
\end{proof}
\begin{customcor}{\ref{cor:opposite sides in C4 under CEO}}
Let \( G \) be a graph with a 4-vertex cycle \( (u,v,w,x) \), and let \( (\vec{G},f) \) be a \( q \)-coloured MINI-orientation of \( G \). 
If \( (u,v)\in E(\vec{G}) \), then \( (w,x)\in E(\vec{G}) \). 
\end{customcor}
\begin{proof}
Suppose that \( (u,v)\in E(\vec{G}) \); that is, \( \vec{G} \) orients edge \( uv \) as \( u\to v \). 
Then, \( \vec{G} \) orients \( (u,v,w,x) \) as either (i)~\( u\to v\to w\to x\to u \), or (ii)~\( u\to v, w\to v, u\to x, w\to x \). 
In both cases, \( \vec{G} \) orients the edge \( wx \) as \( w\to x \); that is, \( (w,x)\in E(\vec{G}) \). 
\end{proof}

\begin{figure}[hbt]
\centering
\begin{subfigure}[b]{0.3\textwidth}
\centering
\begin{tikzpicture}
\path (0,0) node(u)[dot][label=left:\( a \)]{}++(1,1) node(v)[dot][label=\( x \)]{}++(1,-1) node(w)[dot][label=right:\( b \)]{}++(-1,-1) node(x)[dot][label=below:\( y \)]{};
\draw (u)--(v)--(w)--(x)--(u);
\draw (v)--(x);
\end{tikzpicture}
\caption{}
\end{subfigure}\begin{subfigure}[b]{0.3\textwidth}
\centering
\begin{tikzpicture}
\draw (0,0) circle[radius=0.75];
\draw (0,0) circle[radius=1.25];
\draw (15:0.75) node[dot][label={[label distance=-2pt]left:\( u_1 \)}]{}--(15:1.25) node[dot][label=right:\( v_1 \)]{};
\draw (90:0.75) node[dot][label=below:\( u_0 \)]{}--(90:1.25) node[dot][label=\( v_0 \)]{};
\draw (175:0.75) node[dot][label={[yshift=-3pt]right:\( u_{2r} \)}]{}--(175:1.25) node[dot][label=left:\( v_{2r} \)]{};
\path (-90:0.75) node[fill=white]{\dots};
\path (-90:1.25) node(outerCirc)[fill=white]{\dots};
\end{tikzpicture}
\caption{}
\end{subfigure}\begin{subfigure}[b]{0.3\textwidth}
\centering
\begin{tikzpicture}
\tikzset{
every label/.style={label distance=-2pt},
}
\draw (0,0) circle[radius=0.75];
\draw (0,0) circle[radius=1.25];
\draw (0:0.75) node[dot][label=left:\( u_1 \)]{}--(0:1.25) node[dot][label=right:\( v_1 \)]{};
\draw (72:0.75) node[dot][label=below:\( u_0 \)]{}--(72:1.25) node[dot][label=\( v_0 \)]{};
\draw (144:0.75) node[dot][label=below right:\( u_4 \)]{}--(144:1.25) node[dot][label=above left:\( v_4 \)]{};
\draw (-72:0.75) node[dot][label=\( u_2 \)]{}--(-72:1.25) node[dot][label=below:\( v_2 \)]{};
\draw (-144:0.75) node[dot][label=above right:\( u_3 \)]{}--(-144:1.25) node[dot][label=below left:\( v_3 \)]{};
\path (-90:1.25) coordinate(outerCirc){};
\end{tikzpicture}
\caption{}
\end{subfigure}\caption{(a) Diamond, (b) circular ladder graph \( CL_{2r+1} \), and (c) \( CL_5 \).}
\end{figure}

\begin{theorem}\label{thm:diamond obstruction to CEO}
The diamond graph does not admit a \( q \)-coloured MINI-orientation for any \( q\in \mathbb{N} \). 
\end{theorem}
\begin{proof}
Let \( G \) be the diamond graph with vertex set \( \{x,y,a,b\} \) and edge set \( \{ax,xb,by,\allowbreak ya,xy\} \). 
Contrary to the theorem, suppose that \( G \) admits a \( q \)-coloured MINI-orientation \( (\vec{G},f) \) for some \( q\in \mathbb{N} \). 
Without loss of generality, assume that \( \vec{G} \) orients the edge \( xy \) as \( x\to y \). 
By Lemma~\ref{lem:C3 in CEO}, \( \vec{G} \) orients the triangle \( (x,y,a) \) in \( G \) as \( x\to y\to a\to x \). 
Similarly, \( \vec{G} \) orients the triangle \( (x,y,b) \) in \( G \) as \( x\to y\to b\to x \). 
Since \( \vec{G} \) orients the 4-vertex cycle \( (a,x,b,y) \) as \( y\to a\to x \), \( y\to b\to x \), the 4-vertex cycle is not oriented by \( \vec{G} \) as in Figure~\ref{fig:C4 orientations}; a contradiction to Lemma~\ref{lem:C4 in CEO}. 
\end{proof}
\begin{customthm}{\ref{thm:CL2r+1 obstruction to CEO}}
For \( r\in \mathbb{N} \), the circular ladder graph \( CL_{2r+1} \) does not admit a \( q \)\nobreakdash-coloured MINI-orientation for any \( q\in \mathbb{N} \). 
Thus, for \( p\geq 2 \), a \( 2p \)-regular \( (p+2) \)-star colourable graph does not contain \( CL_{2r+1} \) as a subgraph. 
\end{customthm}
\begin{proof}
Let \( G \) be the circular ladder graph \( CL_{2r+1} \) with vertex set \( \{u_0,u_1,\dots,u_{2r} \), \( v_0,v_1,\dots,v_{2r}\} \) and edge set \( \{u_iu_{i+1}\colon i\in \mathbb{Z}_{2r+1}\}\cup \{v_iv_{i+1}\colon i\in \mathbb{Z}_{2r+1}\}\cup \{u_iv_i\colon i\in \mathbb{Z}_{2r+1}\} \), where subscript \( i+1 \) is modulo \( 2r+1 \). 
Contrary to the theorem, suppose that \( G \) admits a \( q \)-coloured MINI-orientation \( (\vec{G},f) \) for some \( q\in \mathbb{N} \). 
Due to Corollary~\ref{cor:opposite sides in C4 under CEO}, there exists \( i\in \mathbb{Z}_{2r+1} \) such that \( (u_i,v_i)\in E(\vec{G}) \) (if \( (v_0,u_0)\in E(\vec{G}) \), then applying Corollary~\ref{cor:opposite sides in C4 under CEO} on \( (v_0,u_0,u_1,v_1) \) yields \( (u_1,v_1)\in E(\vec{G}) \)). 
Without loss of generality, assume that \( (u_0,v_0)\in E(\vec{G}) \). 
Then, applying Corollary~\ref{cor:opposite sides in C4 under CEO} on \( (u_0,v_0,v_1,u_1) \) yields \( (v_1,u_1)\in E(\vec{G}) \). 
Hence, applying Corollary~\ref{cor:opposite sides in C4 under CEO} on \( (v_1,u_1,u_2,v_2) \) yields \( (u_2,v_2)\in E(\vec{G}) \), and thus applying Corollary~\ref{cor:opposite sides in C4 under CEO} on \( (u_2,v_2,v_3,u_3) \) yields \( (v_3,u_3)\in E(\vec{G}) \). 
By repeated application, we get \( (u_{2i},v_{2i})\in E(\vec{G}) \) for \( 0\leq i\leq r \), and \( (v_{2i+1},u_{2i+1})\in E(\vec{G}) \) for \( 0\leq i\leq r-1 \). 
In particular, \( (u_{2r},v_{2r})\in E(\vec{G}) \). 
Hence, applying Corollary~\ref{cor:opposite sides in C4 under CEO} on \( (u_{2r},v_{2r},v_0,u_0) \) yields \( (v_0,u_0)\in E(\vec{G}) \). 
This is a contradiction since we started with \( (u_0,v_0)\in E(\vec{G}) \). 
\end{proof}

\subsection{Proof of Theorem~\ref{thm:orientation mini test in P}}\label{sec:A.3}
\begin{customthm}{\ref{thm:orientation mini test in P}}
Given an orientation \( \vec{G} \) of a graph \( G \), we can test in polynomial time whether there exists an integer \( q^* \) and a \( q^* \)-colouring \( f^* \) of \( G \) such that \( (\vec{G},f^*) \) is a \( q^* \)-coloured MINI-orientation of \( G \). 
\end{customthm}
\begin{proof}
Let \( \vec{G} \) be an orientation of \( G \). 
We define an equivalence relation \( \mathcal{R} \) on \( V(G) \) as follows: \( (x,y)\in \mathcal{R} \) if there exists an \( x,y \)-path \( u_0,u_1,\dots,u_{2t} \) in \( \vec{G} \) with \( t\geq 0 \), \( x=u_0 \) and \( y=u_{2t} \) such that \( u_{2j},u_{2j+2} \) are in-neighbours of \( u_{2j+1} \) for \( 0\leq j<t \).

Let \( V_0,V_1,\dots,V_{q-1} \) be the equivalence classes under \( \mathcal{R} \), and let \( f:V(G)\to\{0,1,\dots,{q-1}\} \) be the function defined as \( f(v)=i \) for all \( v\in V_i \) \( (0\leq i\leq q-1) \). 
Clearly, we can compute \( f \) in polynomial time, and test in polynomial time whether \( f \) is a \( q \)-colouring of \( G \) and \( (\vec{G},f) \) is a \( q \)-coloured MINI-orientation. 
We claim that \( (\vec{G},f^*) \) is a \( q^* \)-coloured MINI-orientation of \( G \) for some integer \( q^* \) and some \( q^* \)-colouring \( f^* \) of \( G \) if and only if \( (\vec{G},f) \) is a \( q \)-coloured MINI-orientation of \( G \). 
To prove this claim, it suffices to show the only if direction. 
Suppose that \( (\vec{G},f^*) \) is a \( q^* \)-coloured MINI-orientation of \( G \).

To prove that \( f \) is \( q \)-colouring of \( G \), it suffices to establish the following claim.\\[5pt]
\emph{Claim 1:}  \( V_i \) is an independent set for \( 0\leq i\leq q-1 \).\\[5pt] 
We prove Claim~1 for \( i=0 \) (the proof is similar for other values of \( i \)). 
To produce a contradiction, assume that \( V_0 \) is not an independent set, say \( xy \) is an edge in \( G \) where \( x,y\in V_0 \). 
Since \( x \) and \( y \) belong to the same equivalence class under \( \mathcal{R} \) (namely \( V_0 \)), there exists an \( x,y \)-path \( u_0,u_1,\dots,u_{2t} \) in \( \vec{G} \) with \( t\geq 0 \), \( x=u_0 \), \( y=u_{2t} \), and \( u_{2j},u_{2j+2} \) are in-neighbours of \( u_{2j+1} \) for \( 0\leq j<t \). 
By definition of \( q^* \)-coloured MINI-orientation, \( f^*(u_{2j})=f^*(u_{2j+2}) \) for \( 0\leq j<t \) (because \( u_{2j} \) and \( u_{2j+2} \) are in-neighbours of \( u_{2j+1} \)). 
Hence \( f^*(u_0)=f^*(u_2)=\dots=f^*(u_{2t}) \); thus, \( f^*(x)=f^*(y) \). 
This is a contradiction since \( f^* \) is a colouring of \( G \) and \( xy\in E(G) \). 
This proves Claim~1. 
So, \( f \) is a \( q \)-colouring of \( G \).\\

It remains to show that \( (\vec{G},f) \) is a \( q \)-coloured MINI-orientation of \( G \). 
Let \( v \) be an arbitrary vertex of \( G \). 
Let \( w_1,\dots,w_p \) be the in-neighbours of \( v \), and let \( x_1,\dots,x_r \) be the out-neighbours of \( v \) in \( \vec{G} \). 
We need to show that all three conditions in the definition of \( q \)-coloured MINI-orientation are satisfied. 
That is, we need to show that the following hold in \( (\vec{G},f) \): (i)~all in-neighbours of \( v \) have the same colour, (ii)~no out-neighbour of \( v \) has the same colour as an in-neighbour of \( v \), and (iii)~no two out-neighbours of \( v \) have the same colour.

For \( 1\leq k<\ell\leq p \), \( (w_k,w_\ell)\in\mathcal{R} \) since \( w_k \) and \( w_\ell \) are in-neighbours of \( v \), and thus \( f(w_k)=f(w_\ell) \). 
Hence, \( f(w_1)=\dots=f(w_p) \). 
This proves (i).

Next, we prove (ii); that is, \( f(w_k)\neq f(x_\ell) \) for \( 1\leq k\leq p \) and \( 1\leq \ell\leq r \). 
To produce a contradiction, assume the contrary; that is, there exist \( k\in\{1,\dots,p\} \) and \( \ell\in\{1,\dots,r\} \) such that \( f(w_k)=f(x_\ell) \). 
Since \( f(w_k)=f(x_\ell) \), \( w_k \) and \( x_\ell \) belong to the same equivalence class under \( \mathcal{R} \). 
That is, there exists an \( w_k,x_\ell \)-path \( u_0,u_1,\dots,u_{2t} \) in \( \vec{G} \) with \( t\geq 0 \), \( w_k=u_0 \), \( x_\ell=u_{2t} \), and \( u_{2j},u_{2j+2} \) are in-neighbours of \( u_{2j+1} \) for \( 0\leq j<t \). 
For \( 0\leq j<t \), \( u_{2j} \) and \( u_{2j+2} \) are in-neighbours of \( u_{2j+1} \) and thus \( f^*(u_{2j})=f^*(u_{2j+2}) \). 
Therefore, \( f^*(u_0)=f^*(u_2)=\dots=f^*(u_{2t}) \); thus, \( f^*(w_k)=f^*(x_\ell) \). 
But, since \( w_k \) is an in-neighbour of \( v \) and \( x_\ell \) is an out-neighbour of \( v \), \( f^*(w_k)\neq f^*(x_\ell) \). 
This contradiction proves (ii).

Finally, we prove (iii); that is, \( f(x_k)\neq f(x_\ell) \) for \( 1\leq k<\ell\leq r \). 
To produce a contradiction, assume the contrary; that is, there exist \( k,\ell \in\{1,\dots,r\} \) with \( k<\ell \) such that \( f(x_k)=f(x_\ell) \). 
Since \( f(x_k)=f(x_\ell) \), \( x_k \) and \( x_\ell \) belong to the same equivalence class under \( \mathcal{R} \). 
That is, there exists an \( x_k,x_\ell \)-path \( u_0,u_1,\dots,u_{2t} \) in \( \vec{G} \) with \( t\geq 0 \), \( x_k=u_0 \), \( x_\ell=u_{2t} \), and \( u_{2j},u_{2j+2} \) are in-neighbours of \( u_{2j+1} \) for \( 0\leq j<t \). 
For \( 0\leq j<t \), \( u_{2j} \) and \( u_{2j+2} \) are in-neighbours of \( u_{2j+1} \) and thus \( f^*(u_{2j})=f^*(u_{2j+2}) \). 
Therefore, \( f^*(u_0)=f^*(u_2)=\dots=f^*(u_{2t}) \); thus, \( f^*(x_k)=f^*(x_\ell) \). 
But, since \( x_k \) and \( x_\ell \) are out-neighbours of \( v \), \( f^*(x_k)\neq f^*(x_\ell) \). 
This contradiction proves (iii). 

Since (i), (ii) and (iii) hold for an arbitrary vertex \( v \) of \( \vec{G} \),  it follows that \( (\vec{G},f) \) is a \( q \)-coloured MINI-orientation. 
\end{proof}

\end{document}